\documentclass{article}

\usepackage[utf8]{inputenc}
\usepackage[T1]{fontenc}
\usepackage{amsmath}
\usepackage{amsfonts}

\usepackage{authblk}

\usepackage{mathtools, mathrsfs}
\usepackage{amsthm}
\usepackage{array}
\usepackage{framed}
\usepackage{subcaption}
\usepackage{color}
\usepackage{enumitem}
\usepackage{dsfont}
\usepackage{listings}
\usepackage{appendix}
\usepackage{amssymb}
\usepackage{esint}

\usepackage{url}

\usepackage{stmaryrd} 
\usepackage{graphics,graphicx}
\usepackage{wrapfig}
\usepackage{caption}

\usepackage{vmargin}

\newtheoremstyle{newremark}
  {5pt}
  {5pt}
  {\rmfamily}
  {}
  {\rmfamily\bf}
  {.}
  {.5em}
  {}

\newtheorem{theorem}{Theorem}
\newtheorem{lemma}[theorem]{Lemma}
\newtheorem{corollary}[theorem]{Corollary}
\newtheorem{proposition}[theorem]{Proposition}

\theoremstyle{newremark}
\newtheorem{remark}[theorem]{Remarque}

\newtheorem*{definition*}{Définition} 
\newtheorem*{notations*}{Notations}

\numberwithin{theorem}{section}
\numberwithin{equation}{section}

\newcommand{\N}{\mathbb{N}}
\newcommand{\R}{\mathbb{R}}

\renewcommand{\bar}[1]{\overline{#1}}

\newcommand{\de}{{\rm d}}
\newcommand{\LL}{\mathop{\hbox{\vrule height 6pt width .5pt depth 0pt
\vrule height .5pt width 3pt depth 0pt}}\nolimits}

\title{\bf  Gamma-convergence of nonlocal energies\\ for partitions}

\author{Thomas Gabard \& Vincent Millot \thanks{thomas.gabard@u-pec.fr ; vincent.millot@u-pec.fr}}
\affil{ \small LAMA, Universit\'e Paris Est Cr\'eteil, Universit\'e Gustave Eiffel, UPEM, CNRS,  Cr\'eteil, France}

\date{ }

\begin{document}

\maketitle


\section{Introduction}

For a measurable set $E\subset \R^n$, $n\geq 2$, an open set $\Omega\subset\R^n$, and a parameter $s\in(0,1/2)$, the fractional $2s$-perimeter $P_{2s}(E,\Omega)$ of $E$ in $\Omega$ is defined to be the integral quantity
\begin{multline*}
P_{2s}(E,\Omega):=\iint_{(E\cap\Omega)\times(E^c\cap\Omega)}\frac{\de x\de y}{|x-y|^{n+2s}} +\iint_{(E\cap\Omega)\times(E^c\cap\Omega^c)}\frac{\de x\de y}{|x-y|^{n+2s}}\\
+\iint_{(E\cap\Omega^c)\times(E^c\cap\Omega)}\frac{\de x\de y}{|x-y|^{n+2s}} \,,
\end{multline*}
where $A^c$ denotes the complement of a set $A\subset\R^n$. 

The $2s$-perimeter functional were introduced in the famous by now article \cite{CRS}. Their main motivation probably comes from \cite{CafSou} where the Bence-Merriman-Osher (BMO) scheme \cite{BMO} is implemented with the fractional heat equation involving  the fractional Laplacian of exponent $s$ (instead of the usual Laplacian). The standard BMO scheme is a well known algorithm to approximate 
motion of boundaries by mean curvature, as proved in \cite{BaGe,Eva}. For the fractional heat equation, it is showed in \cite{CafSou} that the BMO scheme converges to motion by nonlocal mean curvature, a motion which corresponds (at least formally) to a gradient flow of the functional~$P_{2s}$.  

The main question addressed in \cite{CRS} is the regularity of the boundary of a local minimizer $E$ of $P_{2s}(\cdot,\Omega)$. The notion of local minimizer means that 
$P_{2s}(E,\Omega)\leq P_{2s}(F,\Omega)$ for any set $F\subset\R^n$ such that $E\triangle F\Subset\Omega$ (here $\triangle$ denotes the symmetric difference, and $\Subset$ means compact inclusion). It is proved in  \cite{CRS} that $\partial E\cap \Omega$ is locally of class $C^{1,\alpha}$ away from some relatively closed singular subset of Hausdorff dimension at most $n-2$. The partial $C^{1,\alpha}$-regularity has then been upgraded to $C^\infty$ in  \cite{BFV}. Concerning improvements on the size of the singular set, the only general result is in \cite{SV} where it is proved that its Hausdorff dimension is at most $n-3$ (whence full regularity in dimension 2). Compare to the regularity theory for local minimizers of the standard perimeter functional (that we recall below), there is an important gap in dimensions. More precisely, it is well known for the standard perimeter that the singular set of any local minimizer has Hausdorff dimension at most $n-8$ (which is optimal). Observing that $P_{2s}(E,\Omega)$ is essentially the fractional $W^{2s,1}$-seminorm $[\cdot]_{W^{2s,1}(\Omega)}$ of the characteristic function~$\chi_E$ in $\Omega$, we have 
$(1-2s)[\chi_E]_{W^{2s,1}(\Omega)}\sim |D\chi_E|(\Omega)={\rm Per}(E,\Omega)$ as $s\to 1/2^-$ by \cite{Dav} (see also \cite{BouBreMir}), where ${\rm Per}(E,\Omega)$ denotes the (distributional) perimeter of $E$ in $\Omega$, i.e., 
$${\rm Per}(E,\Omega):=\sup\Big\{\int_E{\rm div}\,\varphi\,\de x : \varphi\in C^1_c(\Omega;\R^n)\,,\;\|\varphi\|_{L^\infty(\Omega)}\leq 1\Big\}\,. $$
A set $E$ satisfying ${\rm Per}(E,\Omega)<\infty$ is said to be of finite perimeter in $\Omega$, and in this case,
$${\rm Per}(E,\Omega)=\mathcal{H}^{n-1}(\partial^*E\cap\Omega) \,,$$
where  $\mathcal{H}^{n-1}$ is the $(n-1)$-dimensional Hausdorff measure, and $\partial^*E$ is the reduced boundary of $E$ (a measure theoretic boundary which coincides with the topological boundary if $\partial E$ is smooth enough, see \cite{AFP,Maggi}). A variational convergence of $(1-2s)P_{2s}(\cdot,\Omega)$ toward ${\rm Per}(\cdot,\Omega)$ as $s\to1/2^-$ was first obtained in \cite{CafVal}, and it was used to show that the singular set of any local minimizer of $P_{2s}(\cdot,\Omega)$ has actually an Hausdorff dimension at most $n-8$ for $s$ close enough to $1/2$. A more detailed analysis 
by $\Gamma$-convergence has been obtained in \cite{ADPM}. It is shown that the family of functionals $\{(1-2s)P_{2s}(\cdot,\Omega)\}_{0<s<1/2}$ $\Gamma$-converges for the $L^1(\Omega)$-convergence of characteristic functions to $\omega_{n-1}{\rm Per}(\cdot,\Omega)$, where $\omega_{n-1}$ denotes as usual the volume of the unit ball in $\R^{n-1}$. 
 
By analogy with minimal surfaces obtained as critical points of the perimeter, critical points of $P_{2s}$ are referred to as {\sl nonlocal minimal surfaces} (or to be more accurate, their boundary). Very recently, 
nonlocal minimal surfaces have been used to produce minimal hypersurfaces in Riemannian manifolds through their asymptotic analysis as $s\to1/2^-$, see \cite{CFS,CDSV}. In particular, an analogue of Yau's conjecture for nonlocal minimal surfaces have been proved in \cite{CFS}, whence eventually recovering the standard conjecture letting $s\to 1/2^-$. Let us finally observe that the very simple structure of $P_{2s}$ makes it 
 reasonably flexible to be generalized to other settings, eventually non smooth, and hopefully analyzed as $s\to 1/2^-$, see e.g. \cite{LPZ} and the references therein.  
\vskip5pt

In this article, we are concerned with partitions rather than single sets. Here the notion of partition is understood in the measure theoretic sense more than in the usual one. To be more precise, 
we fix a number $m\geq 3$ of chambers, and we say that a collection $\mathfrak{E}=(E_1,\ldots,E_m)$ of measurable subsets $E_i$ of $\R^n$ (resp. of $\Omega$) is partition of $\R^n$ (resp. of $\Omega$) if $|E_i\cap E_j|=0$ for $i\not=j$ and $|\R^n\setminus \bigcup_i E_i|=0$ (resp. $|\Omega\setminus\bigcup_i E_i|=0$). If each member $E_i$ of the partition $\mathfrak{E}$ turns out to be of finite perimeter in $\Omega$, we say that $\mathfrak{E}$ is a {\it Caccioppoli partition} of $\Omega$ (see e.g. \cite{AFP}). A natural generalization of the perimeter in $\Omega$ for a Caccioppoli partition which takes 
into account the difference between two distinct chambers is 
$$\mathscr{P}_1^{\boldsymbol{\sigma}}(\mathfrak{E},\Omega):=\frac{1}{2}\sum_{i,j=1}^m\sigma_{ij}\,\mathcal{H}^{n-1}(\partial^*E_i\cap\partial^*E_j\cap\Omega)\,,$$
where $\boldsymbol{\sigma}=(\sigma_{ij})$ is a symmetric $m\times m$ matrix such that $\sigma_{ii}=0$ and $\sigma_{ij}>0$ for $i\not=j$. 

The functional $\mathscr{P}_1^{\boldsymbol{\sigma}}$ commonly appears in models for mixtures of immiscible fluids filling the container~$\Omega$. 
More importantly,  $\mathscr{P}_1^{\boldsymbol{\sigma}}$ and dynamic flows associated to it arise in material science for the description of motion of grain boundaries in polycrystals, see \cite{EseOtt}.  
Each chamber $E_i$ represents an individual grain, and the coefficient $\sigma_{ij}$, called surface tension, attached to the interface $\partial E_i\cap\partial E_j$, depends 
on the mismatch between the crystallographic orientations of $E_i$ and $E_j$. A more accurate model should also depends on the local orientation of the interface, but we choose to ignore it here. 
From the mathematical perspective, existence and (partial) regularity theory for (local) minimizers of $\mathscr{P}_1^{\boldsymbol{\sigma}}$ (and actually more general functionals) has probably 
emerged in the famous memoir \cite{Alm}. Existence of minimizers is hinged on the lower semicontinuity of $\mathscr{P}_1^{\boldsymbol{\sigma}}$ with respect to the $L^1$-convergence of characteristic functions. It has been proved in \cite{AmBr,WB} that lower semicontinuity holds if and only if the matrix $\boldsymbol{\sigma}$ satisfies the triangle inequality 
$\sigma_{ij}\leq \sigma_{ik}+\sigma_{kj}$ for all triplet $\{i,j,k\}$. In turn, a partial regularity for the boundaries of minimizers has been obtained in \cite{Leo,WB} under the condition of {\sl strict} triangle inequality. In comparison and very interestingly, a (BMO) inspired scheme for curvature driven motion of networks dissipating $\mathscr{P}_1^{\boldsymbol{\sigma}}$ has been constructed \cite{EseOtt} which allows for (some) matrices $\boldsymbol{\sigma}$ which {\sl do not} satisfy the triangle inequality.  By its nature, this scheme replaces sharp interfacial energies by nonlocal (diffuse) interfacial energies involving a convolution with the heat kernel. If $\boldsymbol{\sigma}$ does not satisfy the triangle inequality, numerical evidence shows that a wetting effect occurs, i.e., the algorithm 
can instantaneously nucleate a new phase between two chambers to release energy. It is then expected that the  scheme converges to the (dissipative) flow of the relaxed functional, i.e., its lower semicontinuous enveloppe. 
\vskip5pt

Motivated by all these considerations, we are interested here in a nonlocal analogue of $\mathscr{P}_1^{\boldsymbol{\sigma}}$ modelled on the fractional $2s$-perimeter $P_{2s}$, replacing sharp interfacial energies by diffuse ones involving the convolution with the singular kernel $|\cdot|^{-(n+2s)}$. To this purpose, it is convenient  to introduce the $2s$-interaction energy 
$\mathcal{I}_{2s}(A,B)$ between two measurable sets $A,B\subset\R^n$ given by 
$$\mathcal{I}_{2s}(A,B):=\iint_{A\times B}\frac{\de x\de y}{|x-y|^{n+2s}}\,.$$
Given a {\sl fixed} number of chambers $m\geq 3$ and an arbitrary  $m\times m$ symmetric matrix $\boldsymbol{\sigma}=(\sigma_{ij})$ satisfying $\sigma_{ii}=0$ and $\sigma_{ij}>0$ for $i\not=j$, we define the fractional $\boldsymbol{\sigma}$-perimeter in $\Omega$ of a partition $\mathfrak{E}=(E_1,\ldots,E_m)$ of $\R^n$ to be 
$$\mathscr{P}_{2s}^{\boldsymbol{\sigma}}(\mathfrak{E},\Omega):=\frac{1}{2}\sum_{i,j=1}^m\sigma_{ij}\Big[\mathcal{I}_{2s}(E_i\cap\Omega,E_j\cap\Omega)+\mathcal{I}_{2s}(E_i\cap\Omega,E_j\cap\Omega^c)+\mathcal{I}_{2s}(E_i\cap\Omega^c,E_j\cap\Omega)\Big]\,. $$
If a partition consists of only one chamber and its complement, it is easy to see that its $\boldsymbol{\sigma}$-perimeter reduces to the ${2s}$-perimeter of the chamber times the corresponding coefficient of the matrix. Therefore, the functional $\mathscr{P}_{2s}^{\boldsymbol{\sigma}}$ is indeed a natural extension of $P_{2s}$ to partitions. It is also important to notice that $\mathscr{P}_{2s}^{\boldsymbol{\sigma}}(\cdot,\Omega)$ is lower semicontinuous with respect to $L^1$-convergence for any such matrix  $\boldsymbol{\sigma}$, as a direct consequence of Fatou's lemma. This is a 
major difference with $\mathscr{P}_1^{\boldsymbol{\sigma}}$ which allows to apply the Direct Method of Calculus of Variations to solve minimization problems involving $\mathscr{P}_{2s}^{\boldsymbol{\sigma}}(\cdot,\Omega)$ with no restrictions on the matrix (such as the triangle inequality). 
\vskip5pt

In the case of identical coefficients, that is $\sigma_{ij}=1$ for $i\not=j$ (that we denote by $\boldsymbol{\sigma}={\bf 1}$ with a slight abuse of notation), we simply have $\mathscr{P}_{2s}^{\bf 1}(\mathfrak{E},\Omega)=\sum_iP_{2s}(E_i,\Omega)$. This functional has been considered in \cite{ColMag} where the existence and partial regularity problem of isoperimetric clusters for 
$\mathscr{P}_{2s}^{\bf 1}$  is addressed at the image of  classical perimeter minimizing clusters, see \cite{Alm} and \cite[Chapter 4]{Maggi}. The analysis in  \cite{ColMag} has then been extended in 
\cite{CesNov} to the case $\sum_i\alpha_i P_{2s}(E_i,\Omega)$ for some coefficients $\alpha_i>0$. This last functional also fits to our setting as it corresponds to the case where the matrix 
 $\boldsymbol{\sigma}$ is given by $\sigma_{ij}=\alpha_i+\alpha_j$. Using the $\Gamma$-convergence  in \cite{ADPM} of $(1-2s)P_{2s}$ toward the standard perimeter, 
  a classification of all (locally) minimizing conical partitions for $m=3$ and $s$ close enough to $1/2$ has also been obtained in \cite{CesNov} in the plane $\R^2$. They consist of either two chambers separated by a straight line, or three chambers separated by three half lines meeting at a single point with angles uniquely determined by the $\alpha_i's$. For $\boldsymbol{\sigma}={\bf 1}$, this  condition on the angles reduces to the usual $120^\circ$ rule. 
 \vskip5pt
 
 In this article, our main goal is to perform the $\Gamma$-convergence analysis as $s\to 1/2^-$ of the family 
 $\{(1-2s)\mathscr{P}_{2s}^{\boldsymbol{\sigma}}(\cdot,\Omega)\}_{0<s<1/2}$  in the spirit of \cite{ADPM} for an arbitrary $m\times m$ matrix $\boldsymbol{\sigma}$ in the class 
 $$\mathscr{S}_m:=\big\{\boldsymbol{\sigma}=(\sigma_{ij})\in\mathscr{M}_{m\times m}(\R): \sigma_{ii}=0\,,\;\sigma_{ij}=\sigma_{ji}>0 \text{ for }i\not=j \}\,. $$
 At first sight, one might think from \cite{ADPM} that the limiting functional should simply be given by (a multiple of) $\mathscr{P}_1^{\boldsymbol{\sigma}}(\cdot,\Omega)$. However, a $\Gamma$-limit being lower semicontinuous by nature, and according to \cite{AmBr}, it cannot be given by  $\mathscr{P}_1^{\boldsymbol{\sigma}}(\cdot,\Omega)$ if $\boldsymbol{\sigma}$ does not satisfies the triangle inequality. Hence an additional relaxation effect must occur as $s\to1/2^-$ (very much as observed in \cite{EseOtt} for the (BMO) scheme). A more reasonable candidate is 
 given by the lower semicontinuous enveloppe  of $\mathscr{P}_1^{\boldsymbol{\sigma}}(\cdot,\Omega)$. Still by  \cite{AmBr}, this functional is given by 
 $$\mathscr{P}_1^{\bar{\boldsymbol{\sigma}}}(\mathfrak{E},\Omega)=\frac{1}{2}\sum_{i,j=1}^m\bar\sigma_{ij}\,\mathcal{H}^{n-1}(\partial^*E_i\cap\partial^*E_j\cap\Omega)\,, $$
 where the {\sl relaxed coefficients} matrix $\bar{\boldsymbol{\sigma}}\in\mathscr{S}_m$ is determined by
 $$\bar{\boldsymbol{\sigma}}:=\sup\Big\{\widetilde{\boldsymbol{\sigma}} : \widetilde{\boldsymbol{\sigma}}=(\widetilde\sigma_{ij})\in\mathscr{S}_m\,,\; \widetilde{\boldsymbol{\sigma}}\leq \boldsymbol{\sigma}\,,\;\widetilde\sigma_{ij}\leq \widetilde\sigma_{ik}+\widetilde\sigma_{kj}\;\forall\{i,j,k\}\subset\{1,\ldots,m\}\Big\}\,. $$
 Here the inequality between matrices is understood component by component. Since the family of matrices satisfying the triangle inequality forms a lattice, this definition of $\bar{\boldsymbol{\sigma}}$ makes sense (see also Lemma \ref{lemmrelaxcoeff} for a further characterization of $\bar{\boldsymbol{\sigma}}$).  
\vskip5pt

Before stating our results, let us introduce some useful notations. First, we shall be concerned with a bounded open set $\Omega\subset\R^n$ with Lipschitz regular boundary. 
The family of partitions $\mathfrak{E}$ of $\R^n$ of finite $\mathscr{P}_{2s}^{\boldsymbol{\sigma}}(\cdot,\Omega)$-energy is denoted by 
$$\mathscr{A}_m^s(\Omega):=\big\{\mathfrak{E}: \text{$\mathfrak{E}=(E_1,\ldots,E_m)$ is a partition of $\R^n$}\,,\;\mathscr{P}_{2s}^{\boldsymbol{\sigma}}(\mathfrak{E},\Omega)<\infty\big\} \,.$$
As in $\cite{ADPM}$, it will be convenient to split $\mathscr{P}_{2s}^{\boldsymbol{\sigma}}(\mathfrak{E},\Omega)$ in two pieces according to interactions within~$\Omega$ and interactions with the exterior of $\Omega$. In other words, we write for $\mathfrak{E}=(E_1,\ldots,E_m)\in \mathscr{A}_m^s(\Omega)$, 
$$\mathscr{P}_{2s}^{\boldsymbol{\sigma}}(\mathfrak{E},\Omega)= \mathscr{E}_{2s}^{\boldsymbol{\sigma}}(\mathfrak{E},\Omega)+\mathscr{F}_{2s}^{\boldsymbol{\sigma}}(\mathfrak{E},\Omega)\,,$$
with 
$$ \mathscr{E}_{2s}^{\boldsymbol{\sigma}}(\mathfrak{E},\Omega):= \frac{1}{2}\sum_{i,j=1}^m\sigma_{ij}\,\mathcal{I}_{2s}(E_i\cap\Omega,E_j\cap \Omega)\,,$$
and 
$$ \mathscr{F}_{2s}^{\boldsymbol{\sigma}}(\mathfrak{E},\Omega):= \frac{1}{2}\sum_{i,j=1}^m\sigma_{ij}\Big[\mathcal{I}_{2s}(E_i\cap\Omega,E_j\cap\Omega^c)+\mathcal{I}_{2s}(E_i\cap\Omega^c,E_j\cap\Omega)\Big]\,. $$
The $\mathscr{E}_{2s}^{\boldsymbol{\sigma}}$-part turns out to be the leading term of the energy, while $ \mathscr{F}_{2s}^{\boldsymbol{\sigma}}$ accounts for lower order contributions near $\partial\Omega$. 
\vskip5pt

We start with a compactness property for sequences of (locally) bounded energy essentially taken from \cite[Theorem 1]{ADPM}. 

\begin{theorem}\label{equicoercivethm}
 Consider a sequence $s_k\to 1/2^-$, and  assume that $\{\mathfrak{E}_k\}_{k\in \N}\subset \mathscr{A}_m^{s_k}(\Omega)$ satisfies
$$ \sup \limits_{k \in \N} \,(1-2s_k) \mathscr{E}_{2s_k}^{\boldsymbol{\sigma}}  (\mathfrak{E}_k,\Omega') < \infty \quad\forall \Omega^\prime\Subset\Omega\,.$$
Then $\{\mathfrak{E}_k\}_{k\in\mathbb{N}}$ is relatively compact in $[L^1_{\rm loc}(\Omega)]^m$, and any accumulation point $\mathfrak{E}$
is a Caccioppoli partition of $\Omega^\prime$
for every open subset $\Omega^\prime\Subset\Omega$. 
\end{theorem}

We may now state our $\Gamma$-convergence result showing the limit is indeed $\mathscr{P}_1^{\bar{\boldsymbol{\sigma}}}(\cdot,\Omega)$, 
the functional with relaxed coefficients $\bar{\boldsymbol{\sigma}}$, as expected.

\begin{theorem}\label{GCthm}
Let $\Omega\subset\R^n$ be a bounded open set with Lipschitz boundary. 
\begin{itemize}
\item[1)] The family $\{(1-2s)\mathscr{P}^{\boldsymbol{\sigma}}_{2s}(\cdot,\Omega)\}_{0<s<1/2}$ $\Gamma(L^1(\Omega))$-converges as $s\to1/2^-$ to $\omega_{n-1}\mathscr{P}_1^{\bar{\boldsymbol{\sigma}}}(\cdot,\Omega)$. 

\item[2)]  For every sequence $s_k\to1/2^-$ and  $\mathfrak{E}_{k}\in\mathscr{A}_m^{s_k}(\Omega)$ such that  $\mathfrak{E}_{k}\to \mathfrak{E}$ in $[L^1(\Omega)]^m$,  
$$\liminf_{k\to\infty}\,(1-2s_k) \mathscr{E}^{\boldsymbol{\sigma}}_{2s_k}(\mathfrak{E}_{k},\Omega)\geq \omega_{n-1}\mathscr{P}_1^{\bar{\boldsymbol{\sigma}}}( \mathfrak{E},\Omega)\,.$$
\end{itemize}
\end{theorem}

For the sake of completeness, we recall that 1) means that for every sequence $s_k\to1/2^-$, we have 
\begin{itemize}
\item[(i)] for every $\{\mathfrak{E}_{k}\}_{k\in\mathbb{N}}=\{(E_{1,k},\ldots,E_{m,k})\}_{k\in\mathbb{N}}\subset\mathscr{A}_m^{s_k}(\Omega)$ and partition $\mathfrak{E}=(E_1,\ldots,E_m)$ of $\Omega$ such that $\chi_{E_{i,k}}\to \chi_{E_{i}}$ in $L^1(\Omega)$ for every $i\in\{1,\ldots,m\}$, 
$$\liminf_{k\to\infty}\,(1-2s_k) \mathscr{P}^{\boldsymbol{\sigma}}_{2s_k}(\mathfrak{E}_{k},\Omega)\geq \omega_{n-1}\mathscr{P}_1^{\bar{\boldsymbol{\sigma}}}( \mathfrak{E},\Omega)\,; $$
\item[(ii)] for every partition  $\mathfrak{E}=(E_1,\ldots,E_m)$ of $\Omega$, there exists $\{\mathfrak{E}_{k}\}_{k\in\mathbb{N}}=\{(E_{1,k},\ldots,E_{m,k})\}_{k\in\mathbb{N}}\subset\mathscr{A}_m^{s_k}(\Omega)$ such that  $\chi_{E_{i,k}}\to \chi_{E_{i}}$ in $L^1(\Omega)$ for every $i\in\{1,\ldots,m\}$, and 
$$\lim_{k\to\infty}\,(1-2s_k) \mathscr{P}^{\boldsymbol{\sigma}}_{2s_k}(\mathfrak{E}_{k},\Omega)= \omega_{n-1}\mathscr{P}_1^{\bar{\boldsymbol{\sigma}}}( \mathfrak{E},\Omega)\,. $$
\end{itemize}

To make  Theorem \ref{GCthm} useful, we have also analyzed the behavior as $s\to1/2^-$ of local minimizers of $\mathscr{P}^{\boldsymbol{\sigma}}_{2s}(\cdot,\Omega)$ in the spirit of \cite[Theorem 3]{ADPM} for local minimizers of $P_{2s}(\cdot,\Omega)$. Similarly to the single phase case, we say that $\mathfrak{E}=(E_{1},\ldots,E_{m})\in\mathscr{A}_m^{s}(\Omega)$ is a local minimizer of $\mathscr{P}^{\boldsymbol{\sigma}}_{2s}(\cdot,\Omega)$ if 
$$\mathscr{P}^{\boldsymbol{\sigma}}_{2s}(\mathfrak{E},\Omega)\leq  \mathscr{P}^{\boldsymbol{\sigma}}_{2s}(\mathfrak{F},\Omega)$$
for every competitor $\mathfrak{F}=(F_{1},\ldots,F_{m})\in\mathscr{A}_m^{s}(\Omega)$ such that $E_i\triangle F_i\Subset\Omega$ for every $i\in\{1,\ldots,m\}$.  
With the analogue definition concerning partitions of $\Omega$ locally minimizing  $\mathscr{P}_1^{\bar{\boldsymbol{\sigma}}}(\cdot,\Omega)$, we have

\begin{theorem}\label{convlocmin}
Let $s_k\to 1/2^-$ an arbitrary sequence. If $\mathfrak{E}_k=(E_{1,k},\ldots,E_{m,k})\in \mathscr{A}_m^{s_k}(\Omega)$ are local minimizers of $\mathscr{P}^{\boldsymbol{\sigma}}_{2s_k}(\cdot,\Omega)$ such that $\chi_{E_{i,k}}\to \chi_{E_i}$ in $L^1_{\rm loc}(\Omega)$ for every $i\in\{1,\ldots,m\}$, then 
\begin{equation}\label{bornsuplocmin}
\limsup_{k\to\infty}\,(1-2s_k)\mathscr{P}^{\boldsymbol{\sigma}}_{2s_k}(\mathfrak{E}_k,\Omega^\prime)<+\infty \quad\forall \Omega^\prime\Subset\Omega\,,
\end{equation}
$\mathfrak{E}=(E_1,\ldots,E_m)$ is a partition of $\Omega$, $\mathfrak{E}$ is a local minimizer of  $\mathscr{P}^{\bar{\boldsymbol{\sigma}}}_{1}(\cdot,\Omega)$, and 
$$\lim_{k\to\infty}\,(1-2s_k)\mathscr{P}^{\boldsymbol{\sigma}}_{2s_k}(\mathfrak{E}_k,\Omega^\prime)= \omega_{n-1}\mathscr{P}^{\bar{\boldsymbol{\sigma}}}_{1}(\mathfrak{E},\Omega^\prime) $$
for every $ \Omega^\prime\Subset\Omega$ such that $\sum_i{\rm Per}(E_i,\partial \Omega^\prime)=0$. 
\end{theorem}

Concerning the proofs of those theorems, many ingredients are of course borrowed from \cite{ADPM} but dealing with partitions introduce some specific and serious difficulties that need to be fully handled. Except perhaps for Theorem \ref{equicoercivethm} which is a quite direct consequence of \cite[Theorem~1]{ADPM}. Theorem \ref{GCthm} is as usual proved in two steps, a suitable upper bound for the $\Gamma-\limsup$ functional, and then a matching lower bound for the $\Gamma-\liminf$. Concerning the upper bound, we shall make use of the density result of polyhedral 
partitions in \cite{Bal}. For polyhedral partitions, a lengthy computation shows that $(1-2s)\mathscr{P}^{\boldsymbol{\sigma}}_{2s}(\cdot,\Omega)$ converges pointwise as $s\to1/2^-$ 
to the functional $\omega_{n-1}\mathscr{P}_1^{{\boldsymbol{\sigma}}}(\cdot,\Omega)$. At this stage, we could have argued that the resulting upper bound extends to arbitrary partitions by density, and then relaxe $\mathscr{P}_1^{{\boldsymbol{\sigma}}}(\cdot,\Omega)$ using \cite{AmBr} to recover  $\mathscr{P}_1^{\bar{\boldsymbol{\sigma}}}(\cdot,\Omega)$. Instead, we have relaxed ``by hands''
the functional $\mathscr{P}_1^{{\boldsymbol{\sigma}}}(\cdot,\Omega)$ restricted to polyhedral partitions, and then concluded by density. We found this strategy more enlightening to understand the relaxation process.  The whole $\Gamma-\limsup$ construction is the object of Section \ref{subsecGLS}. For the $\Gamma-\liminf$ inequality, we first obtain in Section \ref{subsecliminffirst} a lower estimate by means of the {\sl blow up method}. It provides a lower bound in terms of a perimeter functionals involving a matrix whose coefficients are abstractly defined through some asymptotic cell formula. As usual in this kind of business, we have to rework the cell formula to rephrase it as the limit of  minimum values of minimization problems in the cell. In order to do so, one has to reduce the class of competitors in the original cell formula. It requires a delicate construction to ``glue'' different competitors that we perform in Section \ref{geomconstcube}. It is well known that, for partitions, this matter is much more serious than for the single phase problem. In Section~\ref{secminhalfsp}, we address the minimization problem involved in the newly obtained cell formula. 
It consists in finding the minimizers of $\mathscr{P}^{\bar{\boldsymbol{\sigma}}}_{2s}$ that agree with the two components partition induced by the half space outside the cell. Here, 
we are able to prove that the partition induced by the half space itself is actually the unique minimizer. To prove this fact, 
we combine the known minimality of the half space for $P_{2s}$ from \cite{ADPM,CRS} with an argument inspired from \cite{Leo} and based on a classical result from graph theory and optimization, the {\sl max-flow min-cut theorem}. With this result in hands, the limiting value in the cell formula becomes explicitly computable and it returns the matrix $\omega_{n-1}\bar{\boldsymbol{\sigma}}$, whence the matching lower bound.  We prove Theorem \ref{convlocmin} in the very last Section \ref{sectcvlocmin}. Here again, the whole issue is to be able to glue an arbitrary recovery sequence of partitions with the given sequence of minimizing partitions without increasing too much the energy. As in \cite{ADPM}, we shall make use of a generalized co-area formula, but in a clever way to keep the partitioning constraint. Compare to \cite{ADPM}, the gluing construction is not energy sharp, and we have to show in addition that a good layer can be found to perform the gluing with an arbitrary small increase of energy. 
\vskip10pt

\noindent{\bf Notation.} Throughout the paper, we  write $\boldsymbol{\sigma}_{\rm min}:=\min_{i\not= j}\sigma_{ij}$ and $\boldsymbol{\sigma}_{\rm max}:=\max_{i,j}\sigma_{ij}$.

\section{Equi-coercivity}

\begin{proof}[Proof of Theorem \ref{equicoercivethm}]
 The proof is a consequence of  \cite[Theorem 1]{ADPM}. Indeed, let $\Omega' \Subset\Omega$ be an arbitrary open subset. For $\mathfrak{E}_k=(E_{1,k},\ldots,E_{m,k})$, we estimate
\begin{align*}
\mathscr{E}_{2s_k}^{\boldsymbol{\sigma}}  (\mathfrak{E}_k,\Omega') & = \frac{1}{2}\sum_{i=1}^m\Big(\sum_{j\not=i} \sigma_{ij}\,\mathcal{I}_{2s_k}(E_{i,k} \cap \Omega^\prime,E_{j,k} \cap \Omega^\prime) \Big)
\\
& \geq \frac{\boldsymbol{\sigma}_{\rm min}}{2} \sum_{i= 1}^m \mathcal{I}_{2s_k}(E_{i,k} \cap \Omega^\prime,E^c_{i,k} \cap \Omega^\prime) \,,
\end{align*}
using the fact that $\sum_{j\not= i}\chi_{E_{j,k}}=1-\chi_{E_{i,k}}=\chi_{E^c_{i,k}}$ since $\mathfrak{E}_k$ is a partition. Therefore,
$$ \sup_{k\in\mathbb{N}}\, \sum_{i= 1}^m (1-2s_k)\mathcal{I}_{2s_k}(E_{i,k} \cap \Omega^\prime,E^c_{i,k} \cap \Omega^\prime) <\infty\,.$$
According to 
\cite[Theorem 1]{ADPM}, each sequence $\{E_{i,k}\}_{k \in \N}$ is relatively compact in $L^1_{\rm loc}(\Omega) $ and any limit set has  locally finite perimeter in $\Omega$. 
Therefore, $\{\mathfrak{E}_k\}_{k\in\mathbb{N}}$ is relatively compact in $[L^1_{\rm loc}(\Omega)]^m$, and if $\mathfrak{E}=(E_1,\ldots,E_m)$ denotes an accumulation point, then $1=\sum_{i}\chi_{E_{i,k}}\to \sum_{i}\chi_{E_{i}}$ in $L^1_{\rm loc}(\Omega)$ for some (not relabeled) subsequence. Hence $\sum_{i}\chi_{E_{i}}=1$ a.e. in $\Omega$. Since each $E_i$ has locally finite 
perimeter in $\Omega$, we conclude that $\mathfrak{E}$ is a Caccioppoli partition of $\Omega^\prime$ for every open subset $\Omega^\prime\Subset \Omega$. 
\end{proof}

\begin{remark}
By the very same proof as above, if $\sup_{k} \,(1-2s_k) \mathscr{E}_{2s_k}^{\boldsymbol{\sigma}}  (\mathfrak{E}_k,\Omega) < \infty$, then the sequence 
$\{E_{i,k}\}_{k \in \N}$ is relatively compact in $L^1(\Omega) $, and any limit $\mathfrak{E}$ is a Caccioppoli partition of the whole domain $\Omega$. We shall see the full argument in the next section. 
\end{remark}

\section{Gamma convergence inequalities}

The purpose of this section is twofold. In a first subsection, we prove a preliminary lower bound for the $\Gamma-\liminf$ inequality. In the next subsection, we prove the upper bound for the $\Gamma-\limsup$, first in the case of polyhedral partitions, and then in the general case by relaxation, and then approximation. The strategy ressembles to the one in \cite{ADPM}, but some major differences appear  in the case of partitions, and we shall provide as many details as necessary to be self contained.

\subsection{A preliminary $\Gamma-\liminf$ inequality}\label{subsecliminffirst}

For a partition $\mathfrak{E}=(E_1,\ldots,E_m)$ of $\Omega$ made of measurable sets,  we consider the so-called $\Gamma(L^1(\Omega))$-lower limit of $(1-2s)\mathscr{P}^{\boldsymbol{\sigma}}_{2s}(\cdot,\Omega)$ at $\mathfrak{E}$, 
\begin{multline*}
({\mathscr{P}}^{\boldsymbol{\sigma}})_*(\mathfrak{E},\Omega):=\inf\Big\{\liminf_{k\to\infty}\,(1-2s_k)\mathscr{P}^{\boldsymbol{\sigma}}_{2s_k}(\mathfrak{E}_k,\Omega): s_k\to 1/2^-\,,\\
\mathfrak{E}_k=(E_{1,k},\ldots,E_{m,k})\in\mathscr{A}_m^{s_k}(\Omega)\,,\;
\chi_{E_{i,k}}\to \chi_{E_i}\text{ in $L^1(\Omega)$}\;\forall i\in\{1,\ldots,m\}\Big\}\,.
\end{multline*}
We aim  to obtain a preliminary lower bounds for the value of $(\mathscr{P}^{\boldsymbol{\sigma}})_*(\mathfrak{E},\Omega)$ by means of the {\sl blow-up method} first introduced in \cite{FonMul}. To this purpose, and for the rest of our discussion,  we shall denote by $Q$ the canonical unit cube of $\R^n$, i.e., 
$$Q:=(-\frac{1}{2},\frac{1}{2})^n\,,$$ 
and by $H$ the upper half-space of $\R^n$, that is 
 $$H:=\big\{x \in \R^n, x_n \geq 0 \big\}\,.$$ 
We also define the $m\times m$ matrix $\boldsymbol{\Gamma}=(\Gamma_{ij})$ by setting $\Gamma_{ii}=0$, and for 
 $i \neq j$, 
\begin{multline*} 
\Gamma_{ij} \vcentcolon = \inf \bigg\{ \liminf  \limits_{k \rightarrow \infty}  (1 -2s_k) \mathscr{E}_{2s_k}^{\boldsymbol{\sigma}}(\mathfrak{E}_{k},Q) : s_k \to 1/2^-\,,\; \mathfrak{E}_k=(E_{1,k},\ldots, E_{m,k})\in \mathscr{A}_m^{s_k}(Q)\,,\\
\chi_{E_{i,k}} \to \chi_H\,,\; \chi_{E_{j,k}} \to \chi_{H^c}\,,\; \text{and } \chi_{E_{l,k}} \to 0 \text{ for } l \notin \{i,j\}  
 \text{ in } L^1(Q) \bigg\}\, .
\end{multline*}
We start with an elementary observation based on \cite{ADPM}, implying in particular that $\boldsymbol{\Gamma}\in\mathscr{S}_m$. 

\begin{lemma}\label{esticonstfirst}
We have 
$$0< \omega_{n-1} \boldsymbol{\sigma}_{\rm min}\leq \Gamma_{ij}=\Gamma_{ji}\leq \omega_{n-1} \boldsymbol{\sigma}_{ij}<\infty\quad\text{for every $i\not=j$}\,.$$
\end{lemma}

\begin{proof}
Let us fix pair $(i_0,j_0)$ with $i_0\not=j_0$. For sequences $s_k\to 1/2^-$ and  $\{\mathfrak{E}_k\}_{k\in\mathbb{N}}$ as in the definition of $\Gamma_{i_0j_0}$, we have  similarly to the proof of Theorem \ref{equicoercivethm}, 
\begin{multline*}
\mathscr{E}_{2s_k}^{\boldsymbol{\sigma}}(\mathfrak{E}_{k},Q)\geq 
\frac{\boldsymbol{\sigma}_{\rm min}}{2} \sum_{i=1}^m\mathcal{I}_{2s_k}(E_{i,k}\cap Q,E^c_{i,k}\cap Q)\\
\geq  \frac{\boldsymbol{\sigma}_{\rm min}}{2} \mathcal{I}_{2s_k}(E_{i_0,k}\cap Q,E^c_{i_0,k}\cap Q) +  \frac{\boldsymbol{\sigma}_{\rm min}}{2} \mathcal{I}_{2s_k}(E_{j_0,k}\cap Q,E^c_{j_0,k}\cap Q)\,. 
\end{multline*}
According to  \cite[Theorem 2]{ADPM}, we have 
$$\liminf_{k\to\infty}(1-2s_k) \mathcal{I}_{2s_k}(E_{i_0,k}\cap Q,E^c_{i_0,k}\cap Q)\geq \omega_{n-1}{\rm Per}(H,Q)=\omega_{n-1}\,,$$
and 
$$\liminf_{k\to\infty}(1-2s_k) \mathcal{I}_{2s_k}(E_{j_0,k}\cap Q,E^c_{i_0,k}\cap Q)\geq \omega_{n-1}{\rm Per}(H^c,Q)=\omega_{n-1}\,.$$
Therefore, 
$$\liminf_{k\to\infty}(1-2s_k)\mathscr{E}_{2s_k}^{\boldsymbol{\sigma}}(\mathfrak{E}_{k},Q)\geq \omega_{n-1} \boldsymbol{\sigma}_{\rm min}\,, $$
and it follows that $\Gamma_{i_0j_0}\geq \omega_{n-1} \boldsymbol{\sigma}_{\rm min}$. 

On the other hand, for the partition $\mathfrak{E}^*=(E_1^*,\ldots,E_m^*)$ given by $E_{i_0}^*=H$, $E_{j_0}^*=H^c$, and $E_{l}^*=\emptyset$ for $l\not\in\{i_0j_0\}$, we have 
$$\mathscr{E}_{2s}^{\boldsymbol{\sigma}}(\mathfrak{E}^*,Q)= \sigma_{i_0j_0}\,\mathcal{I}_{2s}(H\cap Q,H^c\cap Q)\,.$$ 
A direct computation made in \cite[Lemma 9]{ADPM} leads to  
\begin{equation}\label{convhalfspace}
\lim_{s\to 1/2^-}(1-2s)\, \mathcal{I}_{2s}(H\cap Q,H^c\cap Q)=\lim_{s\to 1/2^-}(1-2s)P_{2s}(H, Q)=\omega_{n-1}\,.
\end{equation}
Hence,
$$\lim_{s\to 1/2^-}(1-2s)\mathscr{E}_{2s}^{\boldsymbol{\sigma}}(\mathfrak{E}^*,Q)= \omega_{n-1}\sigma_{i_0j_0}\,,$$
which leads to $\Gamma_{i_0j_0}\leq \omega_{n-1} \sigma_{i_0j_0}$. 

Finally, since the matrix $\boldsymbol{\sigma}$ is symmetric, one can exchange the role of $i_0$ and $j_0$ in the definition of  $\Gamma_{i_0j_0}$, and it shows that $\Gamma_{i_0j_0}=\Gamma_{j_0i_0}$.
\end{proof}

The lower bound we are seeking for $(\mathscr{P}^{\boldsymbol{\sigma}})_*$ is contained in the following proposition. 

\begin{proposition}\label{proplowerbd1}
Given $s_k\to 1/2$ and a sequence $\mathfrak{E}_k = (E_{1,k}, E_{2,k},...,E_{m,k})\in\mathscr{A}^{s_k}_m(\Omega)$  such that $\chi_{E_{i,k}} \rightarrow \chi_{E_i}$ in $L^1(\Omega)$   for every $i \in \{1,...,m\}$ as $k \rightarrow \infty$, one has 
\begin{equation} \label{Gamma liminf inequality}
\liminf \limits_{k \rightarrow \infty} (1- 2s_k) \mathcal{E}_{2s_k}^{\boldsymbol{\sigma}} (\mathfrak{E}_k,\Omega) \geq \frac{1}{2}\sum \limits_{i ,j=1}^m \Gamma_{ij} \mathcal{H}^{n-1} (\partial^* E_i \cap \partial^* E_j \cap \Omega)=:\mathscr{P}_1^{\boldsymbol{\Gamma}}(\mathfrak{E},\Omega)\,.
\end{equation}
\end{proposition}

\begin{proof}
Without loss of generality, we can assume that the left-hand side of \eqref{Gamma liminf inequality} is finite. We first claim that $\mathfrak{E}$ is a Caccioppoli partition of $\Omega$. 
From Theorem \ref{equicoercivethm}, we already know that $\mathfrak{E}$ is partition of $\Omega$, and that each $E_i$ has locally finite perimeter in $\Omega$. It then remains to show that each $E_i$ has actually finite perimeter in $\Omega$. 

Arguing as in the proof of Theorem \ref{equicoercivethm}, we start estimating 
\begin{multline}\label{eq1liminf}
(1-2s_k) \mathscr{E}_{2s_k}^{\boldsymbol{\sigma}} (\mathfrak{E}_k,\Omega)
 \geq \frac{{\boldsymbol{\sigma}_{\rm min}}}{2}  \sum \limits_{i=1}^m (1-2s_k)\,\mathcal{I}_{2s_k}(E_{i,k} \cap \Omega,E^c_{i,k} \cap \Omega)  \\
= \frac{{\boldsymbol{\sigma}_{\rm min}}}{4}  \sum \limits_{i=1}^m (1-2s_k) \iint_{\Omega \times \Omega} \frac{|\chi_{E_{i,k}}(x) - \chi_{E_{i,k}}(y)|}{|x-y|^{n+2s_k}}\, \de x\de y\,.
\end{multline}
Now, let $\Omega'  \Subset \Omega$ an arbitrary open subset, and $h \in \R^n$ such that $|h| < {\rm dist}(\Omega^\prime, \partial \Omega)/2$. We set $\Omega_{|h|}' = \{x \in \R^n :  {\rm dist}(x,\Omega') < |h|\}$. Since $\Omega'_{2|h|} \subseteq \Omega$, we infer from \eqref{eq1liminf} that 
\begin{align}
\nonumber  (1-2s_k) \mathscr{E}_{2s_k}^{\boldsymbol{\sigma}}  (\mathfrak{E}_k,\Omega) & \geq  \frac{{\boldsymbol{\sigma}_{\rm min}}}{4}  \sum \limits_{i=1}^m (1-2s_k) \iint_{\Omega'_{|h|}\times \Omega'_{2|h|}} \frac{|\chi_{E_{i,k}}(x) - \chi_{E_{i,k}}(y)|}{|x-y|^{n+2s_k}} \de x\de y \\
\nonumber & \geq \frac{{\boldsymbol{\sigma}_{\rm min}}}{4}  \sum \limits_{i=1}^m (1-2s_k)  \int_{\Omega'_{|h|}} \Big(\int_{B_{|h|}(x)} \frac{|\chi_{E_{i,k}}(x) - \chi_{E_{i,k}}(y)|}{|x-y|^{n+2s_k}} \de y\Big) \de x \\
& = \frac{{\boldsymbol{\sigma}_{\rm min}}}{4}  \sum \limits_{i=1}^m (1-2s_k)\int_{B_{|h|}} \frac{\|\tau_{\xi} \chi_{E_{i,k}}- \chi_{E_{i,k}} \|_{L^1(\Omega'_{|h|})}}{|\xi|^{n+2s_k}} \de\xi\,, \label{estiliminf2}
\end{align}
where $\tau_\xi$ denotes the translation by the vector $\xi$, i.e., $\tau_\xi u(x) := u(x + \xi)$. According to \cite[Proposition 4]{ADPM}, we have  
\begin{equation}\label{estitransAmbrDeP}
 \frac{\|\tau_h \chi_{E_{i,k}}- \chi_{E_{i,k}} \|_{L^1(\Omega')} }{|h|^{2s_k}}\leq C (1-2s_k) \int_{B_{|h|}} \frac{\|\tau_{\xi} \chi_{E_{i,k}}- \chi_{E_{i,k}} \|_{L^1(\Omega'_{|h|})}}{|\xi|^{n+2s_k}} \de\xi\,,
 \end{equation}
for a constant $C$ depending only on the dimension $n$. Inserting \eqref{estitransAmbrDeP} in \eqref{estiliminf2} yields 
$$ \sum_{i=1}^m\frac{\|\tau_h \chi_{E_{i,k}}- \chi_{E_{i,k}} \|_{L^1(\Omega')}}{|h|^{2s_k}} \leq C \boldsymbol{\sigma}^{-1}_{\rm min}(1-2s_k) \mathscr{E}_{2s_k}^{\boldsymbol{\sigma}}  (\mathfrak{E}_k,\Omega)\,.$$
Then, letting $k \rightarrow \infty$ leads to 
$$ \sum_{i=1}^m \frac{\|\tau_h \chi_{E_i}- \chi_{E_i} \|_{L^1(\Omega')}}{|h|} \leq C \boldsymbol{\sigma}^{-1}_{\rm min}\liminf_{k\to\infty}(1-2s_k) \mathscr{E}_{2s_k}^{\boldsymbol{\sigma}}  (\mathfrak{E}_k,\Omega)<\infty\,,$$
for a constant $C$ depending only on $n$. Since $|h|$ can be chosen arbitrily small, we conclude that  
$$ \sum_{i=1}^m {\rm Per}(E_i,\Omega^\prime)\leq C \boldsymbol{\sigma}^{-1}_{\rm min}\liminf_{k\to\infty}(1-2s_k) \mathscr{E}_{2s_k}^{\boldsymbol{\sigma}}  (\mathfrak{E}_k,\Omega)\,.$$
By arbitrariness of $\Omega^\prime$, it implies that each $E_i$ has finite perimeter in $\Omega$.
\vskip5pt

Now we consider the measure on $\Omega$  given by
 $$\mu  :=\frac{1}{2} \sum \limits_{i,j=1}^m \Gamma_{ij} \mathcal{H}^{n-1}_{\LL(\partial^*E_i \cap \partial^*E_j \cap \Omega)}\,.$$ 
 Note that, by Lemma \ref{esticonstfirst}, the measure $\mu$ is finite and absolutely continuous with respect to the finite measure $\mathcal{H}^{n-1}_{\LL(\bigcup_i\partial^*E_i \cap \Omega)}$. 
 On the other hand, by a well known property of Cacciopoli partitions, see e.g. \cite[Chapter 4, Section 4.4]{AFP}, there is a $\mathcal{H}^{n-1}$-negligible set $N\subset \Omega$ such that 
 for every $x\in \big(\bigcup_{i}\partial^*E_i\cap\Omega\big)\setminus N$, there exits a unique pair $(i_x,j_x)\subset\{1,\ldots,m\}$ with $i_x\not=j_x$ such that 
 $$x\in \partial^*E_{i_x}\cap\partial^*E_{j_x}\quad\text{and}\quad x\in E_l^0\text{ for } l\not\in\{i_x,j_x\}\,,$$
 where $E_l^0$ denotes the set of points of $E_l$ with $0$ $n$-dimensional density.  
 Since $|E_{i_x}\cap E_{j_x}|=0$, by De Giorgi's Theorem, see \cite[Theorem 3.59]{AFP}, there exists 
 $R_x \in SO(n)$ such that 
 \begin{equation}\label{blowupgamlinf}
 \chi_{(E_{i_x})_{x,r}} \to \chi_{R_xH} \quad\text{and}\quad \chi_{(E_{j_x})_{x,r}} \to \chi_{(R_xH)^c}\quad\text{in $L^1_{\rm loc}(\R^n)$ as $r\to0$}\,,
 \end{equation}
 and for $l\not\in\{i_x,j_x\}$, 
 \begin{equation}\label{blowupgamlinfbis}
 \chi_{(E_{l})_{x,r}} \to 0 \quad\text{in $L^1_{\rm loc}(\R^n)$ as $r\to0$}\,,
  \end{equation}
 where 
 $$(E_i)_{x,r}:= \frac{E_i-x}{r}\,.$$
 Since $\mathcal{H}^{n-1}_{\LL\partial^*E_l\cap\Omega}(E_l^0)=0$ (see e.g.  \cite[Theorem 3.61]{AFP}) and  
 according to \cite[(2.41)]{AFP}, adding to $N$ a $\mathcal{H}^{n-1}$-negligible set if necessary, we can also assume that 
 \begin{equation}\label{vanishn-1dens}
 \lim_{r\to 0}\frac{\mathcal{H}^{n-1}(\partial^*E_l\cap(x+rR_xQ))}{r^{n-1}}=0\quad \text{for $l\not\in\{i_x,j_x\}$}\,. 
 \end{equation}

We now claim  that 
\begin{equation} \label{outil57}
\lim_{r \rightarrow 0} \frac{\mu(x + rR_xQ)}{r^{n-1}} = \Gamma_{i_xj_x}\quad \text{for every } x\in \Big(\bigcup_{i=1}^m\partial^*E_i\cap\Omega\Big)\setminus N\,.
\end{equation}
Indeed, for such a point $x$, we have  by \eqref{blowupgamlinf} and \cite[Theorem 3.59]{AFP}, 
\begin{equation}\label{n-1densityblowup}
\lim_{r\to 0}\frac{\mathcal{H}^{n-1}(\partial^*E_{i_x}\cap(x+rR_xQ))}{r^{n-1}} =1\,.
\end{equation}
Since $\mathfrak{E}=(E_1,\ldots,E_m)$ is a partition of $\Omega$, it follows from  \cite[Theorem II.5.3]{Maggi} (see also \cite[Proof of Proposition 2.2]{Bal}) that 
$$\partial^*E_{i_x}\cap\Omega= \big(\bigcup_{j\not = {i_x}} \partial^*E_{i_x}\cap\partial^*E_j\cap\Omega\big)\cup \widetilde N\,,$$
where  $\widetilde N$ is a $\mathcal{H}^{n-1}$-negligible set. 
As a consequence, 
\begin{multline*}
\frac{\mathcal{H}^{n-1}(\partial^*E_{i_x}\cap(x+rR_xQ))}{r^{n-1}}\leq \frac{\mathcal{H}^{n-1}(\partial^*E_{i_x}\cap \partial^*E_{j_x}\cap(x+rR_xQ))}{r^{n-1}}\\
+\sum_{l\not\in\{i_x,j_x\}}\frac{\mathcal{H}^{n-1}(\partial^*E_{l}\cap(x+rR_xQ))}{r^{n-1}}\,,
\end{multline*}
and it follows from \eqref{vanishn-1dens} and \eqref{n-1densityblowup} that 
\begin{equation}\label{cpadutou}
\lim_{r\to 0} \frac{\mathcal{H}^{n-1}(\partial^*E_{i_x}\cap \partial^*E_{j_x}\cap(x+rR_xQ))}{r^{n-1}} = 1\,.
\end{equation}
Combining \eqref{cpadutou} with \eqref{vanishn-1dens} again, we derive that 
$$\lim_{r \rightarrow 0} \frac{\mu(x + rR_xQ)}{r^{n-1}} =\frac{1}{2}(\Gamma_{i_xj_x}+\Gamma_{j_xi_x}) = \Gamma_{i_xj_x}\,,$$
and \eqref{outil57} is proved. 
\vskip5pt

In the sequel, we denote  by $\mathcal{C}$ the family of all closed $n$-dimensional cubes in $\R^n$, that is  
$$ \mathcal{C} \vcentcolon = \big\{R(x + r\bar{Q}) : x \in \R^n\,,\; r > 0\,,\; R \in SO(n)\big\}\,. $$
For a closed cube $C \in \mathcal{C}$ such that $C\subset\Omega$, we define
$$ \alpha_k(C) \vcentcolon = (1 - 2s_k)\mathcal{E}_{2s_k}^{\boldsymbol{\sigma}}(\mathfrak{E}_k,C) \quad \text{and} \quad \alpha(C) \vcentcolon = \liminf \limits_{k \rightarrow \infty} \alpha_k(C)\,.$$
We set  $C_r(x)\vcentcolon = x + rR_x\bar{Q}$ with $R_x$ as in \eqref{outil57} if $x\in  \big(\bigcup_{i}\partial^*E_i\cap\Omega\big)\setminus N$, and $R_x=I_{\rm d}$ otherwise. 
We claim that
\begin{equation} \label{outil58}
\liminf_{r \rightarrow 0} \frac{\alpha(C_r(x))}{\mu(C_r(x))} \geq  1 \quad\text{for $\mu$-a.e. $x\in\Omega$}\,.
\end{equation}
We postpone the proof \eqref{outil58} and complete the proof of the lemma following the argument in \cite[Proof of Lemma 7]{ADPM}. 
By \eqref{outil58}, for any $\varepsilon > 0$, the family
$$ \mathcal{A}_\varepsilon \vcentcolon = \big\{ C_r(x) \subset \Omega :  (1 + \varepsilon) \alpha (C_r(x)) \geq  \mu(C_r(x))  \big\} $$
is a fine covering of $\mu$-almost all $\Omega$. By the variant of Vitali's theorem used in \cite{ADPM} (see \cite[part~3]{Morse}, the variant here concerns the fact that we are using cubes instead of balls in Vitali's covering theorem), we can extract a countable subfamily of disjoint cubes
$$ \big\{C_l \subset \Omega : l \in L_\varepsilon \big\} \subset \mathcal{A}_\varepsilon\text{ such that } \mu\Big(\Omega \backslash \bigcup \limits_{l \in L_\varepsilon} C_l\Big)= 0\,.$$
Then, 
\begin{multline*}
\mu(\Omega)  = \mu \Big(\bigcup \limits_{l \in L_\varepsilon}C_l\Big)  = \sum \limits_{l \in L_\varepsilon} \mu (C_l) \leq  (1 + \varepsilon) \sum \limits_{l \in L_\varepsilon} \alpha (C_l) \\
 \leq  (1 + \varepsilon) \liminf \limits_{k \rightarrow \infty} \sum \limits_{l \in L_\varepsilon} \alpha_k(C_l) 
 \leq (1 + \varepsilon) \liminf \limits_{k \rightarrow \infty} (1 - 2s_k) \mathcal{E}_{2s_k}^{\boldsymbol{\sigma}} (\mathfrak{E}_k,\Omega)\,.
\end{multline*}
In the last inequality, we used the fact that the cubes $C_l$ are pairwise disjoint and included in~$\Omega$, and that $\mathcal{E}_{2s_k}^{\boldsymbol{\sigma}} (\mathfrak{E}_k,\cdot)$ is a super additive set function.  Since $\varepsilon$ is arbitrary, this concludes the proof of the lemma. 
\vskip5pt

Now we  return to the proof of \eqref{outil58}. Since $\mu$ is absolutely continuous with respect to the measure $\mathcal{H}^{n-1}_{\LL(\bigcup_i\partial^*E_i \cap \Omega)}$, we have 
$$\mu\Big(\Omega\setminus \bigcup_i\partial^*E_i\Big) =0\quad\text{and}\quad \mu(N)=0\,.$$
Hence, it is enough to prove  \eqref{outil58} for $x\in \Big(\bigcup_{i=1}^m\partial^*E_i\cap\Omega\Big)\setminus N$. In view of \eqref{outil57}, it reduces to show that 
\begin{equation} \label{outil59}
\liminf \limits_{r \rightarrow 0} \frac{\alpha (C_r(x))}{r^{n-1}} \geq   \Gamma_{i_xj_x} \quad \text{for every } x\in \Big(\bigcup_{i=1}^m\partial^*E_i\cap\Omega\Big)\setminus N\,.
\end{equation}
For the rest of the proof, we then fix a point  $x\in \Big(\bigcup_{i}\partial^*E_i\cap\Omega\Big)\setminus N$, and we  assume without loss of generality that the rotation $R_x$ in \eqref{blowupgamlinf} is the identity matrix. 
We start choosing a sequence $r_p \rightarrow 0 $ such that
$$ \liminf \limits_{r \rightarrow 0} \frac{\alpha(C_r(x))}{r^{n-1}} = \lim \limits_{p \rightarrow \infty} \frac{\alpha(C_{r_p}(x))}{r_p^{n-1}}\, .$$
By \eqref{blowupgamlinf}, for every $p \in\mathbb{N}$, we can find an integer $ k(p)\geq p$ large enough so that the following conditions holds,
\begin{equation} \label{outil60}
\left\{
    \begin{array}{ll}
         \alpha_{k(p)}(C_{r_p}(x)) \leq \alpha(C_{r_p}(x)) + r_p^n \,,\\[8pt]
\displaystyle         r_p^{1-2s_{k(p)}} \geq 1 - \frac{1}{p} \,,\\[10pt]
\displaystyle         \frac{1}{r_p^n}\int_{C_{r_p}(x)} |\chi_{E_{i,k(p)}} - \chi_{E_{i}}|\, \de y < \frac{1}{p} \;\text{ for every $i\in\{1,\ldots,m\}$}\,. 
  \end{array}
\right.
\end{equation}
From \eqref{outil60} and a change of variables, it follows that 
\begin{align}
\nonumber \frac{\alpha(C_{r_p}(x))}{r_p^{n-1}} & \geq \frac{\alpha_{k(p)}(C_{r_p}(x))}{r_p^{n-1}} - r_p \\
\nonumber & = \frac{1}{r_p^{n-1}}(1-2s_{k(p)}) \mathcal{E}_{2s_{k(p)}}^{\boldsymbol{\sigma}}(\mathfrak{E}_{k(p)},C_{r_p}(x)) - r_p \\
\nonumber & = r_p^{1-2s_{k(p)}}  (1-2s_{k(p)}) \mathcal{E}_{2s_{k(p)}}^{\boldsymbol{\sigma}} \left((\mathfrak{E}_{k(p)})_{x,r_p}, Q \right) - r_p\\
\label{estiinfcpa} & \geq (1 - \frac{1}{p}) (1- 2s_{k(p)}) \mathcal{E}_{2s_{k(p)}}^{\boldsymbol{\sigma}} \left((\mathfrak{E}_{k(p)})_{x,r_p}, Q \right) - r_p\,, 
\end{align}
where we have set 
$$(\mathfrak{E}_{k(p)})_{x,r_p}:=\big((E_{1,k(p)})_{x,r_p},\ldots,(E_{m,k(p)})_{x,r_p}\big)\in\mathscr{A}_m^{s_{k(p)}}(Q)\,.$$
Passing to the limit $p \rightarrow \infty$ in \eqref{estiinfcpa} yields 
\begin{equation} 
\lim \limits_{p \rightarrow \infty} \frac{\alpha(C_{r_p}(x))}{r_p^{n-1}}  \geq \liminf \limits_{p \rightarrow \infty} (1 - 2s_{k(p)})  \mathcal{E}_{2s_{k(p)}}^{\boldsymbol{\sigma}} \left((\mathfrak{E}_{k(p)})_{x,r_p}, Q \right) \,.
\end{equation}\label{gfaim}
From the third equation in \eqref{outil60}, \eqref{blowupgamlinf} and \eqref{blowupgamlinfbis}, we infer that 
$$ \int_Q |\chi_{(E_{i_x,k(p)})_{x,r_p}} - \chi_H| \,\de y \leq  \int_Q |\chi_{(E_{i_x})_{x,r_p}} - \chi_H| \,\de y +  \frac{1}{r_p^n}\int_{C_{r_p}(x)} |\chi_{E_{i_x,k(p)}} - \chi_{E_{i_x}}|\, \de y 
\mathop{\longrightarrow}\limits_{p\to\infty}0\,,$$
$$ \int_Q |\chi_{(E_{j_x,k(p)})_{x,r_p}} - \chi_{H^c}| \,\de y \leq  \int_Q |\chi_{(E_{j_x})_{x,r_p}} - \chi_{H^c}| \,\de y +  \frac{1}{r_p^n}\int_{C_{r_p}(x)} |\chi_{E_{j_x,k(p)}} - \chi_{E_{j_x}}|\, \de y 
\mathop{\longrightarrow}\limits_{p\to\infty}0\,,$$
and 
$$ \big| (E_{l,k(p)})_{x,r_p}\cap Q\big| \leq \big| (E_{l})_{x,r_p}\cap Q\big| +  \frac{1}{r_p^n}\int_{C_{r_p}(x)} |\chi_{E_{l,k(p)}} - \chi_{E_{l}}|\, \de y 
\mathop{\longrightarrow}\limits_{p\to\infty}0\quad\text{for $l\not\in\{i_x,j_x\}$}\,.$$
By construction, $s_{k(p)}\to 1/2$, and it follows from \eqref{gfaim} and the very definition of $\Gamma_{i_x,j_x}$ that 
$$ \lim \limits_{p \rightarrow \infty} \frac{\alpha(C_{r_p}(x))}{r_p^{n-1}} \geq \Gamma_{i_xj_x}\,, $$
which completes the proof.
\end{proof}

As announced,  Proposition \ref{proplowerbd1} yields 

\begin{corollary}
For every partition $\mathfrak{E}$ of $\Omega$ by measurable sets, we have 
$$({\mathscr{P}}^{\boldsymbol{\sigma}})_*(\mathfrak{E},\Omega)\geq  \mathscr{P}_1^{\boldsymbol{\Gamma}}(\mathfrak{E},\Omega)\,.$$
\end{corollary}
\vskip5pt

\subsection{The Gamma$-\limsup $ inequality}\label{subsecGLS}

The purpose of this section is to prove the upper inequality in the $\Gamma$-convergence statement of Theorem \ref{GCthm}. We shall start with an elementary but useful 
description of the matrix $\bar{\boldsymbol{\sigma}}$. The upper bound for the $\Gamma-\limsup$ functional is presented in the next subsection.

\subsubsection{The relaxed coefficients}

We define the $m\times m$ matrix $\bar{\boldsymbol{\sigma}}=(\bar\sigma_{ij})\in\mathscr{S}_m$ by  setting $\bar\sigma_{ii}=0$, and 
\begin{equation}\label{defsigmabar}
\bar\sigma_{ij}:=\inf\bigg\{\sum_{h=0}^{H-1}\sigma_{i_hi_{h+1}} : H\in\mathbb{N}\setminus\{0\}\,, \{i_h\}_{h=0}^H\subset\{1,\ldots,m\}\,,i_0=i\text{ and }i_H=j\bigg\}\,. 
\end{equation}

According to the next lemma, the resulting matrix $\bar{\boldsymbol{\sigma}}$ is precisely  the largest matrix (componentwise) below $\boldsymbol{\sigma}$ satisfying the triangle inequality. 

\begin{lemma}\label{lemmrelaxcoeff}
For every $i\not=j$, the infimum in \eqref{defsigmabar}  is achieved by a path  $\{i_h\}_{h=0}^H\subset\{1,\ldots,m\}$ satisfying $H\leq m-1$ and $i_h\not=i_{h^\prime}$ for $h\not=h^\prime$. 
In addition, $\bar\sigma_{ij}\leq \sigma_{ij}$ for every $\{i,j\}\subset\{1,\ldots,m\}$, and 
\begin{equation}\label{traingineqlem}
\bar \sigma_{ij} \leq \bar \sigma_{ik}+\bar \sigma_{kj} \quad\text{for every $\{i,j,k\}\subset\{1,\ldots,m\}$}\,.
\end{equation}
Moreover, for every $\widetilde{\boldsymbol{\sigma}}=(\widetilde\sigma_{ij})\in \mathscr{S}_m$ satisfying \eqref{traingineqlem} and $\widetilde\sigma_{ij}\leq \sigma_{ij}$ for every $\{i,j\}\subset\{1,\ldots,m\}$, one has $\widetilde\sigma_{ij}\leq \bar \sigma_{ij}$ for every $\{i,j\}\subset\{1,\ldots,m\}$. In particular, $\bar\sigma_{ij}\geq \boldsymbol{\sigma}_{\rm min}$ for every $i\not=j$. 
\end{lemma}

\begin{proof}
{\it Step 1.} We claim that the infimum in \eqref{defsigmabar} can be restricted to admissible paths $\{i_h\}_{h=0}^H\subset\{1,\ldots,m\}$ satisfying $i_h\not=i_{h^\prime}$ for $h\not=h^\prime$. For such a path, one obviously have $H\leq m-1$, so that the infimum in  \eqref{defsigmabar}  can be taken over a finite set, and it is thus achieved. 

To prove the claim, we consider $\{i_h\}_{h=0}^H\subset\{1,\ldots,m\}$  an admissible path for \eqref{defsigmabar}, and we assume that there exists $h<h^\prime$ such that $i_h=i_{h^\prime}$. Then we consider $\{j_k\}_{k=0}^{H-(h^\prime-h)}\subset\{1,\ldots,m\}$ given by 
$$j_k:=\begin{cases}
i_k & \text{for $k\leq h$}\,,\\
i_{k+(h^\prime-h)} & \text{for $k\geq h+1$}\,.
\end{cases} $$
By construction $\{j_k\}_{k=0}^{H-(h^\prime-h)}$ is an admissible path for \eqref{defsigmabar}. Since $j_h=i_h=i_{h+(h^\prime-h)}$, we have 
\begin{align*}
\sum_{k=0}^{H-(h^\prime-h)-1}\sigma_{j_kj_{k+1}}& =\sum_{k=0}^{h-1}\sigma_{i_ki_{k+1}}+\sum_{k=h}^{H-(h^\prime-h)-1}\sigma_{i_{k+h^\prime-h}i_{k+1+h^\prime-h}}\\
&=\sum_{k=0}^{h-1}\sigma_{i_ki_{k+1}}+\sum_{k=h^\prime}^{H-1}\sigma_{i_ki_{k+1}}\\
&\leq \sum_{k=0}^{H-1} \sigma_{i_ki_{k+1}}\,.
\end{align*}
Iterating this procedure (finitely many times), we obtain a new admissible path $\{\tilde \iota_k\}_{k=0}^{\tilde H}$  for \eqref{defsigmabar} satisfying $\tilde\iota_k\not=\tilde \iota_{k^\prime}$ for $k\not=k^\prime$ and
$$\sum_{k=0}^{\tilde H-1} \sigma_{\tilde\iota_k\tilde\iota_{k+1}}\leq \sum_{k=0}^{H-1} \sigma_{i_ki_{k+1}}\,,$$
which proves the claim. 
\vskip5pt

\noindent {\it Step 2.} Considering the trivial path $\{i,j\}$ which is admissible for \eqref{defsigmabar}, we obtain that $\bar\sigma_{ij}\leq \sigma_{ij}$. 
We now prove \eqref{traingineqlem} fixing a triplet $\{i,j,k\}\subset\{1,\ldots,m\}$, and considering an optimal path $\{i_h\}_{h=0}^{H_1}$ for $\bar \sigma_{ik}$ and an optimal path  
$\{j_h\}_{h=0}^{H_2}$ for $\bar \sigma_{kj}$. Then the path $\{\tilde \iota\}_{h=0}^{H_1+H_2}$ given by 
$$\tilde\iota_{h} :=\begin{cases}
i_h & \text{for $h\leq H_1$}\,,\\
j_{h-H_1} &  \text{for $h\geq H_1+1$}\,,
\end{cases}$$
is admissible for $\bar \sigma_{ij}$, so that (since $i_{H_1}=j_0=k$)
$$\bar\sigma_{ij}\leq \sum_{h=0}^{H_1-1}\sigma_{\tilde\iota_{h}\tilde\iota_{h+1}} +\sum_{h=H_1}^{H_1+H_2-1} \sigma_{\tilde\iota_{h}\tilde\iota_{h+1}}  
= \sum_{h=0}^{H_1-1}\sigma_{i_{h}i_{h+1}} +\sum_{h=0}^{H_2-1} \sigma_{j_{h}j_{h+1}} =\bar \sigma_{ik}+\sigma_{kj}\,.$$

\noindent{\it Step 3.} Finally, let us consider $\widetilde{\boldsymbol{\sigma}}=(\widetilde\sigma_{ij})\in \mathscr{S}_m$ satisfying \eqref{traingineqlem} and $\widetilde\sigma_{ij}\leq \sigma_{ij}$ for every $\{i,j\}\subset\{1,\ldots,m\}$. For a  pair $(i,j)$, we consider an optimal path  $\{i_h\}_{h=0}^{H}$ for $\bar \sigma_{ij}$. Since  $\widetilde{\boldsymbol{\sigma}}$ satisfies the triangle inequality \eqref{traingineqlem}, we have 
$$\bar\sigma_{ij}= \sum_{h=0}^{H-1}\sigma_{i_hi_{h+1}}\geq  \sum_{h=0}^{H-1}\widetilde \sigma_{i_hi_{h+1}}\geq \widetilde \sigma_{i_0i_{H}}=\widetilde \sigma_{ij}\,.$$
In particular, this holds for the matrix $\widetilde{\boldsymbol{\sigma}}$ with identical coefficients $\widetilde \sigma_{ij}=\boldsymbol{\sigma}_{\rm min}$ for $i\not=j$, 
showing that $\bar\sigma_{ij}\geq\boldsymbol{\sigma}_{\rm min}$, 
and the proof is complete. 
\end{proof}

\subsubsection{The $\Gamma$-$\limsup$ functional}

Throughout this subsection, we fix an arbitrary sequence $s_k\to 1/2^-$, and we consider the $\Gamma(L^1(\Omega))-\limsup$ functional along $\{s_k\}$ of  $(1-2s)\mathscr{P}^{\boldsymbol{\sigma}}_{2s}(\cdot,\Omega)$. It is defined for a partition $\mathfrak{E}=(E_1,\ldots,E_m)$ of $\Omega$ made of measurable sets by  
\begin{multline*}
({\mathscr{P}}^{\boldsymbol{\sigma}})^*(\mathfrak{E},\Omega):=\inf\Big\{\limsup_{k\to\infty}\,(1-2s_k)\mathscr{P}^{\boldsymbol{\sigma}}_{2s_k}(\mathfrak{E}_k,\Omega):\mathfrak{E}_k=(E_{1,k},\ldots,E_{m,k})\in\mathscr{A}_m^{s_k}(\Omega)\,,\\
\chi_{E_{i,k}}\to \chi_{E_i}\text{ in $L^1(\Omega)$}\;\forall i\in\{1,\ldots,m\}\Big\}\,.
\end{multline*}
Our aim now is to obtain suitable upper bounds for the value of $({\mathscr{P}}^{\boldsymbol{\sigma}})^*(\mathfrak{E},\Omega)$. The final estimate will be provided by Proposition \ref{Gamma lim supbis} at the end of this subsection. 
We shall  start with a restrictive class of partitions made of polyhedral sets. More precisely, we introduce 
\begin{multline*}
\mathscr{B}_m(\Omega):= \Big\{\Pi=(\Pi_1,\ldots,\Pi_m): \text{each } \Pi_i\subset\R^n\text{ is a polyhedron}\,,\;\Pi_i\cap\Pi_j=\emptyset\text{ for $i\not=j$}\,,\\
|\R^n\setminus\bigcup_i\Pi_i|=0\,,\;\mathcal{H}^{n-1}(\partial\Pi_i\cap\partial\Omega)=0\;\forall i\in\{1,\ldots,m\} \Big\}\,.
\end{multline*}
We notice that $\mathscr{B}_m(\Omega)\subset \mathscr{A}_m^{s_k}(\Omega)$ for every $k\in\mathbb{N}$.

\begin{lemma}\label{Gamma limsup polhyedre}
For every $\Pi = (\Pi_1,...,\Pi_m)\in  \mathscr{B}_m(\Omega)$, we have
$$ \limsup_{k\to\infty} \,(1-2s_k) \mathscr{P}^{\boldsymbol{\sigma}}_{2s_k}(\Pi,\Omega) \leq \frac{\omega_{n-1}}{2}\sum \limits_{i, j=1}^m\sigma_{ij} \mathcal{H}^{n-1}(\partial\Pi_i \cap \partial \Pi_j \cap \Omega) = \omega_{n-1}\mathscr{P}_1^{\boldsymbol{\sigma}}(\Pi,\Omega)\,. $$
\end{lemma}

\begin{proof}
The proof of this lemma could be obtained from \cite[Lemma 8]{ADPM} with some modifications, but we prefer to provide all details for completeness. We fix a partition $\Pi = (\Pi_1,...,\Pi_m)\in  \mathscr{B}_m(\Omega)$, and we proceed in two steps. 
\vskip5pt

\noindent{\it Step 1.} We shall prove in this step that 
\begin{equation}\label{liminternonlocpolyh}
\limsup_{k\to\infty} \,(1-2s_k) \mathscr{E}^{\boldsymbol{\sigma}}_{2s_k}(\Pi,\Omega) \leq \frac{\omega_{n-1}}{2}\sum \limits_{i, j=1}^m\sigma_{ij} \mathcal{H}^{n-1}(\partial\Pi_i \cap \partial \Pi_j \cap \Omega)\,. 
\end{equation}
Since $\mathscr{E}^{\boldsymbol{\sigma}}_{2s_k}(\Pi,\Omega)=\frac{1}{2}\sum_{i,j}\sigma_{ij}\,\mathcal{I}_{2s_k}(\Pi_i\cap\Omega,\Pi_j\cap\Omega)$, it is enough to show that 
\begin{equation}\label{toproveinterfpipj}
 \limsup_{k\to\infty} \,(1-2s_k) \mathcal{I}_{2s_k}(\Pi_i\cap\Omega,\Pi_j\cap\Omega) \leq \omega_{n-1}\mathcal{H}^{n-1}(\partial\Pi_i \cap \partial \Pi_j \cap \Omega)\quad\text{for every $i\not=j$}\,.
 \end{equation}
We proceed similarly to  \cite[Lemma 8]{ADPM}, and we fix some pair $(i,j)$ with $i\not= j$. For simplicity, we shall drop the subscript $k$, and write $s$ instead of $s_k$. 
\vskip3pt

For $\varepsilon\in(0,1)$ fixed but arbitrary (small), we define  
$$(\partial \Pi_i)_{\varepsilon,j}  := \big\{ x \in \Omega : {\rm dist}(x,\partial \Pi_i \cap \partial \Pi_j) < \varepsilon \big\} = (\partial \Pi_j)_{\varepsilon,i}\quad\text{and}\quad 
(\partial \Pi_i)_{\varepsilon,j}^-  := (\partial \Pi_i)_{\varepsilon,j} \cap \Pi_i\,.
$$
Since $\Pi_i$ and $\Pi_j$ are disjoint polyhedral domains, $\partial \Pi_i \cap \partial \Pi_j$ can be written as 
\begin{equation}\label{decompbdpibdpj}
\partial \Pi_i \cap \partial \Pi_j=F_{ij}\cup A_{ij}\,,
\end{equation}
where  $F_{ij}$ is a finite union of $(n-1)$-dimensional polyhedron (the faces, possibly empty), and $A_{ij}$ is a set of lower dimensional polyhedral {\sl edges} with Hausdorff dimension less than $(n-2)$. In particular, we have $\mathcal{H}^{n-1}((\partial \Pi_i \cap \partial \Pi_j)\setminus F_{ij})=0$. 

Then we  find $N_\varepsilon$ disjoints cubes $Q_{i,l}^\varepsilon \subset \Omega$ ($1 \leq l \leq N_\varepsilon $) of side length $\varepsilon>0$ (small) such that
\begin{enumerate}
\item[(0)] $N_\varepsilon=0$ if $F_{ij}=\emptyset$;
\item[(1)] if $\widetilde{Q}_{i,l}^\varepsilon $ denotes the dilatation of $Q_{i,l}^\varepsilon $ by a factor $(1+\varepsilon)$, then each cube $\widetilde{Q}_{i,l}^\varepsilon $ intersects exactly one face $\Sigma$ of $F_{ij}$, its barycenter belongs to $\Sigma$, and each of its faces is either parallel or orthogonal to $\Sigma$;  
\item[(2)] $ \mathcal{H}^{n-1}\Big((\partial \Pi_i \cap \partial \Pi_j \cap \Omega \setminus \bigcup \limits_{l = 1}^{N_\varepsilon} Q_{i,l}^\varepsilon \Big) = \big|\mathcal{H}^{n-1}(\partial \Pi_i \cap \partial \Pi_j \cap \Omega ) - N_\varepsilon \varepsilon^{n-1}\big| \mathop{\longrightarrow}\limits_ {\varepsilon \rightarrow 0}0\,$.
\end{enumerate}
For a point $x \in \R^n$, we set
$$ I_{j,s}(x) := \int_{\Pi_j \cap \Omega} \frac{\de y}{|x-y|^{n+2s}} \,.$$
We need to consider several cases. 
\vskip5pt

\noindent{\it Case 1} : $x \in (\Pi_i \cap \Omega) \backslash (\partial \Pi_i)_{\varepsilon,j}^-$. Then for $y \in \Pi_j \cap \Omega$, we have $|x-y| \geq \varepsilon $ so that 
$$ I_{j,s}(x) \leq \int_{B_\varepsilon(x)^c} \frac{\de y}{|x-y|^{n+2s}} = n\omega_{n} \int_\varepsilon^\infty \frac{\de \rho}{\rho^{2s+1}} = \frac{n\omega_{n}}{2s \varepsilon^{2s}}\,. $$ 
Consequently, 
\begin{equation} \label{outil61bis}
\int_{(\Pi_i \cap \Omega) \backslash (\Pi_i)_{\varepsilon,j}^-} I_{j,s}(x)\, \de x \leq \frac{n\omega_{n} |\Pi_i \cap \Omega|}{2s \varepsilon^{2s}}\,.
\end{equation}
\vskip5pt

\noindent{\it Case 2} : $x \in (\partial \Pi_i)_{\varepsilon,j}^- \backslash \bigcup \limits_{l = 1}^{N_\varepsilon} Q_{i,l}^\varepsilon\,$. In this case, we have  

\begin{multline} \label{outil62bis}
I_{j,s}(x) \leq \int_{(B_{{\rm dist}(x,\Pi_j \cap \Omega)}(x))^c} \frac{\de y}{|x - y|^{n+2s}} = n \omega_n \int_{{\rm dist}(x,\Pi_j \cap \Omega)}^\infty \frac{\de \rho}{\rho^{2s+1}} \\
= \frac{n \omega_n}{2s ({\rm dist}(x,\Pi_j \cap \Omega))^{2s}}\,.
\end{multline}
Then we write 
$$F_{ij} \cap \Omega = \bigcup \limits_{k = 1}^{K_{ij}} \Sigma^{ij}_{k}\quad\text{and}\quad A_{ij}\cap\overline\Omega=\bigcup_{\ell=1}^{L_{ij}}S_\ell^{ij} \,,$$
where each $\Sigma_{k}^{ij}$ is the intersection of a face of $F_{ij}$ with $\Omega$,  and each $S_\ell^{ij}$ is a subset of $ A_{ij}$ included in $P_{\ell}^{ij}\cap\overline\Omega$ for an affine subspace $P_\ell^{ij}\subset\R^n$ of dimension $n-2$, i.e., $P_\ell^{ij}=a^{ij}_{\ell}+V_\ell^{ij}$ with $a^{ij}_{\ell}\in\R^n$ and $V_\ell^{ij}\subset\R^N$ a linear subspace with ${\rm dim}\,V_\ell^{ij}=n-2$  (if $n=2$, then $V_\ell^{ij}=\{0\}$ and the $S_\ell^{ij}$'s are simply the singletons $\{a_\ell^{ij}\}$).  

We define
$$ (\partial \Pi_i)_{\varepsilon,j,k}^F := \big\{x \in (\partial \Pi_i)_{\varepsilon,j}^- : {\rm dist}(x,\Pi_j \cap \Omega) = {\rm dist}(x, \Sigma_{k}^{ij}) \big\}\, ,$$
and 
$$(\partial \Pi_i)_{\varepsilon,j,\ell}^{A} :=   \big\{x \in (\partial \Pi_i)_{\varepsilon,j}^- : {\rm dist}(x,\Pi_j \cap \Omega) ={\rm dist}(x,S_\ell^{ij} ) \big\}\, .$$
By construction, 
$$ (\partial \Pi_i)_{\varepsilon,j}^- = \bigcup \limits_{k = 1}^{K_{ij}} (\partial \Pi_i)_{\varepsilon,j,k}^F\cup \bigcup \limits_{\ell = 1}^{L_{ij}} (\partial \Pi_i)_{\varepsilon,j,\ell}^A\,.$$
Next we observe that 
\begin{equation} \label{outil63bis}
 (\partial \Pi_i)_{\varepsilon,j,k}^F \subset \big\{x + t\nu : x \in \Sigma_{k,\varepsilon}^{ij}\, ,\; t\in (0,\varepsilon)\,,\; \nu \text{ is the interior unit normal to } \Sigma_{k,\varepsilon}^{ij} \big\}\,, 
 \end{equation}
where $\Sigma_{k,\varepsilon}^{ij}$ is the set of points $x$ that belong to the same hyperplane as $\Sigma_{k}^{ij}$ and satisfying ${\rm dist}(x,\Sigma_{k}^{ij}) \leq \varepsilon$. 
Since $\Sigma_{k}^{ij}$ has a Lipschitz regular boundary (when seen as a set of an $(n-1)$-dimensional Euclidean space), we have 
\begin{equation}\label{approxareaface}
\mathcal{H}^{n-1}(\Sigma_{k,\varepsilon}^{ij}\setminus \Sigma_{k}^{ij}) \leq C\varepsilon \quad\text{as $\varepsilon \rightarrow 0$}\,.
\end{equation}
Similarly, we notice that 
\begin{equation} \label{outil63bisbis}
 (\partial \Pi_i)_{\varepsilon,j,\ell}^A \subset \big\{x + t\nu : x \in S_{\ell,\varepsilon}^{ij}\, ,\; t\in (0,\varepsilon)\,,\; \nu\in \mathbb{S}^{n-1}\cap (V_\ell^{ij})^\perp  \big\}\,, 
 \end{equation}
where $S_{\ell,\varepsilon}^{ij}:=\{x\in P_{\ell}^{ij}: {\rm dist}(x,S_{\ell}^{ij})\leq\varepsilon\}$. Since $S_{\ell}^{ij}$ is a bounded subset of $P_{\ell}^{ij}$, we have 
\begin{equation}\label{bdhnmoins2}
\mathcal{H}^{n-2}(S_{\ell,\varepsilon}^{ij})\leq C \quad\text{as $\varepsilon \rightarrow 0$}\,.
\end{equation}
Using \eqref{outil62bis}, we may now estimate
\begin{multline}\label{dima1}
\int_{(\partial \Pi_i)_{\varepsilon,j}^- \backslash \bigcup \limits_{l = 1}^{N_\varepsilon} Q_{i,l}^\varepsilon} I_{j,s}(x)\, \de x \leq \frac{n \omega_n}{2s} \int_{(\partial \Pi_i)_{\varepsilon,j}^- \backslash \bigcup \limits_{l = 1}^{N_\varepsilon} Q_{i,l}^\varepsilon} \frac{\de x}{[{\rm dist}(x,\Pi_j \cap \Omega)]^{2s}}  \\
 \leq \frac{n \omega_n}{2s} \sum \limits_{k = 1}^{K_{ij}} \int_{(\partial \Pi_i)_{\varepsilon,j,k}^F \backslash \bigcup \limits_{l = 1}^{N_\varepsilon} Q_{i,l}^\varepsilon} \frac{\de x}{[{\rm dist}(x,\Sigma_{k}^{ij})]^{2s}}  +\frac{n \omega_n}{2s} \sum \limits_{\ell= 1}^{L_{ij}}  \int_{(\partial \Pi_i)_{\varepsilon,j,\ell}^A} \frac{\de x}{[{\rm dist}(x,S_{\ell}^{ij})]^{2s}}  \,.
\end{multline}
On one side, we infer from \eqref{outil63bis} and \eqref{approxareaface} that
\begin{multline}\label{dima2}
\int_{(\partial \Pi_i)_{\varepsilon,j,k}^F \backslash \bigcup \limits_{l = 1}^{N_\varepsilon} Q_{i,l}^\varepsilon} \frac{\de x}{[{\rm dist}(x,\Sigma_{k}^{ij})]^{2s}}\leq \int_{(\partial \Pi_i)_{\varepsilon,j,k}^F \backslash \bigcup \limits_{l = 1}^{N_\varepsilon} Q_{i,l}^\varepsilon} \frac{\de x}{[{\rm dist}(x,\Sigma_{k,\varepsilon}^{ij})]^{2s}}\\
\leq  \int_{\Sigma_{k,\varepsilon}^{ij} \backslash \bigcup \limits_{l = 1}^{N_\varepsilon} Q_{i,l}^\varepsilon} \left( \int_0^\varepsilon \frac{\de t}{t^{2s}} \right)\, \de\mathcal{H}^{n-1}=\frac{\varepsilon^{1-2s}}{(1-2s)}\mathcal{H}^{n-1}\Big(\Sigma_{k,\varepsilon}^{ij} \backslash \bigcup \limits_{l = 1}^{N_\varepsilon} Q_{i,l}^\varepsilon\Big)\\
\leq \frac{\varepsilon^{1-2s}}{(1-2s)}\mathcal{H}^{n-1}\Big(\Sigma_{k}^{ij} \backslash \bigcup \limits_{l = 1}^{N_\varepsilon} Q_{i,l}^\varepsilon\Big) + \frac{C\varepsilon^{2-2s}}{(1-2s)}\,,
\end{multline}
for a constant $C$ independent of $s$ and $\varepsilon$. On the other hand, \eqref{outil63bisbis} and \eqref{bdhnmoins2} leads to 
\begin{multline}\label{dima3}
\int_{(\partial \Pi_i)_{\varepsilon,j,\ell}^A} \frac{\de x}{[{\rm dist}(x,S_{\ell}^{ij})]^{2s}}\leq  \int_{(\partial \Pi_i)_{\varepsilon,j,\ell}^A} \frac{\de x}{[{\rm dist}(x,S_{\ell,\varepsilon}^{ij})]^{2s}}\\
\leq \int_{S_{\ell,\varepsilon}^{ij}} \left(2\pi\int_0^\varepsilon t^{1-2s}\,\de t\right)\,\de \mathcal{H}^{n-2}\leq C\varepsilon^{2-2s}\,,
\end{multline}
for a constant $C$ independent of $s$ and $\varepsilon$.

Gathering \eqref{dima1}, \eqref{dima2}, and \eqref{dima3}, we deduce that 
\begin{equation}\label{dima4}
\int_{(\partial \Pi_i)_{\varepsilon,j}^- \backslash \bigcup \limits_{l = 1}^{N_\varepsilon} Q_{i,l}^\varepsilon} I_{j,s}(x)\, \de x 
\leq \frac{C}{2s(1-2s)}\mathcal{H}^{n-1}\Big(F_{ij} \backslash \bigcup \limits_{l = 1}^{N_\varepsilon} Q_{i,l}^\varepsilon\Big) +\frac{C\varepsilon}{2s(1-2s)}=\frac{o_\varepsilon(1)}{2s(1-2s)}\,,
\end{equation}
with an error term $o_\varepsilon(1)\to 0$ as $\varepsilon\to 0$ which is independent of  $s$.  
\vskip5pt

\noindent{\it Case 3} : $x \in \Pi_i \cap \bigcup \limits_{l = 1}^{N_\varepsilon} Q_{i,l}^\varepsilon$. We separate $ I_{j,s} (x) $ in two parts, that is
$$ I_{j,s} (x)  = \int_{(\Pi_j \cap \Omega) \cap \{|x - y| \geq \varepsilon^2 \}} \frac{\de y}{|x-y|^{n+2s}} + \int_{(\Pi_j \cap \Omega) \cap \{ |x - y| < \varepsilon^2 \}} \frac{\de y}{|x-y|^{n+2s}} 
  =: I_{j,s}^1(x) + I_{j,s}^2(x)\,.$$
Then,
$$ I_{j,s}^1(x) \leq n \omega_n \int_{\varepsilon^2}^\infty \frac{\de\rho}{\rho^{2s+1}}  = \frac{n \omega_n}{2s \varepsilon^{4s}}\,, $$
and since $Q_{i,l}^\varepsilon \subset \Omega$ (and the cubes are mutually disjoint),
\begin{equation} \label{outil65bis}
\int_{\Pi_i \cap \bigcup \limits_{l = 1}^{N_\varepsilon} Q_{i,l}^\varepsilon} I_{j,s}^1(x)\, \de x \leq \frac{n \omega_n |\Omega|}{2s \varepsilon^{4s}}\,.
\end{equation}
Concerning $I_{j,s}^2(x)$, we observe that if $x \in Q_{i,l}^\varepsilon$ and $ |x - y|<\varepsilon^2 $, then $ y \in \widetilde{Q}_{i,l}^\varepsilon$. Thus, 
\begin{multline*}
\int_{\Pi_i \cap \bigcup \limits_{l = 1}^{N_\varepsilon} Q_{i,l}^\varepsilon} I_{j,s}^2(x)\, \de x  \leq \sum \limits_{l = 1}^{N_\varepsilon} \int_{\Pi_i \cap Q_{i,l}^\varepsilon} \int_{\Pi_j \cap \widetilde{Q}_{i,l}^\varepsilon} \frac{\de x \de y}{|x - y|^{n+2s}}  \\
 \leq \sum \limits_{l = 1}^{N_\varepsilon} \int_{\Pi_i \cap \widetilde{Q}_{i,l}^\varepsilon} \int_{\Pi_j \cap \widetilde{Q}_{i,l}^\varepsilon} \frac{\de x \de y}{|x - y|^{n+2s}} \,.
\end{multline*}
By a change of variables, we obtain 
\begin{align*}
\sum \limits_{l = 1}^{N_\varepsilon} \int_{\Pi_i \cap \widetilde{Q}_{i,l}^\varepsilon} & \int_{\Pi_j \cap \widetilde{Q}_{i,l}^\varepsilon} \frac{\de x \de y}{|x - y|^{n+2s}}  =
N_\varepsilon \int_{H \cap (\varepsilon + \varepsilon^2)Q} \int_{H^c \cap (\varepsilon + \varepsilon^2) Q} \frac{\de x\de y}{|x-y|^{n+2s}} \\[5pt]
&= N_\varepsilon (\varepsilon + \varepsilon^2)^{n-2s} \mathcal{I}_{2s}(H\cap Q,H^c\cap Q)\\[5pt]
&= \varepsilon^{1-2s}(1+\varepsilon)^{n-2s} \mathcal{I}_{2s}(H\cap Q,H^c\cap Q) \big(\mathcal{H}^{n-1}(\partial \Pi_i \cap \partial \Pi_j \cap \Omega )+o_\varepsilon(1)\big)\,,
\end{align*}
with an error term $o_\varepsilon(1)\to 0$ as $\varepsilon\to 0$ which is independent of  $s$.  Therefore, 
\begin{equation} \label{outil66bis}
\int_{\Pi_i \cap \bigcup \limits_{l = 1}^{N_\varepsilon} Q_{i,l}^\varepsilon} I_{j,s}^2(x)\, \de x  \leq (1+\varepsilon)^{n-2s} \mathcal{I}_{2s}(H\cap Q,H^c\cap Q) \big(\mathcal{H}^{n-1}(\partial \Pi_i \cap \partial \Pi_j \cap \Omega )+o_\varepsilon(1)\big)\,.
\end{equation}
Combining \eqref{outil65bis} with \eqref{outil66bis} leads to 
\begin{multline} \label{outil66bisbis}
\int_{\Pi_i \cap \bigcup \limits_{l = 1}^{N_\varepsilon} Q_{i,l}^\varepsilon} I_{j,s}(x)\, \de x \\
 \leq (1+\varepsilon)^{n-2s} \mathcal{I}_{2s}(H\cap Q,H^c\cap Q) \big(\mathcal{H}^{n-1}(\partial \Pi_i \cap \partial \Pi_j \cap \Omega )+o_\varepsilon(1)\big)
+\frac{n \omega_n |\Omega|}{2s \varepsilon^{4s}}\,.
\end{multline}
\vskip5pt

\noindent{\it Conclusion of the three cases.} Recalling \cite[Lemma 9]{ADPM} (see \eqref{convhalfspace}), and since 
$$\mathcal{I}_{2s_k}(\Pi_i\cap\Omega,\Pi_j\cap\Omega)=\int_{\Pi_i\cap\Omega} I_{j,s}(x)\, \de x\,,$$
gathering \eqref{outil61bis}, \eqref{dima4}, and \eqref{outil66bisbis}, we obtain that 
$$\limsup_{s\to 1/2}\,(1-2s)\mathcal{I}_{2s_k}(\Pi_i\cap\Omega,\Pi_j\cap\Omega)\leq  \omega_{n-1} \mathcal{H}^{n-1}(\partial \Pi_i \cap \partial \Pi_j \cap \Omega )+o_\varepsilon(1)\,.$$
Letting now $\varepsilon\to 0$ yields \eqref{toproveinterfpipj} which completes to proof of \eqref{liminternonlocpolyh}. 
\vskip5pt

\noindent{\it Step 2.} Once again, we drop the subscript $k$, and write $s$ instead of $s_k$. We shall prove that 
$$\lim_{s\to1/2} \,(1-2s) \mathscr{F}^{\boldsymbol{\sigma}}_{2s_k}(\Pi,\Omega) = 0\,,$$
which, in view of \eqref{liminternonlocpolyh}, completes the proof of the lemma. 

Since $\mathscr{F}^{\boldsymbol{\sigma}}_{2s}(\Pi,\Omega)=\frac{1}{2}\sum_{i,j} \sigma_{ij}  \big[ \mathcal{I}_{2s}(\Pi_i \cap \Omega,\Pi_j \cap \Omega^c) + \mathcal{I}_{2s}(\Pi_i \cap \Omega^c,\Pi_j \cap \Omega) \big] $, it is enough to prove that 
\begin{equation}\label{quasifin2}
\lim_{s\to1/2} \,(1-2s)\, \mathcal{I}_{2s}(\Pi_i\cap\Omega,\Pi_j\cap\Omega^c) =0 \quad\text{for every $i\not=j$}\,.
\end{equation}
For $\delta\in(0,1)$ arbitrary (small), we set 
\begin{equation} \label{Omega deltabis}
\Omega_\delta^{\rm ext}  := \big\{x \in \Omega^c : {\rm dist}(x,\Omega) < \delta \big\} \quad\text{and}\quad \Omega_\delta^{\rm in}  := \big\{x \in \Omega : {\rm dist}(x,\Omega^c) < \delta \big\}\,.
\end{equation}
Note that $\Omega_\delta^{\rm ext}\cup  \Omega_\delta^{\rm in}$ is an open neighborhood of $\partial\Omega$. 
\vskip3pt

For a point $x \in \R^n$, we now set 
$$ I^c_{j,s}(x) := \int_{\Pi_j \cap \Omega^c} \frac{\de y}{|x-y|^{n+2s}} \,,$$
and we consider two distinct cases. 
\vskip5pt

\noindent{\it Case 1}: $x \in \Pi_i \cap (\Omega \setminus \Omega_\delta^{\rm in})$. For $y \in \Pi_j \cap \Omega^c$, we have $|x - y| \geq \delta $ so that
$$ I^c_{j,s}(x) \leq \int_{B(x,\delta)^c} \frac{\de y}{|x-y|^{n+2s}} = n \omega_n \int_\delta^\infty \frac{\de \rho}{\rho^{2s+1}} = \frac{n \omega_n}{2s \delta^{2s}} \,.$$
\vskip5pt

\noindent{\it Case 2}: $x \in \Pi_i \cap  \Omega_\delta^{\rm in}$. Arguing as above in case  $y \in \Omega^c \setminus \Omega_\delta^{\rm ext}$, we estimate 
$$ I^c_{j,s}(x) \leq  \frac{n \omega_n}{2s \delta^{2s}} + \int_{\Pi_j \cap \Omega_\delta^{\rm ext}} \frac{\de y}{|x-y|^{n+2s}} \,.$$ 
Since
$$\mathcal{I}_{2s}(\Pi_i\cap\Omega,\Pi_j\cap\Omega^c)=\int_{\Pi_i\cap\Omega}  I^c_{j,s}(x) \,\de x\,, $$
gathering Case 1 and Case 2 leads to 
\begin{multline*}
\mathcal{I}_{2s}(\Pi_i\cap\Omega,\Pi_j\cap\Omega^c)\leq \mathcal{I}_{2s}(\Pi_i\cap\Omega_{\delta}^{\rm in},\Pi_j\cap\Omega_\delta^{\rm ext}) +\frac{n \omega_n|\Omega|}{s \delta^{2s}}\\
\leq \mathcal{I}_{2s}\big(\Pi_i\cap(\Omega_{\delta}^{\rm in}\cup \Omega_\delta^{\rm ext}),\Pi_j\cap(\Omega_{\delta}^{\rm in}\cup \Omega_\delta^{\rm ext})\big) +\frac{n \omega_n|\Omega|}{s \delta^{2s}}\,.
\end{multline*}
Applying now Step 1 (proof of \eqref{toproveinterfpipj}) in the open set $\Omega_{\delta}^{\rm in}\cup \Omega_\delta^{\rm ext}$ instead of $\Omega$, we infer that 
\begin{equation}\label{quasifin}
\limsup_{s\to 1/2}\,(1-2s)\, \mathcal{I}_{2s}(\Pi_i\cap\Omega,\Pi_j\cap\Omega^c)\leq \omega_{n-1} \mathcal{H}^{n-1}\big(\partial \Pi_i \cap \partial \Pi_j \cap (\Omega_{\delta}^{\rm in}\cup \Omega_\delta^{\rm ext}) \big)\,.
\end{equation}
Since
$$\bigcap_{\delta>0} (\Omega_{\delta}^{\rm in}\cup \Omega_\delta^{\rm ext}) =\partial\Omega\,, $$
by monotone convergence we have 
$$\lim_{\delta\to 0}  \mathcal{H}^{n-1}\big(\partial \Pi_i \cap \partial \Pi_j \cap (\Omega_{\delta}^{\rm in}\cup \Omega_\delta^{\rm ext}) \big)= \mathcal{H}^{n-1}\big(\partial \Pi_i \cap \partial \Pi_j \cap \partial\Omega) =0\,.$$
Therefore, letting $\delta\to 0$ in \eqref{quasifin} we deduce that \eqref{quasifin2} holds, and the proof is complete. 
\end{proof}

\begin{lemma} \label{Polyhedres approchent polyhedre}
Let $\Pi = (\Pi_1,...,\Pi_m)\in  \mathscr{B}_m(\Omega)$. For every $\varepsilon > 0$, there exists a sequence $\{\Pi_q\}_{q \in \N}=\{ (\Pi_{1,q},...,\Pi_{m,q})\}_{q \in \N}$ in $ \mathscr{B}_m(\Omega)$ 
 such that $ \chi_{\Pi_{i,q}} \to \chi_{\Pi_i} $ in $L^1(\Omega)$ for every $i\in\{1,\ldots,m\}$ and satisfying
$$ \limsup \limits_{q \rightarrow \infty} \mathscr{P}_1^{\boldsymbol{\sigma}}(\Pi_q,\Omega) \leq  \mathscr{P}_1^{\bar{\boldsymbol{\sigma}}}(\Pi,\Omega) + \varepsilon \, .$$
\end{lemma}

\begin{proof}
We fix $\Pi = (\Pi_1,...,\Pi_m)\in  \mathscr{B}_m(\Omega)$ and $\varepsilon>0$. We also consider $\delta=\delta(\varepsilon)>0$ small enough, to be determined later. 
We shall use the notations from the proof of Lemma \ref{Gamma limsup polhyedre}. Recalling that each interface $\partial\Pi_i\cap\partial\Pi_j$ with $i\not=j$ can be decomposed as \eqref{decompbdpibdpj}, we can find 
$M_\delta$ disjoint {\sl closed} cubes included in $\Omega$ and  of side length $\delta$, 
$$\Big\{ \{Q^\delta_{ij,l}\}_{l=1}^{N^{ij}_\delta} \Big\}_{i<j}\quad\text{with}\quad \sum_{i<j}N_\delta^{ij} =M_\delta\,,$$ 
such that for $i<j$, 
\begin{enumerate}
\item[(0)] $N^{ij}_\delta=0$ if $F_{ij}=\emptyset$ (i.e., the subcollection $\{Q^\delta_{ij,l}\}_{l=1}^{N^{ij}_\delta}$ is empty);
\item[(1)]  each cube ${Q}_{ij,l}^\delta $ is included in the interior of $\overline\Pi_i\cup\overline\Pi_j$;
\item[(2)] each cube ${Q}_{ij,l}^\delta $ only intersects $F_{ij}$ (and not $A_{i j}$),  ${Q}_{ij,l}^\delta $ intersects 
exactly one face $\Sigma$ of $F_{ij}$, its barycenter belongs to $\Sigma$, and each of its faces is either parallel or orthogonal to $\Sigma$;  
\item[(3)] $ \mathcal{H}^{n-1}\Big((\partial \Pi_i \cap \partial \Pi_j \cap \Omega \setminus \bigcup \limits_{l = 1}^{N^{ij}_\delta} Q_{ij,l}^\delta \Big) = \big|\mathcal{H}^{n-1}(\partial \Pi_i \cap \partial \Pi_j \cap \Omega ) - N^{ij}_\delta \delta^{n-1}\big|=:{\rm e}_{ij}(\delta) \mathop{\longrightarrow}\limits_ {\delta \rightarrow 0}0\,$.
\end{enumerate}
We shall define the sequence $\{\Pi_q\}_{q \in \N}$ by suitably modifying $\Pi$ in each cube $Q^\delta_{ij,l}$ (since each (closed) cube is included in $\Omega$, the construction will ensure that $\Pi_q\in  \mathscr{B}_m(\Omega)$). In other words, we set 
\begin{equation}\label{bougpaorcube}
\Pi_{k,q}\setminus \Big(\bigcup_{i<j}\bigcup_{l=1}^{N^\delta_{ij}}Q^\delta_{ij,l}\Big):=\Pi_{k}\setminus \Big(\bigcup_{i<j}\bigcup_{l=1}^{N^\delta_{ij}}Q^\delta_{ij}\Big)\quad\text{for every $k\in\{1,\ldots,m\}$}\,.
\end{equation}
{\it Construction in $Q^\delta_{ij,l}$.}  First, by conditions (1) and (2), we can translate and rotate coordinates in such a way that  $Q^\delta_{ij,l}=\delta Q$, as well as $\Pi_i\cap \delta Q=\delta Q\cap\{x_n<0\}$ and $\Pi_j\cap \delta Q=\delta Q\cap\{x_n>0\}$. In particular, $\Pi_k\cap \delta Q=\emptyset$ for $k\not\in\{i,j\}$. 
By Lemma \ref{lemmrelaxcoeff}, we can find $\{i_h\}_{h=0}^H\subset\{1,\ldots,m\}$ with $H\leq m-1$ such that $i_0=i$, $i_H=j$,  $i_h\not=i_{h^\prime}$ for $h\not=h^\prime$, and 
$$\bar\sigma_{ij}=\sum_{h=0}^{H-1}\sigma_{i_hi_{h+1}}\,. $$
{\it Case 1.} If $H=1$, then we simply set 
$$\Pi_{k,q}\cap Q^\delta_{ij,l}:=\Pi_{k}\cap Q^\delta_{ij,l}\quad\text{for every $k\in\mathbb{N}$}\,.$$
{\it Case 2 : $H\geq 2$.} We set $t_q:=2^{-(q+1)}/(H-1)$. We define 
$$\Pi_{i_0,q}\cap Q^\delta_{ij,l}:=\delta Q\cap \Big\{x_n<-\frac{\delta2^{-(q+1)}}{2}\Big\} \quad\text{and} \quad\Pi_{i_H,q}\cap Q^\delta_{ij,l}:=\delta Q\cap \Big\{x_n>\frac{\delta 2^{-(q+1)}}{2}\Big\}\,,  $$
and for $h\in\{1,\ldots, H-1\}$, 
$$\Pi_{i_h,q}\cap Q^\delta_{ij,l}:=\delta Q\cap \Big\{-\frac{\delta2^{-(q+1)}}{2}+\delta(h-1)t_q<x_n<-\frac{\delta2^{-(q+1)}}{2}+\delta ht_q\Big\}\,.$$ 
Then, we simply set 
$$\Pi_{k,q}\cap Q^\delta_{ij,l}:=\emptyset\quad\text{for $k\in\{1,\ldots,m\}\setminus\{i_h\}^H_{h=0}$}\,.$$

\noindent{\it Estimating the energy outside the cubes.} By \eqref{bougpaorcube} and condition (3), we have 
\begin{equation}\label{energorcub}
\mathscr{P}_1^{\boldsymbol{\sigma}}\Big(\Pi_q,\Omega\setminus\Big(\bigcup_{i<j}\bigcup_{l=1}^{N^\delta_{ij}}Q^\delta_{ij,l}\Big)\Big)= \mathscr{P}_1^{\boldsymbol{\sigma}}\Big(\Pi,\Omega\setminus\Big(\bigcup_{i<j}\bigcup_{l=1}^{N^\delta_{ij}}Q^\delta_{ij,l}\Big)\Big)
\leq \boldsymbol{\sigma}_{\rm max}\sum_{i<j}{\rm e}_{ij}(\delta)\,.
\end{equation}
\vskip5pt
\noindent{\it Estimating the energy inside the cubes.} Since in the above construction we obviously have 
$$\mathcal{H}^{n-1}\Big(\delta Q\cap\Big\{x_n=-\frac{\delta 2^{-(q+1)}}{2}+\delta ht_q\Big\}\Big)=\delta^{n-1}\,,$$ 
 we infer that for each $Q^\delta_{ij,l}$,
$$\mathscr{P}_1^{\boldsymbol{\sigma}}(\Pi_q,  Q^\delta_{ij,l}) \leq \delta^{n-1}\sum_{h=0}^{H-1}\sigma_{i_hi_{h+1}} +C\delta^{n-2}2^{-q}=\delta^{n-1}\bar\sigma_{ij}+C\delta^{n-1}2^{-q}\,, $$
for a constant $C$ independent of $q$ and $\delta$ (the second term in the right hand side comes from the contribution on the lateral part of the cube).   Therefore, by condition (3), 
\begin{multline*}
\mathscr{P}_1^{\boldsymbol{\sigma}}\Big(\Pi_q, \bigcup_{l=1}^{N_\delta^{ij}} Q^\delta_{ij,l}\Big)
 \leq N_\delta^{ij} \delta^{n-1}\bar\sigma_{ij} +CN_\delta^{ij}\delta^{n-2}2^{-q}\\
 \leq \bar\sigma_{ij}\mathcal{H}^{n-1} (\partial\Pi_i\cap\partial\Pi_j\cap\Omega)+\boldsymbol{\sigma}_{\rm max}{\rm e}_{ij}(\delta) + CN_\delta^{ij}\delta^{n-1}2^{-q}\,.
 \end{multline*}
 Therefore, 
\begin{equation}\label{energdanscub}
\mathscr{P}_1^{\boldsymbol{\sigma}}\Big(\Pi_q, \bigcup_{i<j}\bigcup_{l=1}^{N_\delta^{ij}} Q^\delta_{ij,l}\Big)\leq  \mathscr{P}_1^{\bar{\boldsymbol{\sigma}}}(\Pi,\Omega)
+\boldsymbol{\sigma}_{\rm max}\sum_{i<j} {\rm e}_{ij}(\delta) +  CM_\delta\delta^{n-1}2^{-q}\,.
\end{equation}
\vskip5pt
\noindent{\it Conclusion.} Gathering \eqref{energorcub} and \eqref{energdanscub}, we deduce that 
\begin{equation}\label{finalsteprelaxcoef}
\mathscr{P}_1^{\boldsymbol{\sigma}}(\Pi_q,\Omega)\leq   \mathscr{P}_1^{\bar{\boldsymbol{\sigma}}}(\Pi,\Omega) +2 \boldsymbol{\sigma}_{\rm max}\sum_{i<j} {\rm e}_{ij}(\delta)  +  CM_\delta\delta^{n-1}2^{-q}\,.
\end{equation}
On the other hand, we clearly have
$$\|\chi_{\Pi_{k,q}}-\chi_{\Pi_k}\|_{L^1(\Omega)}\leq  \sum_{i<j}\sum_{l=1}^{N^{ij}_\delta}\|\chi_{\Pi_{k,q}}-\chi_{\Pi_k}\|_{L^1(Q^\delta_{ij,l})}\leq 
M_\delta\delta^{n}2^{-q}\mathop{\longrightarrow}\limits_{q\to\infty}0\quad \text{$\forall k\in\{1,\ldots,m\}$}\,.$$
Letting $q\to\infty$ in \eqref{finalsteprelaxcoef}, we deduce that 
$$\limsup_{q\to\infty} \mathscr{P}_1^{\boldsymbol{\sigma}}(\Pi_q,\Omega)\leq   \mathscr{P}_1^{\bar{\boldsymbol{\sigma}}}(\Pi,\Omega) +2 \boldsymbol{\sigma}_{\rm max}\sum_{i<j} {\rm e}_{ij}(\delta) \,.$$
Choosing the parameter $\delta$ small enough in such a way that $\sum_{i<j} {\rm e}_{ij}(\delta)\leq \varepsilon/(2 \boldsymbol{\sigma}_{\rm max})$, the conclusion follows. 
\end{proof}

To reach the case of a general partition, we shall make use of the following approximation result by polyhedral partitions essentially taken from \cite[Lemma 3.1]{Bal}. 

\begin{lemma}[\cite{Bal}]\label{approxpolyhlem}
For every Caccioppoli partition $\mathfrak{E}=(E_1,\ldots,E_m)$ of $\Omega$, there exists a sequence $\{\Pi_h\}_{h\in\mathbb{N}}=\{(\Pi_{1,h},\ldots,\Pi_{m,h})\}_{h\in\mathbb{N}}$ in $\mathscr{B}_m(\Omega)$ such that 
\begin{itemize}
\item[(1)] $\chi_{\Pi_{i,h}}\to \chi_{E_i}$ in $L^1(\Omega)$ as $h\to\infty$ for every $i\in\{1,\ldots,m\}$;

\item[(2)] $\mathscr{P}_1^{\bar{\boldsymbol{\sigma}}}(\Pi_h,\Omega)\to \mathscr{P}_1^{\bar{\boldsymbol{\sigma}}}(\mathfrak{E},\Omega)$ as $h\to\infty$. 
\end{itemize}
\end{lemma}

\begin{proof}[Sketch of proof]
According  \cite[Proof of Lemma 3.1]{Bal}, the sequence  $\{\Pi_h\}_{h\in\mathbb{N}}$ provided by \cite[Lemma 3.1]{Bal}  belongs to $\mathscr{B}_m(\Omega)$. It satisfies condition (1),  and 
$${\rm Per}(\Pi_{i,h},\Omega)\to {\rm Per}(E_{i},\Omega) \quad\text{and}\quad {\rm Per}(\Pi_{i,h}\cup \Pi_{j,h},\Omega)\to {\rm Per}(E_{i}\cup E_j,\Omega) \quad\text{as $h\to\infty$}\,,$$
for every $i,j\in\{1,\ldots,m\}$. Since $\Pi_h$ and $\mathfrak{E}$ are partitions, by \cite[Theorem II.5.3]{Maggi} we have
$${\rm Per}(\Pi_{i,h}\cup \Pi_{j,h},\Omega)= {\rm Per}(\Pi_{i,h},\Omega)+{\rm Per}(\Pi_{j,h},\Omega)-2\mathcal{H}^{n-1}(\partial\Pi_{i,h}\cap \partial\Pi_{j,h}\cap\Omega)\,,$$
and 
$${\rm Per}(E_{i}\cup E_j,\Omega) ={\rm Per}(E_{i},\Omega)+{\rm Per}(E_{j},\Omega)-2\mathcal{H}^{n-1}(\partial^*E_i\cap\partial^*E_j\cap\Omega)\,,$$
for every $i\not=j$. Consequently, $\mathcal{H}^{n-1}(\partial\Pi_{i,h}\cap \partial\Pi_{j,h}\cap\Omega)\to \mathcal{H}^{n-1}(\partial^*E_i\cap\partial^*E_j\cap\Omega)$ as $h\to\infty$ 
for every $i\not=j$, and condition (2) follows.
\end{proof}

We finally observe that the functional $({\mathscr{P}}^{\boldsymbol{\sigma}})^*$ is lower semicontinuous with respect to the $L^1(\Omega)$-convergence by construction. This fact is absolutely standard, but we provide a proof for completeness. 

\begin{lemma}\label{lscgamlimsup}
If $\{\mathfrak{E}_h\}_{h\in\mathbb{N}}=\{(E_{1,h},\ldots,E_{m,h})\}_{h\in\mathbb{N}}$ and $\mathfrak{E}=(E_1,\ldots,E_m)$ are measurable partitions of $\Omega$ such that 
$\chi_{E_{i,h}}\to \chi_{E_i}$ in $L^1(\Omega)$ for every $i\in\{1,\ldots,m\}$ as $h\to\infty$, then 
\begin{equation}\label{lscineq}
({\mathscr{P}}^{\boldsymbol{\sigma}})^*(\mathfrak{E},\Omega)\leq \liminf_{h\to\infty} ({\mathscr{P}}^{\boldsymbol{\sigma}})^*(\mathfrak{E}_h,\Omega)\,. 
\end{equation}
\end{lemma}

\begin{proof}
We may assume that the right hand side \eqref{lscineq} is finite. Moreover, extracting a subsequence if necessary, we can assume that the $\liminf$ in \eqref{lscineq} is a limit. By definition of $({\mathscr{P}}^{\boldsymbol{\sigma}})^*(\mathfrak{E}_h,\Omega)$, for each $h$, we can a sequence $\{\mathfrak{E}^h_k\}_{k\in\mathbb{N}}=\{(E^h_{1,k},\ldots,E^h_{m,k})\}_{k\in\mathbb{N}}$ 
such that $\mathfrak{E}^h_k\in\mathscr{A}_m^{s_k}(\Omega)$, $\chi_{E^h_{i,k}}\to\chi_{E_{i,h}}$ in $L^1(\Omega)$ as $k\to\infty$ for every $i\in\{1,\ldots,m\}$, and 
$$\limsup_{k\to\infty}\,(1-2s_k)\mathscr{P}^{\boldsymbol{\sigma}}_{2s_k}(\mathfrak{E}^h_k,\Omega) \leq ({\mathscr{P}}^{\boldsymbol{\sigma}})^*(\mathfrak{E}_h,\Omega) +2^{-(h+1)}\,.$$
By induction, for every $h\in\mathbb{N}$ we can find $N_h\in\mathbb{N}$ such that $N_{h+1}\geq N_h+1$ and for which  
$$\max_{i=1,\ldots,m}\big\|\chi_{E^h_{i,k}} -\chi_{E_{i,h}}\big\|_{L^1(\Omega)}\leq 2^{-h}\quad \forall k\geq N_h\,,$$
and 
$$(1-2s_{k})\mathscr{P}^{\boldsymbol{\sigma}}_{2s_{k}}(\mathfrak{E}^h_{k},\Omega) \leq ({\mathscr{P}}^{\boldsymbol{\sigma}})^*(\mathfrak{E}_h,\Omega) +2^{-h}\quad \forall k\geq N_h\,.$$
By construction, the sequence $\{N_h\}_{h\in\mathbb{N}}$ is strictly increasing. Therefore, for every $k\in\mathbb{N}$ with $k\geq N_0$, there exists 
a unique  $h(k)\in\mathbb{N}$ such that $N_{h(k)}\leq k<N_{h(k)+1}$.  
By construction, we have $h(k)\to \infty$ as $k\to\infty$. Now we define $\mathfrak{F}_k=(F_{1,k},\ldots,F_{m,k}):=\mathfrak{E}_k^{h(k)}$ for $k\geq N_0$, and  
$\mathfrak{F}_k:=\mathfrak{E}_k^{0}$ for $k<N_0$. By our choice of $h(k)$, we then have 
$$\max_{i=1,\ldots,m}\big\|\chi_{F_{i,k}} -\chi_{E_{i,h(k)}}\big\|_{L^1(\Omega)}\leq 2^{-h(k)}\quad \forall k\geq N_0\,,$$
and 
\begin{equation}\label{diagseqlsc}
(1-2s_{k})\mathscr{P}^{\boldsymbol{\sigma}}_{2s_{k}}(\mathfrak{F}_{k},\Omega) \leq ({\mathscr{P}}^{\boldsymbol{\sigma}})^*(\mathfrak{E}_{h(k)},\Omega) +2^{-h(k)}\quad \forall k\geq N_0\,.
\end{equation}
In addition,
$$\max_{i=1,\ldots,m}\big\|\chi_{F_{i,k}} -\chi_{E_{i}}\big\|_{L^1(\Omega)}\leq 2^{-h(k)}+\max_{i=1,\ldots,m}\| \chi_{E_{i,h(k)}}-\chi_{E_{i}}\|_{L^1(\Omega)}\mathop{\longrightarrow}\limits_{k\to\infty}0\,.$$
Therefore, by the very definition of $({\mathscr{P}}^{\boldsymbol{\sigma}})^*(\mathfrak{E},\Omega)$ and \eqref{diagseqlsc}, we infer that  
$$({\mathscr{P}}^{\boldsymbol{\sigma}})^*(\mathfrak{E},\Omega)\leq \limsup_{k\to\infty} \,(1-2s_{k})\mathscr{P}^{\boldsymbol{\sigma}}_{2s_{k}}(\mathfrak{F}_{k},\Omega) 
\leq\lim_{k\to\infty} ({\mathscr{P}}^{\boldsymbol{\sigma}})^*(\mathfrak{E}_{h(k)},\Omega)=\lim_{h\to\infty} ({\mathscr{P}}^{\boldsymbol{\sigma}})^*(\mathfrak{E}_{h},\Omega)\,,$$
and the proof is complete. 
\end{proof}

We have now reached the final step of our construction for obtaining the appropriate upper bound on $(\mathscr{P}^{\boldsymbol{\sigma}})^*$. 

\begin{proposition} \label{Gamma lim supbis}
For every partition $\mathfrak{E}=(E_1,\ldots,E_m)$ of $\Omega$ by measurable sets, we have
$$(\mathscr{P}^{\boldsymbol{\sigma}})^*(\mathfrak{E},\Omega) \leq \omega_{n-1}\mathscr{P}_1^{\bar{\boldsymbol{\sigma}}}(\mathfrak{E},\Omega) \,.$$
\end{proposition}

\begin{proof}
{\it Step 1.} We claim that for $\Pi \in  \mathscr{B}_m(\Omega)$, we have 
\begin{equation}\label{relaxapproxseq}
(\mathscr{P}^{\boldsymbol{\sigma}})^*(\Pi,\Omega) \leq \omega_{n-1}\mathscr{P}_1^{\bar{\boldsymbol{\sigma}}}(\Pi,\Omega) \,.
\end{equation}
Indeed, let us fix  $\Pi \in  \mathscr{B}_m(\Omega)$ and $\varepsilon>0$ arbitrary. 
By Lemma  \ref{Gamma limsup polhyedre}, we have
$$(\mathscr{P}^{\boldsymbol{\sigma}})^*(\Pi,\Omega) \leq \omega_{n-1}\mathscr{P}_1^{\boldsymbol{\sigma}}(\Pi,\Omega) \,.$$
Then we consider the sequence $\{\Pi_q\}_{q \in \N}$ in $ \mathscr{B}_m(\Omega)$ provided by Lemma~\ref{Polyhedres approchent polyhedre} to approximate $\Pi$.  
According to Lemma \ref{lscgamlimsup}, we have 
$$
(\mathscr{P}^{\boldsymbol{\sigma}})^*(\Pi,\Omega) \leq \liminf_{q\to\infty}(\mathscr{P}^{\boldsymbol{\sigma}})^*(\Pi_q,\Omega)\leq  \omega_{n-1} \limsup_{q\to\infty} \mathscr{P}_1^{\boldsymbol{\sigma}}(\Pi_q,\Omega)\leq \omega_{n-1}\mathscr{P}_1^{\bar{\boldsymbol{\sigma}}}(\Pi,\Omega) +\varepsilon\,.
$$
Then \eqref{relaxapproxseq} follows letting $\varepsilon\to 0$. 
\vskip5pt

\noindent{\it Step 2.} Now we consider an arbitrary partition $\mathfrak{E}$ of $\Omega$. If $\mathscr{P}_1^{\bar{\boldsymbol{\sigma}}}(\mathfrak{E},\Omega)=+\infty$, there is nothing to prove. Otherwise $\mathfrak{E}$ is a Caccioppoli partition. In this case, we consider the sequence $\{\Pi_h\}_{h \in \N}$ in $ \mathscr{B}_m(\Omega)$ provided by Lemma \ref{approxpolyhlem} to approximate $\mathfrak{E}$.  
According to Lemma \ref{lscgamlimsup} again and \eqref{relaxapproxseq}, we have
$$(\mathscr{P}^{\boldsymbol{\sigma}})^*(\mathfrak{E},\Omega) \leq \liminf_{h\to\infty}(\mathscr{P}^{\boldsymbol{\sigma}})^*(\Pi_h,\Omega)\leq  \omega_{n-1} \lim_{h\to\infty} \mathscr{P}_1^{\bar{\boldsymbol{\sigma}}}(\Pi_h,\Omega)= \omega_{n-1}\mathscr{P}_1^{\bar{\boldsymbol{\sigma}}}(\mathfrak{E},\Omega) \,,$$
which proves the proposition.
\end{proof}

\section{A further lower estimate on the $\Gamma$-$\liminf$}\label{geomconstcube}

In section our goal is to obtain a refined lower bound on the $\Gamma-\liminf$ functional $({\mathscr{P}}^{\boldsymbol{\sigma}})_*(\cdot,\Omega)$ compare to  
$\mathscr{P}_1^{\boldsymbol{\Gamma}}(\cdot,\Omega)$. More precisely, we seek for a new matrix $\bar{\boldsymbol{\Gamma}}\in\mathscr{S}_m$ below 
$\boldsymbol{\Gamma}$ (componentwise) such that $\mathscr{P}_1^{\bar{\boldsymbol{\Gamma}}}(\cdot,\Omega)$ still provides a meaningful lower bound of 
$({\mathscr{P}}^{\boldsymbol{\sigma}})_*(\cdot,\Omega)$, and such that its coefficients are given more explicitly  (in a useful way).  This matrix $\bar{\boldsymbol{\Gamma}}$ 
turns out to be given by $\bar{{\Gamma}}_{ii}=0$, and for $i\not=j$, 
\begin{multline*}
\overline\Gamma_{ij}:=\liminf_{s\to1/2^-}\,\inf \Big\{ (1-2s)\mathscr{P}^{\bar{\boldsymbol{\sigma}}}_{2s}(\mathfrak{E},Q): 
\mathfrak{E}=(E_{1},\ldots,E_m)\in\mathscr{A}_m^{2s}(Q)\,,\\
E_i\cap Q^c=H\cap Q^c\,,\;E_j\cap Q^c=H^c\cap Q^c\,,\text{ and } E_l\cap Q^c=\emptyset \text{ for }l\not\in\{i,j\}\Big\}\,.
\end{multline*}

The following proposition is the main result of this section. It ensures that $\mathscr{P}_1^{\bar{\boldsymbol{\Gamma}}}(\cdot,\Omega)$ indeed provides a suitable lower bound. 

\begin{proposition}\label{propgammainf}
For every $i\not=j$, we have 
$$\Gamma_{ij}\geq \overline\Gamma_{ij}\,. $$
In particular, 
$$({\mathscr{P}}^{\boldsymbol{\sigma}})_*(\mathfrak{E},\Omega)\geq  \mathscr{P}_1^{\overline{\boldsymbol{\Gamma}}}(\mathfrak{E},\Omega)$$
for every partition $\mathfrak{E}$ of $\Omega$ by measurable sets. 
\end{proposition}

The proof of Proposition \ref{propgammainf} requires two approximation lemmas to be able to pass from $\boldsymbol{\Gamma}$ to $\bar{\boldsymbol{\Gamma}}$. 
In Lemma \ref{lemmodif1} below, we first pass to $\boldsymbol{\Gamma}$ to some intermediate matrix $\boldsymbol{\Gamma}^\flat$. In the rest of the section, we shall 
denote for $r_2>r_1>0$, 
 $$A_{r_1,r_2}=r_2Q\setminus r_1 Q\,.$$
The proof of Proposition \ref{propgammainf} is postponed at the end of the section. 

\begin{lemma}\label{lemmodif1}
For every $i\not=j$, we have 
$$\Gamma_{ij}=\Gamma^\flat_{ij}\,,$$
where
$$\Gamma^\flat_{ij}:=\inf \bigg\{ \liminf  \limits_{k \rightarrow \infty}  (1 -2s_k) \mathscr{E}_{2s_k}^{{\boldsymbol{\sigma}}}(\mathfrak{E}_{k},Q) : \big(\{s_k\}_{k\in\mathbb{N}},\{\mathfrak{E}_k\}_{k\in\mathbb{N}}\big)\in \mathscr{C}^\flat_{ij}\Big\}\,,$$
and
\begin{multline*}
 \mathscr{C}^\flat_{ij}:=\Big\{ \big(\{s_k\}_{k\in\mathbb{N}},\{\mathfrak{E}_k\}_{k\in\mathbb{N}}\big) : s_k\to1/2^-\,,\;\mathfrak{E}_k=(E_{1,k},\ldots,E_{m,k})\in\mathscr{A}_m^{s_k}(Q)\,,\\
 \chi_{E_{i,k}}\to \chi_H\,,\;\chi_{E_{j,k}}\to \chi_{H^c}\,,\; \chi_{E_{l,k}}\to 0 \text{ for $l\not\in\{i,j\}$ in $L^1(Q)$}\,,\\
 \text{there exists $\delta>0$ such that } E_{l,k}\cap A_{1-\delta,1}=\emptyset \text{ for $l\not\in\{i,j\}$ and all $k\in\mathbb{N}$}\Big\}\,.
 \end{multline*}
\end{lemma}

\begin{proof}
\noindent{\it Step 1.} Since the class of competitors for $\Gamma_{ij}$ is larger than $\mathscr{C}^\flat_{ij}$, we obviously have $\Gamma_{ij}\leq\Gamma^\flat_{ij}$, and we only need to prove the reverse inequality. 
To this purpose, we fix $\varepsilon>0$ arbitrary and we consider a competitor $\big(\{s_k\}_{k\in\mathbb{N}},\{\mathfrak{E}_k\}_{k\in\mathbb{N}}\big)$ for $\Gamma_{ij}$ such that 
\begin{equation} \label{outil72bisfirst}
\liminf \limits_{k \rightarrow \infty}\, (1-2s_k) \mathscr{E}^{\boldsymbol{\sigma}}_{2s_k}(\mathfrak{E}_{k},Q) \leq \Gamma_{ij} + \frac{\varepsilon}{2} \, .
\end{equation}
Extracting a subsequence if necessary, we may assume that the $\liminf$ above is actually a limit.
\vskip3pt

 We claim that, whenever $\delta> 0$ is small enough, we have 
\begin{equation} \label{outil71bis}
 \liminf \limits_{k \rightarrow \infty} (1-2s_k) (1-2s_k) \mathscr{E}^{\boldsymbol{\sigma}}_{2s_k}(\mathfrak{E}_k,A_{1-\delta,1})  \leq \varepsilon \,.
\end{equation}
In fact, it is enough to have $\delta>0$ in such a way that 
\begin{equation} \label{outil70bis}
\big(1 - (1-\delta)^{n-1}\big) \Gamma_{ij} \leq \frac{\varepsilon}{2} \,.
\end{equation}
Indeed, consider the partition $ \mathfrak{F}_k = (F_{1,k},...,F_{m,k}) $ defined by
$$F_{l,k} = \frac{1}{(1-\delta)}\,E_{l,k} \quad\text{for $l\in\{1,\ldots,m\}$}\,.$$
Clearly  the sequence $\{\mathfrak{F}_k\}_{k\in\mathbb{N}}$ is admissible for  $\Gamma_{ij}$ (since $H$ is invariant under dilatations), and 
by a change of variables we have
$$\mathscr{E}^{\boldsymbol{\sigma}}_{2s_k}(\mathfrak{F}_{k},Q) = (1-\delta)^{2s_k - n} \mathscr{E}^{\boldsymbol{\sigma}}_{2s_k}(\mathfrak{E}_{k},Q \backslash A_{1-\delta,1}) \,.$$
Since $\mathscr{E}^{\boldsymbol{\sigma}}_{2s_k}(\mathfrak{E}_{k},\cdot)$ is superadditive, we have 
\begin{multline*} 
\lim\limits_{k \rightarrow \infty} (1-2s_k)\mathscr{E}^{\boldsymbol{\sigma}}_{2s_k} (\mathfrak{E}_{k},Q) \\
\geq \liminf \limits_{k \rightarrow \infty} (1-2s_k)\mathscr{E}^{\boldsymbol{\sigma}}_{2s_k}(\mathfrak{E}_{k},Q \backslash A_{1-\delta,1}) 
+ \liminf \limits_{k \rightarrow \infty} (1-2s_k) \mathscr{E}^{\boldsymbol{\sigma}}_{2s_k} (\mathfrak{E}_{k},A_{1-\delta,1})\,,
\end{multline*}
which leads to 
\begin{multline*}
\Gamma_{ij} + \frac{\varepsilon}{2}  \geq (1-\delta)^{n-1}\liminf \limits_{k \rightarrow \infty} (1-2s_k) \mathscr{E}^{\boldsymbol{\sigma}}_{2s_k} (\mathfrak{F}_{k},Q) + \liminf \limits_{k \rightarrow \infty} (1-2s_k) \mathscr{E}^{\boldsymbol{\sigma}}_{2s_k} (\mathfrak{E}_{k},A_{1-\delta,1})\\
\geq (1-\delta)^{n-1}\Gamma_{ij}+ \liminf \limits_{k \rightarrow \infty} (1-2s_k) \mathscr{E}^{\boldsymbol{\sigma}}_{2s_k} (\mathfrak{E}_{k},A_{1-\delta,1})\,.
\end{multline*}
Therefore, \eqref{outil71bis} follows from condition \eqref{outil70bis}. 
\vskip5pt

 \noindent{\it Step 2.} Now we fix $\delta_2 > \delta_1 > 0$ such that $\delta_2$ satisfies \eqref{outil71bis}, and we set 
 $\delta_3 := \frac{\delta_1 + \delta_2}{2}$. Then we extract a further (not relabeled) subsequence achieving the $\liminf$ in  \eqref{outil71bis}, so that 
 \begin{equation}\label{condepsFbis}
  \lim \limits_{k \rightarrow \infty}(1-2s_k) \mathscr{E}^{\boldsymbol{\sigma}}_{2s_k}(\mathfrak{E}_k,A_{1-\delta_2,1})   \leq \varepsilon \,.
  \end{equation}
 For $t \in ((1-\delta_3)/2,(1-\delta_1)/2)$ (to be chosen depending on $k$), we define
 $$L^+_t:=Q\cap\{x_1>t\}\quad\text{and} \quad L^-_t:=Q\cap\{x_1<t\}\,.$$
We define again a new  of partition $\mathfrak{F}_{k}(t)=(F^t_{1,k},\ldots,F^t_{m,k})\in\mathscr{A}_m^{s_k}(Q)$ depending on $t$ as follows: 
$$ \begin{cases}
           F^t_{l,k} \cap (L_t^-\cup Q^c) = E_{l,k} \cap (L^-_t \cup Q^c)\quad \text{ for  $l\in\{1,\ldots,m\}$}\,,  \\[5pt]
           F^t_{j,k} \cap L_t^+ = E_{j,k} \cap L_t^+ \,,\\[5pt]
           F^t_{i,k} \cap L^+_t = \bigcup \limits_{l \neq j}E_{l,k} \cap L^+_t \,,\\[5pt]
           F^t_{l,k} \cap L^+_t = \emptyset \quad \text{ for $l\not\in\{i,j\}$}\,. 
    \end{cases}
$$
(Note that $\mathfrak{F}_{k}(t)$ belongs to $\mathscr{A}_m^{s_k}(Q)$ by the computations below, and since $\mathcal{I}_{s_k}(Q,Q^c)<\infty$.) 
We shall compare the values of $\mathscr{E}^{\boldsymbol{\sigma}}_{2s_k}(\mathfrak{F}_{k}(t),Q)$ and $\mathscr{E}^{\boldsymbol{\sigma}}_{2s_k}(\mathfrak{E}_{k},Q)$.  To this purpose, we set 
$$ B_k := \bigcup \limits_{l \not\in\{i,j\}} E_{l,k}\cap Q \quad \text{and}\quad  B_k^t = B_k \cap L^+_t \,.$$
We first notice that 
\begin{equation}\label{decompFk}
F^t_{l,k}\cap B_k^t=\begin{cases}
\emptyset & \text{for $l\not =i$}\,,\\
B_k^t & \text{for $l=i$}\,,
\end{cases}
\;\text{and}\;
(F^t_{l,k}\setminus B_k^t)\cap Q=
\begin{cases}
E_{l,k}\cap L_t^-   & \text{for $l\not\in\{i,j\}$}\,,\\
E_{l,k}\cap Q & \text{for $l\in\{i,j\}$}\,.
\end{cases}
\end{equation}
and 
\begin{equation}\label{decompEk}
E_{l,k}\cap B_k^t=\begin{cases}
E_{l,k}\cap L^+_t & \text{for $l\not\in\{i,j\}$}\,,\\
\emptyset & \text{for $l\in\{i,j\}$}\,,
\end{cases}
\;\text{and}\;
(E_{l,k}\setminus B_k^t)\cap Q=
\begin{cases}
E_{l,k}\cap L_t^-   & \text{for $l\not\in\{i,j\}$}\,,\\
E_{l,k}\cap Q & \text{for $l\in\{i,j\}$}\,.
\end{cases}
\end{equation}
In particular, $F^t_{l,k}\triangle E_{l,k}\subset B_k^t$ for every $l\in\{1,\ldots,m\}$. Let us write 
$$\boldsymbol{\Delta}_k(t):=\mathscr{E}^{\boldsymbol{\sigma}}_{2s_k}(\mathfrak{F}_{k}(t),Q)-\mathscr{E}^{\boldsymbol{\sigma}}_{2s_k}(\mathfrak{E}_{k},Q)=\frac{1}{2}\sum_{l\not=h}\sigma_{lh}\Delta^k_{lh}(t)\,,$$
with 
$$\Delta^k_{lh}(t):=\mathcal{I}_{2s_k}(F^t_{l,k}\cap Q,F^t_{h,k}\cap Q)- \mathcal{I}_{2s_k}(E_{l,k}\cap Q,E_{h,k}\cap Q)\,.$$
Then we have for $l\not=h$, 
\begin{align*}
\Delta^k_{lh}(t)=& \big[\mathcal{I}_{2s_k}(F^t_{l,k}\cap B_k^t,F^t_{h,k}\cap B_k^t)- \mathcal{I}_{2s_k}(E_{l,k}\cap B_k^t,E_{h,k}\cap B_k^t)\big]\\
&+ \big[\mathcal{I}_{2s_k}(F^t_{l,k}\cap B_k^t,(F^t_{h,k}\setminus B_k^t)\cap Q)- \mathcal{I}_{2s_k}(E_{l,k}\cap B_k^t,(E_{h,k}\setminus B_k^t)\cap Q)\big]\\
&+  \big[\mathcal{I}_{2s_k}((F^t_{l,k}\setminus B_k^t)\cap Q,F^t_{h,k}\cap B_k^t)- \mathcal{I}_{2s_k}((E_{l,k}\setminus B_k^t)\cap Q,E_{h,k}\cap B_k^t)\big]\,.
\end{align*}
Using \eqref{decompFk} and \eqref{decompEk}, we infer that for $l\not\in\{i,j\}$,
$$\Delta^k_{lh}(t)\leq 
\begin{cases}
0 & \text{for $h\not\in\{i,l\}$}\,,\\[5pt]
 \mathcal{I}_{2s_k}(E_{l,k}\cap L^-_t,B_k^t) & \text{for $h=i$}\,,
\end{cases}$$
as well as 
$$\Delta^k_{ih}(t)\leq 
\begin{cases}
 \mathcal{I}_{2s_k}(B_k^t,E_{h,k}\cap L_t^-)  & \text{for $h\not\in\{i,j\}$}\,,\\[5pt]
 \mathcal{I}_{2s_k}(B_k^t,E_{j,k}\cap Q) & \text{for $h=j$}\,,
\end{cases}$$
and 
$$\Delta^k_{jh}(t)\leq 
\begin{cases}
0 & \text{for $h\not\in\{i,j\}$}\,,\\[5pt]
\mathcal{I}_{2s_k}(E_{j,k}\cap Q, B_k^t) & \text{for $h=i$}\,.
\end{cases}$$
Therefore, by symmetry of $\boldsymbol{\sigma}$, 
\begin{multline}\label{ineq1deltak}
\boldsymbol{\Delta}_k(t)\leq \sigma_{ij}\,\mathcal{I}_{2s_k}(E_{j,k}\cap Q, B_k^t) +\sum_{l\not\in\{i,j\}}\sigma_{li} \,\mathcal{I}_{2s_k}(E_{l,k}\cap L^-_t,B_k^t)\\
\leq \boldsymbol{\sigma}_{\rm max}\,\mathcal{I}_{2s_k}(E_{j,k}\cap Q, B_k^t) + \boldsymbol{\sigma}_{\rm max}\,\mathcal{I}_{2s_k}(B_k\cap L_t^-,B_k^t)\,.
\end{multline}
We shall now estimate separately the two terms in the right hand side above. 

Concerning the first term, we write it as 
\begin{multline}\label{mercr1}
\mathcal{I}_{2s_k}(E_{j,k}\cap Q, B_k^t)=\mathcal{I}_{2s_k}(E_{j,k}\cap Q\cap \{x_1<(1-\delta_2)/2\}, B_k^t)\\
+\mathcal{I}_{2s_k}(E_{j,k}\cap Q\cap \{(1-\delta_2)/2<x_1\}, B_k^t)\,.
\end{multline}
Then,
\begin{multline}\label{mercr2}
\mathcal{I}_{2s_k}(E_{j,k}\cap Q\cap \{x_1<1-\delta_2\}, B_k^t)\\
\leq \mathcal{I}_{2s_k}(Q\cap \{x_1<(1-\delta_2)/2\},Q\cap\{x_1>(1-\delta_3)/2\})
\leq \frac{4^{n+2s}}{(\delta_2-\delta_1)^{n+2s}}\,.
\end{multline}
On the other hand,
\begin{multline}\label{mercr3}
\mathcal{I}_{2s_k}(E_{j,k}\cap Q\cap \{(1-\delta_2)/2<x_1\}, B_k^t)\\
\leq \frac{1}{\boldsymbol{\sigma}_{\rm min}}\sum_{l\notin\{i,j\}}\sigma_{jl}\,\mathcal{I}_{2s_k}(E_{j,k}\cap A_{1-\delta_2,1}, 
E_{l,k}\cap A_{1-\delta_2,1})
\leq  \frac{1}{\boldsymbol{\sigma}_{\rm min}} \mathscr{E}^{\boldsymbol{\sigma}}_{2s_k}(\mathfrak{E}_{k},A_{1-\delta_2,1})\,.
\end{multline}
Combining  \eqref{mercr1}, \eqref{mercr2}, and \eqref{mercr3} leads to 
\begin{equation}\label{olalacpa}
\boldsymbol{\sigma}_{\rm max}\,\mathcal{I}_{2s_k}(E_{j,k}\cap Q, B_k^t) \leq \frac{\boldsymbol{\sigma}_{\rm max}}{\boldsymbol{\sigma}_{\rm min}} \mathscr{E}^{\boldsymbol{\sigma}}_{2s_k}(\mathfrak{E}_{k},A_{1-\delta_2,1})+\frac{\boldsymbol{\sigma}_{\rm max}\,4^{n+2s}}{(\delta_2-\delta_1)^{n+2s}}\,.
\end{equation}

To estimate the term 
$$J_k(t):=\mathcal{I}_{2s_k}(B_k\cap L_t^-,B_k^t)\,,$$ 
we shall rely on an average argument to select a suitable value of $t$ (depending on $k$). We write points $x\in\R^n$ 
as $x=(x_1,x^\prime)$, use Fubini's Theorem and a change of variables to get
\begin{multline*}
J_k(t)
\leq  \int_{B^t_k} \left[ \int_{-\frac{1}{2}}^t \int_{\R^{n-1}} \frac{1}{(|x_1 - y_1|^2 + |x' - y'|^2)^{\frac{n+2s_k}{2}}} \,\de y' \de y_1 \right] \,\de x \\
 = \Lambda_k\int_{B^t_k} \left[ \int_{-\frac{1}{2}}^t \frac{1}{(x_1 - y_1)^{1+2s_k}}\, 
\de y_1 \right] \,\de x \,,
\end{multline*}
where
$$ \Lambda_{k} := \int_{\R^{n-1}} \frac{\de z}{(1+|z|^2)^{\frac{n+2s_k}{2}}} \leq C\,, $$
for a constant $C$ independent of  $k$. Therefore,
\begin{equation}\label{estJk}
J_k(t)\leq \frac{C}{2s_k} \int_{B^t_k}\frac{1}{(x_1-t)^{2s_k}}\,\de x\,.
\end{equation}
In turn, by the co-area formula and the fact that $B_k\subset Q$, 
$$\int_{B^t_k}\frac{1}{(x_1-t)^{2s_k}}\,\de x=\int_t^{1/2}\frac{\mathcal{H}^{n-1}(B_k\cap\{x_1=\rho\})}{(\rho-t)^{2s_k}} \,\de \rho=\int_0^{1/2}\frac{\mathcal{H}^{n-1}(B_k\cap\{x_1=r+t\})}{r^{2s_k}} \,\de r\,.$$ 
Integrating this identity over $t\in ((1-\delta_3)/2,(1-\delta_1)/2)$ and using Fubini's Theorem again, we obtain
$$\int_{\frac{1-\delta_3}{2}}^{\frac{1-\delta_2}{2}}\int_{B^t_k}\frac{1}{(x_1-t)^{2s_k}}\,\de x\,\de t
=\int_{0}^{\frac{1}{2}} \frac{1}{r^{2s_k}}\Big(\int_{r+\frac{1-\delta_3}{2}}^{r+\frac{1-\delta_2}{2}}\mathcal{H}^{n-1}(B_k\cap\{x_1=\rho\})\,\de \rho\Big)\,\de r\leq \frac{|B_k|}{(1-2s_k)}\,.$$
As a consequence, we can find some $t_k\in((1-\delta_3)/2,(1-\delta_1)/2)$ such that 
$$ \int_{B^{t_k}_k}\frac{1}{(x_1-t_k)^{2s_k}}\,\de x\leq \frac{4|B_k|}{(1-2s_k)(\delta_2-\delta_1)}\,.$$
Inserting this estimate in \eqref{estJk}, we obtain
\begin{equation}\label{olalacpa2}
J_k(t_k)\leq  \frac{C|B_k|}{s_k(1-2s_k)(\delta_2-\delta_1)}\,.
\end{equation}
Finally, gathering \eqref{ineq1deltak},  \eqref{olalacpa} and \eqref{olalacpa2} yields 
\begin{equation}\label{olalacpa3}
\boldsymbol{\Delta}_k(t_k)\leq \frac{\boldsymbol{\sigma}_{\rm max}}{\boldsymbol{\sigma}_{\rm min}} \mathscr{E}^{\boldsymbol{\sigma}}_{2s_k}(\mathfrak{E}_{k},A_{1-\delta_2,1})+\frac{\boldsymbol{\sigma}_{\rm max}\,4^{n+2s}}{(\delta_2-\delta_1)^{n+2s}} +  \frac{C\boldsymbol{\sigma}_{\rm max}|B_k|}{s_k(1-2s_k)(\delta_2-\delta_1)}\,,
\end{equation}
\vskip5pt

We now set $\mathfrak{F}_k=(F_{1,k},\ldots,F_{m,k}):=\mathfrak{F}_k(t_k)$, and we observe that the sequence $\{\mathfrak{F}_k\}_{k\in\mathbb{N}}$ is an admissible competitor for $\Gamma_{ij}$. Indeed,  
our choice of $\{\mathfrak{E}_k\}_{k\in\mathbb{N}}$ insures that 
\begin{equation}\label{olalacpa4}
|B_k|=\sum_{l\not\in\{i,j\}}|E_{l,k}| \mathop{\longrightarrow}\limits_{k\to\infty}0\,.
\end{equation}
Therefore,  $|F_{l,k}\cap Q|\leq |E_{l,k}|\to 0$ as $k\to\infty$ for $l\not\in\{i,j\}$, $\chi_{F_{j,k}}=\chi_{E_{j,k}}\to \chi_{H^c}$ in $L^1(Q)$ as $k\to\infty$, and  
$$\|\chi_{F_{i,k}}-\chi_{H}\|_{L^1(Q)}\leq  \|\chi_{F_{i,k}}-\chi_{E_{i,k}}\|_{L^1(Q)}+\|\chi_{E_{i,k}}-\chi_{H}\|_{L^1(Q)}\leq |B_k|+\|\chi_{E_{i,k}}-\chi_{H}\|_{L^1(Q)} \mathop{\longrightarrow}\limits_{k\to\infty}0\,.$$
On the other hand, we deduce from \eqref{outil72bisfirst}, \eqref{condepsFbis}, \eqref{olalacpa3} and \eqref{olalacpa4} that 
\begin{multline*}
\limsup_{k\to\infty}\,(1-2s_k)\mathscr{E}^{\boldsymbol{\sigma}}_{2s_k}(\mathfrak{F}_{k},Q)\\
=\lim_{k\to\infty}\,(1-2s_k)\mathscr{E}^{\boldsymbol{\sigma}}_{2s_k}(\mathfrak{E}_{k},Q) +\limsup_{k\to\infty}\,(1-2s_k)\boldsymbol{\Delta}_k(t_k)
\leq  \Gamma_{ij}+C\varepsilon\,,
\end{multline*}
for a constant $C$ depending only on $\boldsymbol{\sigma}$. 
By construction $\{\mathfrak{F}_k\}_{k\in\mathbb{N}}$ satisfies the property that $F_{l,k}\cap Q\cap \{x_1>(1-\delta_1)/2\}=\emptyset$ for $l\not\in\{ij\}$ and every $k\in\mathbb{N}$. 
\vskip5pt

\noindent{\it Step 3.} Starting from the sequence $\{\mathfrak{F}_k\}_{k\in\mathbb{N}}$, we repeat successively the construction in Step 1 and Step 2 to modify $\{\mathfrak{F}_k\}_{k\in\mathbb{N}}$ near the remaining faces of $Q$. It produces a new (not relabeled) subsequence $(\{s_k\}_{k\in\mathbb{N}},\{\mathfrak{G}_k\}_{k\in\mathbb{N}})\in  \mathscr{C}^\flat_{ij}$ satisfying 
$$\Gamma^\flat_{ij}\leq \liminf_{k\to\infty}\,(1-2s_k)\mathscr{E}^{\boldsymbol{\sigma}}_{2s_k}(\mathfrak{G}_{k},Q)\leq\limsup_{k\to\infty}\,(1-2s_k)\mathscr{E}^{\boldsymbol{\sigma}}_{2s_k}(\mathfrak{G}_{k},Q)\leq   \Gamma_{ij}+C\varepsilon\,,$$
for a constant $C$ depending only on $n$ and $\boldsymbol{\sigma}$. By arbitrariness of $\varepsilon>0$, we conclude that $\Gamma^\flat_{ij}\leq \Gamma_{ij}$, and the proof is complete. 
\end{proof}

In the following lemma, we shall make use of the \cite[Proof of Proposition 11]{ADPM}. It allows us to pass from $\boldsymbol{\Gamma}^\flat$ (hence from $\boldsymbol{\Gamma}$) to some further intermediate but last matrix $\boldsymbol{\Gamma}^\sharp$.  

\begin{lemma}\label{lemmasharp}
For every $i\not=j$, we have 
$$\Gamma_{ij}=\Gamma^\sharp_{ij}\,,$$
where
$$\Gamma^\sharp_{ij}:=\inf \bigg\{ \liminf  \limits_{k \rightarrow \infty}  (1 -2s_k) \mathscr{E}_{2s_k}^{{\boldsymbol{\sigma}}}(\mathfrak{E}_{k},Q) : \big(\{s_k\}_{k\in\mathbb{N}},\{\mathfrak{E}_k\}_{k\in\mathbb{N}}\big)\in \mathscr{C}^\sharp_{ij}\Big\}\,,$$
and
\begin{multline*}
 \mathscr{C}^\sharp_{ij}:=\Big\{ \big(\{s_k\}_{k\in\mathbb{N}},\{\mathfrak{E}_k\}_{k\in\mathbb{N}}\big)\in  \mathscr{C}^\flat_{ij} :  \text{there exists $\delta>0$ such that } \\
 E_{i,k}\cap A_{1-\delta,1}=H\cap  A_{1-\delta,1} \text{ and }  E_{j,k}\cap A_{1-\delta,1}=H^c\cap  A_{1-\delta,1}  \text{ for all $k\in\mathbb{N}$}\Big\}\,.
 \end{multline*}
\end{lemma}

\begin{proof}
By Lemma \ref{lemmodif1}, we have $\Gamma_{ij}=\Gamma^\flat_{ij}$. Hence we have to prove that $\Gamma^\sharp_{ij}=\Gamma^\flat_{ij}$. Noticing that  
$ \mathscr{C}^\sharp_{ij}\subset  \mathscr{C}^\flat_{ij}$, we have $\Gamma^\sharp_{ij}\geq \Gamma^\flat_{ij}$, and we only have to prove the reverse inequality. 
Once again, we fix $\varepsilon>0$ arbitrary and we consider a competitor $\big(\{s_k\}_{k\in\mathbb{N}},\{\mathfrak{E}_k\}_{k\in\mathbb{N}}\big)\in  \mathscr{C}^\flat_{ij} $  such that 
\begin{equation} \label{outil72bis}
\liminf \limits_{k \rightarrow \infty} (1-2s_k) \mathscr{E}^{\boldsymbol{\sigma}}_{2s_k}(\mathfrak{E}_{k},Q) \leq \Gamma^\flat_{ij} + \frac{\varepsilon}{2} \,,
\end{equation}
and we may assume that the $\liminf$ above is actually a limit. By definition of $\mathscr{C}^\flat_{ij} $, there exists $\delta\in(0,1/4)$ such that $E_{l,k}\cap A_{1-\delta,1}=\emptyset$ for $l\not\in\{i,j\}$ and all $k\in\mathbb{N}$. 

Using that $\Gamma_{ij}=\Gamma^\flat_{ij}$,  we can repeat the argument from Step 1 in the proof of Lemma \ref{lemmodif1} to show that for $\bar\delta\in(0,\delta)$ small enough, then 
\begin{equation}\label{seclaycond}
\liminf \limits_{k \rightarrow \infty} (1-2s_k) \mathscr{E}^{\boldsymbol{\sigma}}_{2s_k}(\mathfrak{E}_{k},A_{1-\bar\delta,1}) \leq \varepsilon\,.
\end{equation}
We fix such a $\bar\delta$ arbitrarily small, and then we extract a subsequence  achieving the $\liminf$ above.  

Next we set 
$$\delta_0:=\frac{3\bar\delta}{4}\,,\quad \delta_1:=\frac{\bar\delta}{2}\,,\quad \delta_2:=\frac{\bar\delta}{4}\,,\quad\text{and}\quad\Omega_{\bar\delta}:=Q\setminus (1-\delta_0)\overline Q\,.$$
We observe that by symmetry of $\boldsymbol{\sigma}$ and the fact that $E_{l,k}\cap \Omega_{\bar\delta}=\emptyset$ for $l\not\in\{i,j\}$, 
\begin{equation}\label{reducidentcst}
\mathscr{E}^{\boldsymbol{\sigma}}_{2s_k}(\mathfrak{E}_{k},\Omega_{\bar\delta})= 
\sigma_{ij}\,\mathcal{I}_{2s_k}(E_{i,k}\cap\Omega_{\bar\delta},E_{j,k}\cap\Omega_{\bar\delta})=\sigma_{ij}\,\mathcal{I}_{2s_k}(E_{i,k}\cap\Omega_{\bar\delta},E^c_{i,k}\cap\Omega_{\bar\delta})\,.
\end{equation}
Following the notation of \cite{ADPM} to ease the reader, we write for measurable sets $E$ and $O$, 
$$\mathcal{J}^1_{2s_k}(E,O):= \mathcal{I}_{2s_k}(E\cap O,E^c\cap O)\,.$$
Then we reproduce verbatim the proof  \cite[Proposition 11]{ADPM} in the open set $\Omega_{\bar\delta}$ with the sets $E_{i,k}$ and $H$ (as well as $\delta_1$ and $\delta_2$). The only exception is that we use a smooth cut-off function $\varphi\in C^\infty(Q)$ compactly supported in $(1-\delta_2)Q$ and 
satisfying $\varphi=1$ in $(1-\delta_1)Q$ (instead of a cut-off function compactly supported in $\Omega_{\bar \delta}$). It produces a set $\widetilde F_{i,k}\subset \Omega_{\bar\delta}$ 
which satisfies 
\begin{itemize}
\item[(a)] $\|\chi_{\widetilde F_{i,k}} -\chi_H\|_{L^1(\Omega_{\bar \delta})}\leq \|\chi_{E_{i,k}} -\chi_H\|_{L^1(\Omega_{\bar \delta})}$;
\item[(b)]  $\widetilde F_{i,k}\cap A_{1-\delta_0,1-\delta_1}=E_{i,k}\cap A_{1-\delta_0,1-\delta_1}$ and $\widetilde F_{i,k}\cap A_{1-\delta_2,1}=H\cap A_{1-\delta_2,1}$; 
\item[(c)] the estimate 
\begin{multline}\label{estiADPM}
\mathcal{J}^1_{2s_k}(\widetilde F_{i,k},\Omega_{\bar \delta})\leq \mathcal{J}^1_{2s_k}(E_{i,k},\Omega_{\bar \delta})+\mathcal{J}^1_{2s_k}(H,A_{1-\delta_1-\varepsilon,1})
+\frac{C}{\varepsilon^{n+2s_k}}\\
+C_{\bar\delta}\Big[\frac{\|\chi_{E_{i,k}}-\chi_{H}\|_{L^1(A_{1-\delta_1,1-\delta_2})}}{(1-2s_k)}+ \|\chi_{E_{i,k}}-\chi_{H}\|_{L^1(\Omega_{\bar\delta})}\Big]\,,
\end{multline}
for a constant $C$ depending only on $n$, and a constant $C_{\bar\delta}$ depending on $n$ and $\bar\delta$. 
\end{itemize}
\vskip5pt

We are now able to define a competitor $\mathfrak{F}_k=(F_{1,k},\ldots,F_{m,k})\in\mathscr{A}_m^{2s_k}(Q)$ by setting  
$$F_{l,k}\cap(1-\delta_0)Q :=E_{l,k}\cap (1-\delta_0)Q \quad\text{for every $l\in\{1,\ldots,m\}$}\,, $$
and
$$F_{l,k}\cap A_{1-\delta_0,1}:= \begin{cases}
\emptyset & \text{for $l\not\in\{i,j\}$}\,,\\
\widetilde F_{i,k} & \text{for $l=i$}\,,\\
(\widetilde F_{i,k})^c\cap A_{1-\delta_0,1}  & \text{for $l=j$}\,,
\end{cases}$$ 
as well as $F_{l,k}\setminus Q=\emptyset$ for $l\not\in\{i,j\}$, $F_{i,k}\setminus Q=H\setminus Q$, and $F_{j,k}\setminus Q=H^c\setminus Q$. By conditions (a) and (b) above, the sequence 
$\{\mathfrak{F}_k\}_{k\in\mathbb{N}}$ satisfies the requirements of the class $\mathscr{C}^\sharp_{ij}$. 

We now observe that 
$$\mathscr{E}^{\boldsymbol{\sigma}}_{2s_k}(\mathfrak{F}_{k},Q)=\mathscr{E}^{\boldsymbol{\sigma}}_{2s_k}(\mathfrak{E}_{k},(1-\delta_0)Q)
+\mathscr{E}^{\boldsymbol{\sigma}}_{2s_k}(\mathfrak{F}_{k},\Omega_{\bar\delta})+\mathscr{R}_k\,, $$
where  
$$\mathscr{R}_k:= \mathscr{R}^1_k+\mathscr{R}^2_k\,,$$
with 
$$  \mathscr{R}^1_k:=\sigma_{ij}\Big(\mathcal{I}_{2s_k}(F_{i,k}\cap A_{1-\delta_0,1}, E_{j,k}\cap A_{1-\bar\delta,1-\delta_0})+\mathcal{I}_{2s_k}(F_{j,k}\cap A_{1-\delta_0,1}, E_{i,k}\cap A_{1-\bar\delta,1-\delta_0})\Big)\,,$$
and 
\begin{multline*}
\mathscr{R}^2_k:=\sum_{l=1}^m\sigma_{il} \mathcal{I}_{2s_k}\big(F_{i,k}\cap A_{1-\delta_0,1}, E_{l,k}\cap (1-\bar\delta)Q\big)\\
+\sum_{l=1}^m\sigma_{jl} \mathcal{I}_{2s_k}\big(F_{j,k}\cap A_{1-\delta_0,1}, E_{l,k}\cap (1-\bar\delta)Q\big)\,.
\end{multline*}
Exactly as in \eqref{reducidentcst}, we have 
$$\mathscr{E}^{\boldsymbol{\sigma}}_{2s_k}(\mathfrak{F}_{k},\Omega_{\bar\delta})=\sigma_{ij}\,\mathcal{I}_{2s_k}(F_{i,k}\cap\Omega_{\bar\delta},F^c_{i,k}\cap\Omega_{\bar\delta})
=\sigma_{ij}\,\mathcal{J}^1_{2s_k}(\widetilde F_{i,k},\Omega_{\bar \delta})\,.$$ 
Using estimate \eqref{estiADPM}, \eqref{reducidentcst}, and the super additivity of $\mathscr{E}^{\boldsymbol{\sigma}}_{2s_k}(\mathfrak{E}_{k},\cdot)$, we deduce that 
\begin{multline}\label{optestisupF}
\mathscr{E}^{\boldsymbol{\sigma}}_{2s_k}(\mathfrak{F}_{k},Q)\leq \mathscr{E}^{\boldsymbol{\sigma}}_{2s_k}(\mathfrak{E}_{k},Q)+\mathscr{R}_k+\sigma_{ij}\,\mathcal{J}^1_{2s_k}(H,A_{1-\delta_1-\varepsilon,1})\\
+\frac{C\sigma_{ij}}{\varepsilon^{n+2s_k}}+C_{\bar\delta}\,\sigma_{ij}\Big[\frac{\|\chi_{E_{i,k}}-\chi_{H}\|_{L^1(A_{1-\delta_1,1-\delta_2})}}{(1-2s_k)}+ \|\chi_{E_{i,k}}-\chi_{H}\|_{L^1(\Omega_{\bar\delta})}\Big]\,.
\end{multline}
To estimate $\mathscr{R}_k$, we first notice that 
\begin{equation}\label{bdsupR2}
\mathscr{R}^2_k\leq \frac{C m \boldsymbol{\sigma}_{\rm max}}{(\bar\delta)^{n+2s_k}} \,,
\end{equation}
for a constant $C$ depending only on $n$. Concerning $\mathscr{R}^1_k$, we have  
\begin{multline*}
\mathcal{I}_{2s_k}(F_{i,k}\cap A_{1-\delta_0,1}, E_{j,k}\cap A_{1-\bar\delta,1-\delta_0})=\mathcal{I}_{2s_k}(E_{i,k}\cap A_{1-\delta_0,1-\delta_1}, E_{j,k}\cap A_{1-\bar\delta,1-\delta_0})\\
+ \mathcal{I}_{2s_k}(F_{i,k}\cap A_{1-\delta_1,1}, E_{j,k}\cap A_{1-\bar\delta,1-\delta_0})\,,
\end{multline*}
and 
\begin{multline*}
\mathcal{I}_{2s_k}(F_{j,k}\cap A_{1-\delta_0,1}, E_{i,k}\cap A_{1-\bar\delta,1-\delta_0})=\mathcal{I}_{2s_k}(E_{j,k}\cap A_{1-\delta_0,1-\delta_1}, E_{i,k}\cap A_{1-\bar\delta,1-\delta_0})\\
+ \mathcal{I}_{2s_k}(F_{j,k}\cap A_{1-\delta_1,1}, E_{i,k}\cap A_{1-\bar\delta,1-\delta_0})\,. 
\end{multline*}
Hence, 
\begin{equation}\label{bdsupR1}
\mathscr{R}^1_k\leq  \mathscr{E}^{\boldsymbol{\sigma}}_{2s_k}(\mathfrak{E}_{k},A_{1-\bar\delta,1}) +\frac{C\boldsymbol{\sigma}_{\rm max}}{(\bar\delta)^{n+2s_k}}\,.
\end{equation}
Combining \eqref{outil72bis}, \eqref{seclaycond}, \eqref{optestisupF}, \eqref{bdsupR2}, and \eqref{bdsupR1}, we deduce that 
$$ \Gamma^\sharp_{ij}\leq \limsup_{k\to\infty}\,(1-2s_k)\mathscr{E}^{\boldsymbol{\sigma}}_{2s_k}(\mathfrak{F}_{k},Q)\leq \Gamma^\flat_{ij} 
+ \frac{3\varepsilon}{2} +\boldsymbol{\sigma}_{\rm max}\limsup_{k\to\infty}\,(1-2s_k)\, \mathcal{J}^1_{2s_k}(H,A_{1-\delta_1-\varepsilon,1})\,.$$
According to \cite{ADPM}, we have 
\begin{multline*}
\limsup_{k\to\infty}\,(1-2s_k)\, \mathcal{J}^1_{2s_k}(H,A_{1-\delta_1-\varepsilon,1})=\lim_{k\to\infty}\,(1-2s_k)\, \mathcal{J}^1_{2s_k}(H,A_{1-\delta_1-\varepsilon,1})\\
= \omega_{n-1}{\rm Per}(H,A_{1-\delta_1-\varepsilon,1})= \omega_{n-1} \big(1-(1-\bar \delta/2-\varepsilon)^{n-1}\big)\,. 
\end{multline*}
Therefore, 
$$  \Gamma^\sharp_{ij}\leq \Gamma^\flat_{ij} 
+ \frac{3\varepsilon}{2} + \omega_{n-1}\boldsymbol{\sigma}_{\rm max} \big(1-(1-\bar \delta/2-\varepsilon)^{n-1}\big)\,,$$
and the conclusion follows letting first $\bar\delta\to0$, and then $\varepsilon\to0$. 
\end{proof}

We are now ready to prove Proposition \ref{propgammainf}. 

\begin{proof}[Proof of Proposition \ref{propgammainf}.]
Let $\big(\{s_k\}_{k\in\mathbb{N}},\{\mathfrak{E}_k\}_{k\in\mathbb{N}}\big)\in  \mathscr{C}^\sharp_{ij}$ be an arbitrary sequence. Since the value $\mathscr{E}_{2s_k}^{{\boldsymbol{\sigma}}}(\mathfrak{E}_{k},Q)$ does not depend on $\mathfrak{E}_{k}$ outside $Q$, we may assume that $\mathfrak{E}_{k}=(E_{1,k},\ldots, E_{m,k})$ satisfies $E_{i,k}\cap Q^c=H\cap Q^c$, $E_{j,k}\cap Q^c=H^c\cap Q^c$, and $E_{l,k}\cap Q^c=\emptyset$ for $l\not\in\{1,j\}$.  Setting
\begin{multline*}
\overline\Gamma_{ij}(s):=\inf \Big\{ (1-2s)\mathscr{P}^{\bar{\boldsymbol{\sigma}}}_{2s}(\mathfrak{E},Q): 
\mathfrak{E}=(E_{1},\ldots,E_m)\in\mathscr{A}_m^{2s}(Q)\,,\\
E_i\cap Q^c=H\cap Q^c\,,\;E_j\cap Q^c=H^c\cap Q^c\,,\text{ and } E_l\cap Q^c=\emptyset \text{ for }l\not\in\{i,j\}\Big\}\,,
\end{multline*}
we then have 
\begin{multline}\label{vendrcpa-2}
\overline\Gamma_{ij}(s_k) \leq (1-2s_k)\mathscr{P}^{\bar{\boldsymbol{\sigma}}}_{2s_k}(\mathfrak{E}_k,Q)\leq (1-2s_k)\mathscr{P}^{{\boldsymbol{\sigma}}}_{2s_k}(\mathfrak{E}_k,Q)\\
= (1-2s_k)\mathscr{E}^{{\boldsymbol{\sigma}}}_{2s_k}(\mathfrak{E}_k,Q)+(1-2s_k)\mathscr{F}^{{\boldsymbol{\sigma}}}_{2s_k}(\mathfrak{E}_k,Q)\,,
\end{multline}
since $\bar\sigma_{lh}\leq \sigma_{lh}$ for every $l,h\in\{1,\ldots,m\}$. 

By definition of the class  $\mathscr{C}^\sharp_{ij}$, there exists $\bar \delta>0$ such that for all $k\in\mathbb{N}$, 
$E_{l,k}\subset (1-\bar\delta)Q$ for $l\not\in\{i,h\}$, $E_{i,k}\cap A_{1-\bar\delta,1}=H\cap A_{1-\bar\delta,1}$, and $E_{j,k}\cap A_{1-\bar\delta,1}=H^c\cap A_{1-\bar\delta,1}$.  Hence, it easily follows that for $\delta\in(0,\bar\delta)$, 
\begin{multline}\label{vendrcpa-1}
\mathscr{F}^{{\boldsymbol{\sigma}}}_{2s_k}(\mathfrak{E}_k,Q)\\\leq 
\sigma_{ij}\Big[\mathcal{I}_{2s_k}(H\cap A_{1-\delta,1}, H^c\cap A_{1,1+\delta})+ \mathcal{I}_{2s_k}(H^c\cap A_{1-\delta,1}, H\cap A_{1,1+\delta})\Big]+\frac{C m \boldsymbol{\sigma}_{\rm max}}{\delta^{n+2s_k}}\\
\leq 2\boldsymbol{\sigma}_{\rm max}\, \mathcal{I}_{2s_k}(H\cap A_{1-\delta,1+\delta}, H^c\cap A_{1-\delta,1+\delta})+\frac{C m \boldsymbol{\sigma}_{\rm max}}{\delta^{n+2s_k}}\,,
\end{multline}
for a constant $C$ depending only on $n$. On the other hand, it follows from \cite[Lemma 8 \& 9]{ADPM} that 
\begin{equation}\label{vendrcpa}
 \limsup_{k\to\infty}\,(1-2s_k)\mathcal{I}_{2s_k}(H\cap A_{1-\delta,1+\delta}, H^c\cap A_{1-\delta,1+\delta})\leq \omega_{n-1}\big[(1+\delta)^{n-1}-(1-\delta)^{n-1}\big] \,.
 \end{equation}
Combining \eqref{vendrcpa-2}, \eqref{vendrcpa-1}, and \eqref{vendrcpa} leads to  
$$\overline\Gamma_{ij}\leq \liminf_{k\to\infty}\,  (1-2s_k)\mathscr{P}^{\bar{\boldsymbol{\sigma}}}_{2s_k}(\mathfrak{E}_k,Q)\leq  \liminf_{k\to\infty}\, (1-2s_k)\mathscr{E}^{{\boldsymbol{\sigma}}}_{2s_k}(\mathfrak{E}_k,Q)+C\boldsymbol{\sigma}_{\rm max}\delta\,. $$
Letting $\delta\to0$, we deduce that $\overline\Gamma_{ij}\leq  \liminf_{k}\, (1-2s_k)\mathscr{E}^{{\boldsymbol{\sigma}}}_{2s_k}(\mathfrak{E}_k,Q)$. By arbitrariness of $\big(\{s_k\}_{k\in\mathbb{N}},\{\mathfrak{E}_k\}_{k\in\mathbb{N}}\big)$, we derive that $\overline\Gamma_{ij}\leq \Gamma_{ij}^\sharp$, and the conclusion follows from Lemma \ref{lemmasharp}. 
\end{proof}

\section{Minimality  of the half-space}\label{secminhalfsp}

We recall from the previous the previous section that for each $i\not=j$, 
\begin{equation}\label{defsecbargammaij}
\overline\Gamma_{ij}:=\liminf_{s\to 1/2^-} \,\overline\Gamma_{ij}(s)\,\,,
\end{equation}
where 
\begin{multline*}
\overline\Gamma_{ij}(s):=(1-2s)\inf\Big\{\mathscr{P}^{\bar{\boldsymbol{\sigma}}}_{2s}(\mathfrak{E},Q): 
\mathfrak{E}=(E_{1},\ldots,E_m)\in\mathscr{A}_m^{2s}(Q)\,,\\
E_i\cap Q^c=H\cap Q^c\,,\;E_j\cap Q^c=H^c\cap Q^c\,,\text{ and } E_l\cap Q^c=\emptyset \text{ for }l\not\in\{i,j\}\Big\}\,.
 \end{multline*}
We aim to determine for an arbitrary matrix $\bar{\boldsymbol{\sigma}}\in\mathscr{S}_m$ satisfying the triangle inequality if the half space $H$ itself solves the minimization problem above. To be more precise, given $i\not=j$, we consider the competitor $\mathfrak{H}_{ij}=(H_1,\ldots,H_m)\in \mathscr{A}_m^s(Q)$ defined by
$$H_i:=H\,,\quad H_j:=H^c\,,\quad H_l=\emptyset \;\text{ for }l\not\in\{i,j\}\,,$$
and we are interesting in determining if $\mathfrak{H}_{ij}$ solves the minimization problem. According to the next proposition, a positive answer imply the $\Gamma$-convergence of $(1-2s)\mathscr{P}^{{\boldsymbol{\sigma}}}_{2s}(\cdot,\Omega)$   as $s\to 1/2^-$.

\begin{proposition}
Assume that for all $s\in(0,1/2)$,  
\begin{equation}\label{mincondH}
(1-2s)\mathscr{P}^{\bar{\boldsymbol{\sigma}}}_{2s}(\mathfrak{H}_{ij},Q)=\overline\Gamma_{ij}(s)\quad \text{for all $i\not= j$}\,.
\end{equation}
Then $\overline\Gamma_{ij}=\omega_{n-1}\bar\sigma_{ij}$ for all $i\not= j$,  the $\liminf$ in \eqref{defsecbargammaij} is a limit, and as a consequence, the functionals  $(1-2s)\mathscr{P}^{{\boldsymbol{\sigma}}}_{2s}(\cdot,\Omega)$ $\Gamma(L^1(\Omega))$-converge  to $\omega_{n-1}\mathscr{P}^{\bar{\boldsymbol{\sigma}}}_{1}(\cdot,\Omega)$ as $s\to 1/2^-$, i.e., 
\begin{equation}\label{liminf=limsup}
({\mathscr{P}}^{\boldsymbol{\sigma}})_*(\mathfrak{E},\Omega)=({\mathscr{P}}^{\boldsymbol{\sigma}})^*(\mathfrak{E},\Omega)=\omega_{n-1}\mathscr{P}^{\bar{\boldsymbol{\sigma}}}_{1}(\mathfrak{E},\Omega)
\end{equation}
for every partition $\mathfrak{E}=(E_1,\ldots,E_m)$ of $\Omega$ by measurable sets. 
\end{proposition}

\begin{proof}
By assumption, we have
$$\overline\Gamma_{ij}(s)=(1-2s)\mathscr{P}^{\bar{\boldsymbol{\sigma}}}_{2s}(\mathfrak{H}_{ij},Q)=\bar\sigma_{ij}(1-2s)P_{2s}(H,Q) \,,$$
and by  \cite[Lemma 9]{ADPM}, 
$$\lim_{s\to 1/2^-}(1-2s)P_{2s}(H,Q) =\omega_{n-1}\,.$$
Then  \eqref{liminf=limsup} follows from Proposition \ref{Gamma lim supbis} and Proposition \ref{propgammainf}. 
\end{proof}

Hence we are only left with the problem of proving  \eqref{mincondH}, which the purpose of the following theorem.

\begin{theorem}\label{thmminhalspace}
For all $s\in(0,1/2)$ and $i\not=j$,
$$(1-2s)\mathscr{P}^{\bar{\boldsymbol{\sigma}}}_{2s}(\mathfrak{H}_{ij},Q)=\overline\Gamma_{ij}(s)\,.$$
\end{theorem}

We shall prove this theorem in the very last subsection. Before this, we would like to prove it for some very specific cases of $\bar{\boldsymbol{\sigma}}$ for which 
the functional $\mathscr{P}^{\bar{\boldsymbol{\sigma}}}_{2s}$ can be represented in a different way in the spirit of \cite{EseOtt}. We believe that these representations can be useful for different purposes, 
for instance when studying the (local) regularity  of minimizers, and it motivates our discussion.

\subsection{The additive case}

We say that the matrix $\bar{\boldsymbol{\sigma}}$ is {\sl additive} if there exist values $\alpha_1,\ldots,\alpha_m\geq 0$ such that
$$\bar\sigma_{ij}=\alpha_i+\alpha_j \quad\text{for all $i\not=j$}\,.$$
In this case, an elementary computation shows that 
$$ \mathscr{P}^{\bar{\boldsymbol{\sigma}}}_{2s}(\mathfrak{E},\Omega)= \sum_{k=1}^m\alpha_k P_{2s}(E_k,\Omega)$$
for every partition $\mathfrak{E}=(E_1,\ldots,E_m)\in\mathscr{A}_m^s(\Omega)$. 
\vskip5pt

Therefore, if $\mathfrak{E}=(E_1,\ldots,E_m)\in\mathscr{A}_m^s(Q)$ is a competitor for  $\overline\Gamma_{ij}(s)$, then 
$$ \mathscr{P}^{\bar{\boldsymbol{\sigma}}}_{2s}(\mathfrak{E},Q)\geq \alpha_i P_{2s}(E_i,Q)+\alpha_j P_{2s}(E_j,Q)\,.$$
By the minimality of $H$ proved in \cite[Proposition 17]{ADPM}, and since $E_1\cap Q^c=H\cap Q^c$, $E_j\cap Q^c=H^c\cap Q^c$, we have 
$$ \mathscr{P}^{\bar{\boldsymbol{\sigma}}}_{2s}(\mathfrak{E},Q)\geq \alpha_i P_{2s}(H,Q)+\alpha_j P_{2s}(H^c,Q)=(\alpha_i+\alpha_j)P_{2s}(H,Q)
= \mathscr{P}^{\bar{\boldsymbol{\sigma}}}_{2s}(\mathfrak{H}_{ij},Q)\,,$$
and the minimality of $\mathfrak{H}_{ij}$ follows.

\subsection{The $3$-phases and $4$-phases cases}

\subsubsection{The $3$-phases case} For $m=3$, the triangle inequality \eqref{traingineqlem} satisfied by $\bar{\boldsymbol{\sigma}}$ implies that $\bar{\boldsymbol{\sigma}}$ is actually additive. Indeed, one has $\bar\sigma_{ij}=\alpha_i+\alpha_j$ with 
$$ \alpha_1 = \frac{\bar\sigma_{12} + \bar\sigma_{13} - \bar\sigma_{23}}{2}\,,\;  \alpha_2 = \frac{\bar\sigma_{12} + \bar\sigma_{23} - \bar\sigma_{13}}{2}\,,\;\text{ and } \alpha_3 = \frac{\bar\sigma_{23} + \bar\sigma_{13} - \bar\sigma_{12}}{2}\,.$$
Therefore, the minimality of  $\mathfrak{H}_{ij}$ for $m=3$ follows from the additive case.

\subsubsection{The $4$-phases case} 

For $m=4$, the matrix  $\bar{\boldsymbol{\sigma}}$ may no longer be additive, but a quite similar property actually holds. Following computations from  \cite{BM}, we consider the set of coefficients 
$$  \left\{
    \begin{array}{ll}
         \tilde{\alpha}_1 = \frac{1}{2}(\bar\sigma_{12} + \bar\sigma_{13} + \bar\sigma_{14})\\[5pt]
         \tilde{\alpha}_2 = \frac{1}{2}(\bar\sigma_{12} +\bar\sigma_{23} + \bar\sigma_{24})\\[5pt]
         \tilde{\alpha}_3 = \frac{1}{2}(\bar\sigma_{13} + \bar\sigma_{23} + \bar\sigma_{34}) \\[5pt]
         \tilde{\alpha}_4 = \frac{1}{2}(\bar\sigma_{14} + \bar\sigma_{24} + \bar\sigma_{34})
    \end{array}
\right.
\text{ and }
 \left\{
    \begin{array}{ll}
         \tilde{\alpha}_5 = \frac{1}{2}(\bar\sigma_{12} + \bar\sigma_{34}) \\[5pt]
         \tilde{\alpha}_6 = \frac{1}{2}(\bar\sigma_{13} + \bar\sigma_{24}) \\[5pt]
         \tilde{\alpha}_7 = \frac{1}{2}(\bar\sigma_{14} + \bar\sigma_{23}) \\[5pt]
    \end{array}
\right. \,,  $$
which satisfy 
$$  \left\{
    \begin{array}{ll}
        \bar\sigma_{12} =  (\tilde{\alpha}_1+ \tilde{\alpha}_2) - (\tilde{\alpha}_6+ \tilde{\alpha}_7)\\[5pt]
         \bar\sigma_{13}  =   (\tilde{\alpha}_1+ \tilde{\alpha}_3) - (\tilde{\alpha}_5+ \tilde{\alpha}_7)\\[5pt]
        \bar\sigma_{14}  = (\tilde{\alpha}_1+ \tilde{\alpha}_4) - (\tilde{\alpha}_5+ \tilde{\alpha}_6)\\[5pt]
         \bar\sigma_{23}  =   (\tilde{\alpha}_2+ \tilde{\alpha}_3) -(\tilde{\alpha}_5+ \tilde{\alpha}_6) \\[5pt]
           \bar\sigma_{24}  = (\tilde{\alpha}_2+ \tilde{\alpha}_4) - (\tilde{\alpha}_5+ \tilde{\alpha}_7)\\[5pt]
             \bar\sigma_{34}  = (\tilde{\alpha}_3+ \tilde{\alpha}_4) - (\tilde{\alpha}_6+ \tilde{\alpha}_7)
    \end{array}
\right.\,.$$
Based on these relations, a direction computation shows that 
\begin{multline*}
 \mathscr{P}^{\bar{\boldsymbol{\sigma}}}_{2s}(\mathfrak{E},\Omega)=\sum_{k=1}^4 \tilde{\alpha}_kP_{2s}(E_k,\Omega)\\
  - \Big[\tilde{\alpha}_5P_{2s}(E_1\cup E_2,\Omega)
+ \tilde{\alpha}_6P_{2s}(E_1\cup E_3,\Omega)+ \tilde{\alpha}_7P_{2s}(E_1\cup E_4,\Omega)\Big]\,.
\end{multline*}
for every partition $\mathfrak{E}=(E_1,\ldots,E_4)\in\mathscr{A}_4^s(\Omega)$.  
On the other hand, one may observe that 
$$ \sum \limits_{k= 1}^4 P_{2s}(E_k,\Omega) = P_{2s}(E_1 \cup E_2,\Omega) + P_{2s}(E_1 \cup E_3,\Omega) + P_{2s}(E_1 \cup E_4,\Omega) \,.$$ 
Therefore, introducing a further parameter $\alpha_*\geq 0$, we have 
\begin{multline*}
 \mathscr{P}^{\bar{\boldsymbol{\sigma}}}_{2s}(\mathfrak{E},\Omega)=\sum_{k=1}^4 (\tilde{\alpha}_k-\alpha_*)P_{2s}(E_k,\Omega)\\
  + \Big[(\alpha_*-\tilde{\alpha}_5)P_{2s}(E_1\cup E_2,\Omega)
+(\alpha_*- \tilde{\alpha}_6)P_{2s}(E_1\cup E_3,\Omega)+ (\alpha_*-\tilde{\alpha}_7)P_{2s}(E_1\cup E_4,\Omega)\Big]\,.
\end{multline*}
Choosing 
$$\alpha_*:=\max\{\tilde{\alpha}_5,\tilde{\alpha}_6,\tilde{\alpha}_7\}\,, $$
we notice that the triangle inequality \eqref{traingineqlem} implies that $\tilde{\alpha}_k-\alpha_*\geq 0$ for each $k\in\{1,2,3,4\}$. 
\vskip5pt

Back to the minimality problem of the half space, we may assume that $i=1$ and $j=2$, relabelling the indices if necessary.  Then, if $\mathfrak{E}=(E_1,\ldots,E_4)\in\mathscr{A}_4^s(Q)$ 
is a competitor for $\overline\Gamma_{12}(s)$, we have 
\begin{multline*}
 \mathscr{P}^{\bar{\boldsymbol{\sigma}}}_{2s}(\mathfrak{E},Q)\geq  
(\tilde{\alpha}_1-\alpha_*)P_{2s}(E_1,Q)+ (\tilde{\alpha}_2-\alpha_*)P_{2s}(E_2,Q)\\
+(\alpha_*- \tilde{\alpha}_6)P_{2s}(E_1\cup E_3,Q)+(\alpha_*- \tilde{\alpha}_6)P_{2s}(E_1\cup E_4,Q)\,.
\end{multline*}
Again, by the minimality of $H$  in \cite[Proposition 17]{ADPM} and the facts that $E_1\cap Q^c=H\cap Q^c$, $E_2\cap Q^c=H^c\cap Q^c$, $(E_1\cup E_3)\cap Q^c=H\cap Q^c$,  $(E_1\cup E_4)\cap Q^c=H\cap Q^c$, we deduce that 
\begin{align*}
 \mathscr{P}^{\bar{\boldsymbol{\sigma}}}_{2s}(\mathfrak{E},Q)& \begin{multlined}[t]
 \geq  (\tilde{\alpha}_1-\alpha_*)P_{2s}(H,Q)+ (\tilde{\alpha}_2-\alpha_*)P_{2s}(H^c,Q)\\[5pt]
\qquad\qquad +(\alpha_*- \tilde{\alpha}_6)P_{2s}(H,Q)+(\alpha_*- \tilde{\alpha}_6)P_{2s}(H,Q)
\end{multlined} \\
&=(\tilde{\alpha}_1+\tilde{\alpha}_2- \tilde{\alpha}_6- \tilde{\alpha}_7)P_{2s}(H,Q)\\
&=  \mathscr{P}^{\bar{\boldsymbol{\sigma}}}_{2s}(\mathfrak{H}_{12},Q) \,,
\end{align*}
whence the announced minimality of $\mathfrak{H}_{12}$.

\subsection{The $\ell_1$-embeddable case}

For $m\geq 5$, there are in general no suitable decompositions as in the cases $m=3$ and $m=4$. The natural assumption under which such decomposition  
holds is requiring that the matrix $\bar{\boldsymbol{\sigma}}$ is $\ell_1$-embeddable. This condition means that we can find a dimension $d\geq 1$ and $m$ points ${\bf a}_1,\ldots,{\bf a}_m\in \R^d$ 
such that $\bar\sigma_{ij}=\|{\bf a}_i-{\bf a}_j\|_1$ for all $i,j\in\{1,\ldots,m\}$, where $\|\cdot\|_1$ denotes the usual $1$-norm in $\R^d$.  It is well known that for $m\in\{3,4\}$, any matrix $\bar{\boldsymbol{\sigma}}$ satisfying the triangle inequality  \eqref{traingineqlem} is  $\ell_1$-embeddable (see e.g. \cite{DezLau} and above). In turn, requiring that $\bar{\boldsymbol{\sigma}}$ is $\ell_1$-embeddable is equivalent to requiring that $\bar{\boldsymbol{\sigma}}$ belongs to the  {\sl cut cone} (see e.g. \cite{DezLau}), i.e., it can be written as 
$$\bar{\boldsymbol{\sigma}}=\sum_{J\subset\{1,\ldots,m\}}\lambda_J\,\boldsymbol{\delta}^J $$
for some coefficients $\lambda_J\geq 0$, where each $\boldsymbol{\delta}^J=(\delta_{ij}^J)$ is the {\sl cut matrix} defined by 
$$ \delta_{ij}^J=1\text{ if }{\rm Card}(J\cap\{i,j\})=1\,,\text{ and }\, \delta_{ij}^J=0\text{ otherwise}\,.$$
Therefore, if $\bar{\boldsymbol{\sigma}}$ is $\ell_1$-embeddable, then a direct computation shows that 
$$ \mathscr{P}^{\bar{\boldsymbol{\sigma}}}_{2s}(\mathfrak{E},\Omega)=\sum_{J\subset\{1,\ldots,m\}} \lambda_JP_{2s}\Big(\bigcup_{k\in J}E_k,\Omega\Big) $$
for every partition $\mathfrak{E}=(E_1,\ldots,E_m)\in\mathscr{A}_m^s(\Omega)$.  
\vskip5pt

Concerning the minimality of $\mathfrak{H}_{ij}$ in the $\ell_1$-embeddable case, we consider again an arbitrary competitor $\mathfrak{E}=(E_1,\ldots,E_m)\in\mathscr{A}_m^s(Q)$ 
 for $\overline\Gamma_{ij}(s)$. Given $J\subset\{1,\ldots,m\}$, we observe that for $i\in J$ and $j\not\in J$, we have $\bigcup_{k\in J}E_k\cap Q^c=H\cap Q^c$, while $\bigcup_{k\in J}E_k\cap Q^c=H^c\cap Q^c$ for $i\not\in J$ and $j\in J$. Still by the minimality of $H$  in \cite[Proposition 17]{ADPM}, we deduce that 
\begin{align*}
 \mathscr{P}^{\bar{\boldsymbol{\sigma}}}_{2s}(\mathfrak{E},Q)&\geq 
\sum_{J\subset\{1,\ldots,m\}: i\in J,j\not \in J} \lambda_JP_{2s}\Big(\bigcup_{k\in J}E_k,Q\Big)
+\sum_{J\subset\{1,\ldots,m\}: i\not\in J,j \in J} \lambda_J P_{2s}\Big(\bigcup_{k\in J}E_k,Q\Big)\\
& \geq \sum_{J\subset\{1,\ldots,m\}: i\in J,j\not \in J} \lambda_JP_{2s}(H,Q)
+\sum_{J\subset\{1,\ldots,m\}: i\not\in J,j \in J} \lambda_JP_{2s}(H^c,Q)\\
& = \Big(\sum_{J\subset\{1,\ldots,m\}}\lambda_J\delta_{ij}^J\Big)P_{2s}(H,Q)\\
& =  \mathscr{P}^{\bar{\boldsymbol{\sigma}}}_{2s}(\mathfrak{H}_{ij},Q)\,,
\end{align*}
and the minimality of $\mathfrak{H}_{ij}$ follows once again.

\subsection{The general case}

As mentioned above, for $m$ and $\bar{\boldsymbol{\sigma}}$ arbitrary, there are no suitable decompositions of $\mathscr{P}^{\bar{\boldsymbol{\sigma}}}_{2s}$ we can rely on, and we need a (slightly)  different approach. The proof of Theorem \ref{thmminhalspace} in the general case still makes use of the minimality of $H$ in the case $m=2$ taken from \cite{CRS}  or \cite[Proposition~17]{ADPM}, and it crucially rests on the following {\sl replacement} lemma. 

\begin{lemma}\label{leonardilemma}
Let $\mathfrak{E}\in\mathscr{A}_m^s(Q)$ be an arbitrary competitor for $\overline\Gamma_{ij}(s)$. There exists an admissible competitor $\mathfrak{F}=(F_{1},\ldots,F_m)\in\mathscr{A}_m^s(Q)$
 for $\overline\Gamma_{ij}(s)$ such that $F_l=\emptyset$ for $l\not\in\{i,j\}$ and
 $$\mathscr{P}^{\bar{\boldsymbol{\sigma}}}_{2s}(\mathfrak{F},Q)\leq \mathscr{P}^{\bar{\boldsymbol{\sigma}}}_{2s}(\mathfrak{E},Q)\,. $$
\end{lemma}

We postpone at the end of the subsection the proof of this lemma. 

\begin{proof}[Proof of Theorem \ref{thmminhalspace}]
Consider an arbitrary $\mathfrak{E}\in\mathscr{A}_m^s(Q)$ for $\overline\Gamma_{ij}(s)$. Then Lemma \ref{leonardilemma} yields 
$$\mathscr{P}^{\bar{\boldsymbol{\sigma}}}_{2s}(\mathfrak{E},Q)\geq  \mathscr{P}^{\bar{\boldsymbol{\sigma}}}_{2s}(\mathfrak{F},Q)=\bar\sigma_{ij}P_{2s}(F_i,Q)\,, $$
since $F_j=(F_i)^c$ and $F_l=\emptyset$ for $l\not\in\{i,j\}$. On the other hand, since $F_i\cap Q^c=H\cap Q^c$,  \cite[Proposition 17]{ADPM} tells us that 
$$\bar\sigma_{ij}P_{2s}(F_i,Q)\geq \bar\sigma_{ij}P_{2s}(H,Q)=\mathscr{P}^{\bar{\boldsymbol{\sigma}}}_{2s}(\mathfrak{H}_{ij},Q)\,.$$
Therefore $\mathscr{P}^{\bar{\boldsymbol{\sigma}}}_{2s}(\mathfrak{E},Q)\geq \mathscr{P}^{\bar{\boldsymbol{\sigma}}}_{2s}(\mathfrak{H}_{ij},Q)$, and the conclusion follows. 
\end{proof}

\begin{remark}[Uniqueness]
It is proved in  \cite[Proposition 17]{ADPM} that  $H$ is in fact the unique minimizer for the two phases problem. As a consequence, by the argument above, $\mathfrak{H}_{ij}$ 
is even the  unique minimizer for $\overline\Gamma_{ij}(s)$. 
\end{remark}

Concerning the proof of Lemma  \ref{leonardilemma}, we shall follow closely an original argument from~\cite{Leo}  (in the context of the usual perimeter for partitions) where 
a competitor is represented as a directed graph, and the abstract construction of a better competitor is obtained by the {\sl max-flow min-cut  theorem} from graph theory (see e.g. \cite{Bon}). To be as clear as possible and self contained, we shall borrow (almost all) the notations from \cite{Leo} and the basics from graph theory. 
\vskip5pt

A direct graph $G$ is a finite set of vertices $\{v_i\}^m_{i=1}$ connected by oriented arcs. An arc (or edge) joining $v_i$ to $v_j$ is an ordered pair ${\rm e}_{ij}:=(v_i,v_j)$ where $v_i$ is called the tail, and $v_j$ the head. The set $A$ of all arcs is then a subset of $G\times G$. We say that the graph $G$ is weighted if a number $p_{ij}\geq 0$ is associated to each arc ${\rm e}_{ij}$, and the fonction ${\rm c}:A\to[0,\infty)$ given by ${\rm c}({\rm e}_{ij})=p_{ij}$ is called capacity function. A pair $(G,{\rm c})$ is called a network. 

We shall be concerned with a network $(G,{\rm c})$ where each vertex $v_i$ has no arc to itself (i.e., ${\rm e}_{ii}\not\in A$), and every $v_i\not=v_j$ are connected in both ways (i.e., ${\rm e}_{ij}\in A$ and ${\rm e}_{ji}\in A$). In other words, $A=G\times G\setminus\Delta$ where $\Delta$ is the diagonal of $G\times G$, and $G$ is said to be completely connected.  The capacity function ${\rm c}$ will be chosen symmetric, that is $p_{ij}=p_{ji}$, and allowed to vanish.    

Given such a network $(G,{\rm c})$ and a pair of vertices ${\rm s}$ and ${\rm t}$ (called the source and the sink), a function $f:A\to[0,\infty)$ is said to be a flow from ${\rm s}$ to ${\rm t}$ if 
\begin{itemize}
\item[1)] ({\sl Capacity constraint}) $f({\rm e}_{ij})\leq {\rm c}({\rm e}_{ij})=p_{ij}$ for all ${\rm e}_{ij}\in A$; 
\item[2)] ({\sl Conservation of flows})  the (so-called) node function $\phi_f:G\to\R$ defined by
$$\phi_f(v_i):=\sum_{j\not= j}f({\rm e}_{ij})-f({\rm e}_{ji}) $$ 
satisfies $\phi_f(v_i)=0$ for $v_i\not\in\{{\rm s},{\rm t}\}$, and $\phi_f({\rm s})=-\phi_f({\rm t})\geq 0$. 
\end{itemize}
The value $\phi_f({\rm s})$ is called intensity of $f$, and we shall write it as $\|f\|:=\phi_f({\rm s})$. 

Then, a bipartition $\mathcal{K}=(K_1,K_2)$ of $G$ is said to be a cut with respect to ${\rm s}$ and ${\rm t}$ if ${\rm s}\in K_1$ and ${\rm t}\in K_2$. The size (or length) of  $\mathcal{K}$ is defined by 
$$\mathscr{L}(\mathcal{K}):= \sum_{i\in K_1\,,\;j\in K_2}p_{ij}\,.$$ 
It turns out that $\|f\|\leq \mathscr{L}(\mathcal{K})$ for any flow $f$ from ${\rm s}$ to ${\rm t}$ and any cut $\mathcal{K}$ with respect to ${\rm s}$ and ${\rm t}$. 

Finally, we say that a flow $f$ from ${\rm s}$ to ${\rm t}$ is a maximum flow if $\|f\|\geq \|f^\prime\|$ for any other flow $f^\prime$ from ${\rm s}$ to ${\rm t}$. Similarly, a cut $\mathcal{K}$ with respect to ${\rm s}$ and ${\rm t}$ is said to be minimal if $\mathscr{L}(\mathcal{K})\leq \mathscr{L}(\mathcal{K}^\prime)$ for any other cut $\mathcal{K}^\prime$ with respect to ${\rm s}$ and ${\rm t}$. Maximum flows  and minimal cuts exist and algorithms to compute them lead to the famous max-flow min-cut theorem asserting that  $\|f\|=\mathscr{L}(\mathcal{K})$ whenever $f$ is a maximum flow 
from ${\rm s}$ to ${\rm t}$ and $\mathcal{K}$ is minimal cut  with respect to ${\rm s}$ and ${\rm t}$. 
\vskip5pt

We still need one last ingredient from \cite{Leo}, which is a suitable ``flow decomposition''. First, define a path from ${\rm s}$ to ${\rm t}$ to be an ordered $d$-tuple (for some $d\geq 1$) of arcs 
$\gamma=({\rm e}_{j_0j_1},\ldots,{\rm e}_{j_{d-1}j_d})$ such that $v_{j_0}={\rm s}$, $v_{j_d}={\rm t}$, and $v_{i_k}\not=v_{i_{l}}$ for $k\not=l$ (i.e., no self intersections). By \cite[Theorem~4.6]{Leo}, given an arbitrary flow 
$f$ from  ${\rm s}$ to ${\rm t}$, there exist paths $\gamma_1,\ldots,\gamma_h$ from ${\rm s}$ to ${\rm t}$ and non negative constants $\alpha_1(f),\ldots,\alpha_h(f)$ such that 
\begin{itemize}
\item[1)] $\displaystyle \|f\|=\sum_{i=1}^h\alpha_i(f)\,$; 

\item[2)] $\displaystyle \sum_{\{i : {\rm e}_{kl}\in \gamma_i\}}\alpha_i(f)\leq f({\rm e}_{kl})\leq p_{kl}\,$ for all arc ${\rm e}_{kl}\in E$. 
\end{itemize}
In addition, we observe that the family $\gamma_1,\ldots,\gamma_h$ can be chosen with the following property: {\sl if an arc ${\rm e}_{kl}$ belongs to some path $\gamma_i$, then 
the opposite arc ${\rm e}_{lk}$ do not belong to any of the $\gamma_j$'s}. Indeed, assume that ${\rm e}_{kl}\in\gamma_i$ and  ${\rm e}_{lk}\in \gamma_j$. 
By  the no self intersections condition, $j\not=i$ and $k,l\not\in\{{\rm s},{\rm t}\}$ (i.e., ${\rm e}_{kl}$ and ${\rm e}_{lk}$ cannot be the first or the last arc of each path).  
Then we write $\gamma_i=(\gamma^{-}_i,{\rm e}_{kl},\gamma_i^+)$ and $\gamma_j=(\gamma^{-}_j,{\rm e}_{lk},\gamma_j^+)$, and we first replace $\gamma_i$ by 
$\widetilde\gamma_i=(\gamma^{-}_i,\gamma_j^+)$ and $\gamma_j$ by $\widetilde\gamma_j=(\gamma^{-}_j,\gamma_i^+)$. The resulting $\widetilde\gamma_i$ and $\widetilde\gamma_j$ 
may contain self intersections, i.e., repeating vertices (necessarily distinct from ${\rm s}$ and ${\rm t}$), that are {\sl cycles}. Removing those cycles, we obtain two  paths $\widehat\gamma_i$ and $\widehat\gamma_j$  from ${\rm s}$ to ${\rm t}$ which do not contain ${\rm e}_{kl}$ or ${\rm e}_{lk}$. Repeating 
this operation as many times as necessary (recall that the graph is finite), we obtain a family of paths satisfying the required property. Since the procedure only removes arcs from the original family, inequality 2) 
above remains valid, and can be rephrased as 
\begin{itemize}
\item[$2)^\prime$] $\displaystyle \sum_{\{i : {\rm e}_{kl}\in \gamma_i\}}\alpha_i(f) +  \sum_{\{i : {\rm e}_{lk}\in \gamma_i\}}\alpha_i(f)\leq p_{kl}\,$ for all arc ${\rm e}_{kl}\in E$. 
\end{itemize}
\vskip5pt

We are now ready to prove Lemma \ref{leonardilemma}.

\begin{proof}[Proof of Lemma  \ref{leonardilemma}]
Without loss of generality, we may assume for simplicity that $i=1$ and $j=2$, and we consider $\mathfrak{E}=(E_1,\ldots,E_m)\in\mathscr{A}_m^s(Q)$  a competitor for $\overline\Gamma_{12}(s)$. Then we represent $\mathfrak{E}$ as the completely connected directed graph $G$ where each vertex $v_i$ corresponds to the chamber~$E_i$. To make it a network, we define the capacity 
function ${\rm c}:A\to[0,\infty)$ to be 
$${\rm c}({\rm e}_{ij})=p_{ij}:= \mathcal{I}_{2s}(E_i\cap Q,E_j\cap Q)+ \mathcal{I}_{2s}(E_i\cap Q,E_j\cap Q^c)+ \mathcal{I}_{2s}(E_i\cap Q^c,E_j\cap Q)\,.$$
Next we consider a minimal cut $\mathcal{K}=(K_1,K_2)$ with respect to  $v_1$ and $v_2$. It allows us to define $\mathfrak{F}=(F_1,\ldots,F_m)\in\mathscr{A}_m^s(Q)$ by 
$$F_i:=\begin{cases} 
\displaystyle \bigcup_{k \in K_i} E_k & \text{for $i\in\{1,2\}$}\,,\\[8pt]
\;\;\emptyset & \text{for $k\geq 3$}\,.
\end{cases}$$
Notice that $F_1\cap Q^c=H\cap Q^c$ and $F_2\cap Q^c=H^c\cap Q^c$ since $E_k\subset Q$ for $k\geq 3$, so that $\mathfrak{F}$ is an admissible competitor for $\overline\Gamma_{12}(s)$. 
In addition, we observe that (by symmetry of $\bar{\boldsymbol{\sigma}}$)
\begin{align*}
\mathscr{P}^{\bar{\boldsymbol{\sigma}}}_{2s}(\mathfrak{F},Q)& =\bar\sigma_{12}\Big[ \mathcal{I}_{2s}(F_1\cap Q,F_2\cap Q)+ \mathcal{I}_{2s}(F_1\cap Q,F_2\cap Q^c)+ \mathcal{I}_{2s}(F_1\cap Q^c,F_2\cap Q)\Big]\\
& = \bar\sigma_{12}\sum_{i\in K_1\,,\;j\in K_2}p_{ij}\\
&=  \bar\sigma_{12} \mathscr{L}(\mathcal{K})\,.
\end{align*}
Now we consider a maximum flow $f$ from $v_1$ to $v_2$. By the max-flow min-cut theorem, the triangle inequality \eqref{traingineqlem}, and the symmetry of $\bar{\boldsymbol{\sigma}}$, we have 
\begin{align*}
\mathscr{P}^{\bar{\boldsymbol{\sigma}}}_{2s}(\mathfrak{F},Q) & =  \bar\sigma_{12}\, \mathscr{L}(\mathcal{K}) = \bar\sigma_{12}\|f\| =\sum_{i=1}^h  \bar\sigma_{12} \,\alpha_i(f)\\
& \leq \sum_{i=1}^h\Big( \sum_{(k,l):{\rm e}_{kl}\in\gamma_i} \bar\sigma_{kl}\,\alpha_i(f)\Big) = \sum_{k\not=l} \Big( \sum_{i :{\rm e}_{kl}\in\gamma_i} \bar\sigma_{kl}\,\alpha_i(f)\Big)\\
& = \sum_{k<l} \Big(\sum_{i :{\rm e}_{kl}\in\gamma_i} \bar\sigma_{kl}\,\alpha_i(f)+\sum_{i :{\rm e}_{lk}\in\gamma_i} \bar\sigma_{lk}\,\alpha_i(f)\Big)\\
& \leq  \sum_{k<l}  \bar\sigma_{kl}\,p_{kl}= \mathscr{P}^{\bar{\boldsymbol{\sigma}}}_{2s}(\mathfrak{E},Q) \,,
\end{align*}
which completes the proof of the lemma.
\end{proof}

\section{Convergence of local minimizers}\label{sectcvlocmin}

In this final section, our goal is to prove Theorem \ref{convlocmin}. The proof is very much inspired from \cite[Proof of Theorem 3]{ADPM}, but it still contains some major differences. 
As in Section \ref{geomconstcube}, the new difficulty compare to  \cite{ADPM} is the gluing technique to match two competitors near a given interface. The gluing problem for partitions is a well known issue compare to the case $m=2$ since the use of some co-area type formula cannot be used, at least directly. 

Before starting the proof, we would like to underline the following elementary fact about local minimality for the functional $\mathscr{P}^{\boldsymbol{\sigma}}_{2s}\,$: it is a local property. 
More precisely, if a partition $\mathfrak{E}$ is a local minimizer of $\mathscr{P}^{\boldsymbol{\sigma}}_{2s}(\cdot,\Omega)$ and $\mathfrak{F}$ is a competing partition which differs from $\mathfrak{E}$ only in a compact  subset of an open set $\Omega^\prime\subset\Omega$, then 
$$\mathscr{P}^{\boldsymbol{\sigma}}_{2s}(\mathfrak{F},\Omega)-\mathscr{P}^{\boldsymbol{\sigma}}_{2s}(\mathfrak{E},\Omega)
=\mathscr{P}^{\boldsymbol{\sigma}}_{2s}(\mathfrak{F},\Omega^\prime)-\mathscr{P}^{\boldsymbol{\sigma}}_{2s}(\mathfrak{E},\Omega^\prime)\geq 0\,,$$
so that $\mathfrak{E}$ is a local minimizer of $\mathscr{P}^{\boldsymbol{\sigma}}_{2s}(\cdot,\Omega^\prime)$.

\begin{proof}[Proof of Theorem \ref{convlocmin}]
{\it Step 1.} We start with the proof of \eqref{bornsuplocmin}. Let $\Omega^\prime\Subset\Omega$. Upon considering a larger set, we may assume that $\Omega^\prime$ has a Lipschitz boundary. 
Then we consider the competitor 
$\mathfrak{F}_k=(F_{1,k},\ldots,F_{m,k})\in\mathscr{A}_m^{s_k}(\Omega)$ defined by $F_{1,k}:=E_{1,k}\cup \Omega^\prime$, and 
$F_{i,k}:=E_{i,k}\cap(\Omega^c\cup(\Omega\setminus\Omega^\prime))$ for $i\geq 2$. The local minimality of  $\mathfrak{E}_k$ yields 
\begin{multline*}
\mathscr{P}^{\boldsymbol{\sigma}}_{2s_k}(\mathfrak{E}_k,\Omega^\prime)
\leq \mathscr{P}^{\boldsymbol{\sigma}}_{2s_k}(\mathfrak{F}_k,\Omega^\prime)=\mathscr{F}_{2s_k}(\mathfrak{F}_k,\Omega^\prime)\\
=\sum_{j=2}^m\sigma_{1j}\,\mathcal{I}_{2s_k}(\Omega^\prime,F_{j,k}\cap (\Omega^\prime)^c)\leq \boldsymbol{\sigma}_{\rm max}P_{2s_k}(\Omega^\prime,\R^n)\,.
\end{multline*}
In view of e.g. \cite[Proposition 16]{ADPM} or \cite{Dav}, we have 
\begin{multline*}
\limsup_{k\to\infty}\,(1-2s_k)\mathscr{P}^{\boldsymbol{\sigma}}_{2s_k}(\mathfrak{E}_k,\Omega^\prime)
\leq  \boldsymbol{\sigma}_{\rm max} \limsup_{k\to\infty}\,(1-2s_k)P_{2s_k}(\Omega^\prime,\R^n)\\
\leq \frac{n\omega_n}{2}\boldsymbol{\sigma}_{\rm max}{\rm Per}(\Omega^\prime,\R^n)<\infty\,, 
\end{multline*}
 and \eqref{bornsuplocmin} is proved. Note that Theorem \ref{equicoercivethm} now implies that the limit $\mathfrak{E}=(E_1,\ldots,E_m)$ is a partition of $\Omega$, and a Caccioppoli partition of $\Omega^\prime$ for every open subset $\Omega^\prime\Subset\Omega$. 
\vskip5pt

\noindent{\it Step 2.} We now complete the proof of the theorem. Arguing exactly as in  \cite[Proof of Theorem~3]{ADPM}, it is enough to provide the argument in the case where $\Omega^\prime$ is a ball. The case of an open set $\Omega^\prime$ with Lipschitz boundary can be obtained through the same lines of proof, and the general case then follows  by an abstract measure theoretic argument. 

From now on, we then assume that $\Omega^\prime=B_R\Subset\Omega$ (we omit to specify its center for simplicity). 
Following \cite[Proof of Theorem~3]{ADPM}, we introduce the regular monotone set function 
$\boldsymbol{\alpha}_k:\mathcal{P}(\Omega)\to[0,\infty)$ defined for an open subset $A\subset\Omega$ by 
$$\boldsymbol{\alpha}_k(A):= (1-2s_k)\mathscr{E}^{\boldsymbol{\sigma}}_{2s_k}(\mathfrak{E}_k,A)\,,$$
and extended to an arbitrary subset $B\in\mathcal{P}(\Omega)$ by 
$$\boldsymbol{\alpha}_k(B):=\inf\{\boldsymbol{\alpha}_k(A) : A\subset\Omega \text{ open}\,,\;B\subset A \} \,.$$
By Step 1 and De Giorgi-Letta's Theorem (see e.g. \cite[Theorem 21]{ADPM}), we can extract a (not relabeled) subsequence such that $\boldsymbol{\alpha}_k$ weakly converges as $k\to\infty$ to some regular monotone set function $\boldsymbol{\alpha}$. Since  each $\boldsymbol{\alpha}_k$ is super additive, $\boldsymbol{\alpha}$ is also super additive. 

We consider a ball $B_R\Subset\Omega$  satisfying $\boldsymbol{\alpha}(\partial B_R)=0$ (note that there are at most countably  many $R$'s such that $\boldsymbol{\alpha}(\partial B_R)>0$ 
since  $\boldsymbol{\alpha}$ is locally finite and super additive). We aim to show that $\mathfrak{E}$ is local minimizer of 
$\mathscr{P}_1^{\bar{\boldsymbol{\sigma}}}(\cdot,B_R)$, and that  
$$\lim_{k\to\infty}\,(1-2s_k)\mathscr{P}^{\boldsymbol{\sigma}}_{2s_k}(\mathfrak{E}_k,B_R)=\omega_{n-1} \mathscr{P}_1^{\bar{\boldsymbol{\sigma}}}(\mathfrak{E},B_R)\,.$$
To this purpose, we consider an arbitrary Caccioppoli partition  $\mathfrak{F}=(F_1,,\ldots,F_m)$ of $B_R$  satisfying $F_i\triangle E_i\subset B_{R-\delta}$ for every $i\in\{1,\ldots,m\}$ for some $\delta>0$ small. By Theorem \ref{GCthm} (and a standard diagonal argument, see e.g. the proof of Lemma \ref{lscgamlimsup}), there exists a sequence 
$\mathfrak{F}_k=(F_{1,k},,\ldots,F_{m,k})\in\mathscr{A}_m^{s_k}(B_R)$ such that $\chi_{F_{i,k}}\to\chi_{E_i}$ in $L^1(B_R)$ for each $i\in\{1,\ldots,m\}$, and 
\begin{equation}\label{approxcompetit}
\lim_{k\to\infty}\, \,(1-2s_k)\mathscr{P}^{\boldsymbol{\sigma}}_{2s_k}(\mathfrak{F}_k,B_R)=\omega_{n-1} \mathscr{P}_1^{\bar{\boldsymbol{\sigma}}}(\mathfrak{F},B_R)\,.
\end{equation}
Next we fix a small $\varepsilon>0$ arbitrary. According to e.g. \cite[Proposition 22]{ADPM}, since $\boldsymbol{\alpha}(\partial B_R)=0$, we have 
$$0=\lim_{\bar\delta\to 0^+}\limsup_{k\to\infty}\boldsymbol{\alpha}_k(\overline B_{R+\bar\delta}\setminus B_{R-\bar\delta})= \lim_{\bar\delta\to 0^+}\limsup_{k\to\infty}\boldsymbol{\alpha}_k( B_{R+\bar\delta}\setminus  B_{R-\bar\delta})\,.$$
Therefore, whenever $\bar\delta\in(0,\delta)$ is small enough, we have  
\begin{equation}\label{goodlay1}
\limsup_{k\to\infty}\, (1-2s_k)\mathscr{E}^{\boldsymbol{\sigma}}_{2s_k}(\mathfrak{E}_k,B_{R+\bar\delta}\setminus  B_{R-\bar\delta})\leq \varepsilon\,.
\end{equation}
We then choose $\bar\delta\in(0,\delta)$ in such a way that  $\mathscr{P}_1^{\bar{\boldsymbol{\sigma}}}(\mathfrak{E},\partial B_{R-\bar\delta})=0$ (again, there are at most countably many $\bar\delta$'s such that this property fails).  
By Theorem \ref{GCthm}, \eqref{goodlay1} and the fact that $\mathfrak{F}$ coincides with $\mathfrak{E}$ outside $B_{R-\delta}$, we have 
\begin{align}\label{goodlay2}
\nonumber \varepsilon\geq \liminf_{k\to\infty}\, (1-2s_k)\mathscr{E}^{\boldsymbol{\sigma}}_{2s_k}(\mathfrak{E}_k,B_R\setminus \overline B_{R-\bar\delta})& \geq \omega_{n-1} \mathscr{P}_1^{\bar{\boldsymbol{\sigma}}}(\mathfrak{E},B_R\setminus \overline B_{R-\bar\delta})\\
\nonumber &=\omega_{n-1} \mathscr{P}_1^{\bar{\boldsymbol{\sigma}}}(\mathfrak{E},B_R\setminus  B_{R-\bar\delta})\\
&= \omega_{n-1} \mathscr{P}_1^{\bar{\boldsymbol{\sigma}}}(\mathfrak{F},B_R\setminus B_{R-\bar\delta})\,.
\end{align}
In view of \eqref{approxcompetit} and Theorem \ref{GCthm} again, by super additivity of $\mathscr{E}^{\boldsymbol{\sigma}}_{2s_k}(\mathfrak{F}_k,\cdot)$, 
\begin{align*}
\omega_{n-1} \mathscr{P}_1^{\bar{\boldsymbol{\sigma}}}(\mathfrak{F},B_R) & \geq 
\liminf_{k\to\infty} \,(1-2s_k)\mathscr{E}^{\boldsymbol{\sigma}}_{2s_k}(\mathfrak{F}_k,B_{R-\bar\delta})
+ \limsup_{k\to\infty} \,(1-2s_k)\mathscr{E}^{\boldsymbol{\sigma}}_{2s_k}(\mathfrak{F}_k, B_R\setminus B_{R-\bar\delta})\\
&\geq  \omega_{n-1} \mathscr{P}_1^{\bar{\boldsymbol{\sigma}}}(\mathfrak{F},B_{R-\bar\delta}) + \limsup_{k\to\infty} \,(1-2s_k)\mathscr{E}^{\boldsymbol{\sigma}}_{2s_k}(\mathfrak{F}_k, B_R\setminus B_{R-\bar\delta})\,.
\end{align*}
Combining \eqref{goodlay1} and \eqref{goodlay2}, we deduce that 
\begin{equation}\label{goodlay3}
\limsup_{k\to\infty} \,(1-2s_k)\mathscr{E}^{\boldsymbol{\sigma}}_{2s_k}(\mathfrak{F}_k, B_R\setminus B_{R-\bar\delta}) \leq \varepsilon\,.
\end{equation}
Since we are concerned with large $k$'s, by \eqref{goodlay1} and \eqref{goodlay3} we may assume that  
\begin{equation}\label{layerestim}
(1-2s_k)\mathscr{E}^{\boldsymbol{\sigma}}_{2s_k}(\mathfrak{E}_k,B_{R+\bar\delta}\setminus  B_{R-\bar\delta})\leq 2\varepsilon\;\text{ and }\; (1-2s_k)\mathscr{E}^{\boldsymbol{\sigma}}_{2s_k}(\mathfrak{F}_k, B_R\setminus B_{R-\bar\delta}) \leq 2  \varepsilon\,.
\end{equation}
\vskip5pt

\noindent{\it Step 3. (Gluing)} In the sequel, we  write $D_{r_1,r_2}:=B_{r_2}\setminus \overline B_{r_1}$  for $r_2> r_1>0$, and we set 
$$\delta_0:=\frac{3\bar\delta}{4}\,, \quad \delta_1:=\frac{\bar\delta}{2}\,,  \quad \delta_2:=\frac{\bar\delta}{4}\,,\quad\text{and}\quad \Omega_{\bar\delta}:=D_{R-\delta_0,R}\,.$$
We consider a smooth cut-off function $\varphi\in C^\infty(B_R)$ such that $0\leq \varphi\leq 1$, $\varphi=0$ in $D_{R-\delta_2,R}$, $\varphi=1$ in $B_{R-\delta_1}$, 
$|\nabla\varphi|\leq C/\bar\delta$, and $|\nabla^2\varphi|\leq C/\bar\delta^2$ for some universal constant $C$. Next we introduce the $\R^m$-valued functions 
$u_k:\Omega_{\bar\delta}\to\R^m$ and $v_k:\Omega_{\bar\delta}\to\R^m$ defined by 
$$u_k:=(\chi_{E_{1,k}},\ldots,\chi_{E_{m,k}}) \quad\text{and}\quad v_k:=(\chi_{F_{1,k}},\ldots,\chi_{F_{m,k}})\,,$$
and we set 
$$w_k:=(1-\varphi)u_k+\varphi v_k\,. $$
We still need to define auxiliary functions. For this, we write $({\rm e}_1,\ldots,{\rm e}_m)$ the canonical basis of $\R^m$ (i.e., ${\rm e}_i$ is the vector whose coodinates are zero except the $i$-th one which is one). For $i\in\{1,\ldots,m-1\}$, we define the function $d_{i,k}:\Omega_{\bar\delta}\to\R$ by 
$$d_{i,k}(x):={\rm dist}(w_k(x),{\rm e}_i)\,. $$
Notice that $0\leq d_{i,k}\leq \sqrt{2}$ since $w_k(x)$ is a convex combination of the ${\rm e}_j$'s. 

Following the notations in \cite{ADPM}, we set 
$$\mathcal{F}_{2s_k}(d_{i,k},\Omega_{\bar\delta}):=\iint_{\Omega_{\bar\delta}\times \Omega_{\bar\delta}}\frac{|d_{i,k}(x)-d_{i,k}(y)|}{|x-y|^{n+2s_k}}\,\de x\de y \,. $$
Since the distance function is $1$-Lipschitz, it is easy to see that $\mathcal{F}_{2s_k}(d_{i,k},\Omega_{\bar\delta})<\infty$ (see also the computations below). 
By the co-area type formula in \cite[Lemma 10]{ADPM} (more precisely, a tiny variant of it), we have  
\begin{multline*}
\mathcal{F}_{2s_k}(d_{i,k},\Omega_{\bar\delta})=2\int_0^{\sqrt{2}}\mathcal{I}_{2s_k}\big(\{d_{i,k}<t\}\cap \Omega_{\bar\delta}, \{d_{i,k}<t\}^c\cap \Omega_{\bar\delta}\big)\,\de t\\
 \geq 2\int_0^{1/2}\mathcal{I}_{2s_k}\big(\{d_{i,k}<t\}\cap \Omega_{\bar\delta}, \{d_{i,k}<t\}^c\cap \Omega_{\bar\delta}\big)\,\de t\,.
\end{multline*}
As a consequence, for each $i\in\{1,\ldots,m-1\}$, we can find some $t_i^k\in(0,1/2)$ such that for 
$$\widetilde G_{i,k}:= \{d_{i,k}<t_{i}^k\}\cap \Omega_{\bar\delta}\,,$$
we have 
\begin{equation}\label{energcuti}
\mathcal{I}_{2s_k}(\widetilde G_{i,k}\cap \Omega_{\bar\delta}, (\widetilde G_{i,k})^c\cap \Omega_{\bar\delta}) \leq \mathcal{F}_{2s_k}(d_{i,k},\Omega_{\bar\delta})\,.
\end{equation}
Since $w_k=u_k$ in $D_{R-\delta_2,R}$ and $w_k=v_k$ in $D_{R-\delta_1,R-\bar\delta}$, we have for $i\in\{1,\ldots,m-1\}$, 
$$\widetilde G_{i,k}\cap D_{R-\delta_2,R}= E_{i,k}\cap  D_{R-\delta_2,R}\quad\text{and}\quad \widetilde G_{i,k}\cap D_{R-\delta_1,R-\delta_0}= F_{i,k}\cap  D_{R-\delta_1,R-\delta_0}\,.$$
Now the elementary but important observation is that $B({\rm e}_i,1/2)\cap B({\rm e}_j,1/2)=\emptyset$ for $i\not=j$. It implies that 
$$\widetilde G_{i,k}\cap \widetilde G_{j,k}=\emptyset \quad\text{for every $i,j\in\{1,\ldots,m-1\}$ with $i\not=j$}\,.$$
We can now define
$$\widetilde G_{m,k}:= \Omega_{\bar\delta}\setminus\Big(\bigcup^{m-1}_{i=1} \widetilde G_{i,k}\Big)\,,$$
which satisfies $\widetilde G_{m,k}\cap D_{R-\delta_2,R}= E_{m,k}\cap  D_{R-\delta_2,R}$ and  $\widetilde G_{m,k}\cap D_{R-\delta_1,R-\delta_0}= F_{m,k}\cap  D_{R-\delta_1,R-\delta_0}$, and 
 $(\widetilde G_{1,k},\ldots,\widetilde G_{m,k})$ forms a partition of $ \Omega_{\bar\delta}$. In addition, we easily infer from \eqref{energcuti} that 
\begin{equation}\label{estiperlastcomp}
\mathcal{I}_{2s_k}(\widetilde G_{m,k}\cap \Omega_{\bar\delta}, (\widetilde G_{m,k})^c\cap \Omega_{\bar\delta}) \leq \sum_{i=1}^{m-1} \mathcal{F}_{2s_k}(d_{i,k},\Omega_{\bar\delta})\,.
\end{equation}
It allow us to define a partition $\mathfrak{G}_k=(G_{1,k},\ldots,G_{m,k})\in\mathscr{A}^{s_k}_m(B_R)$ by setting 
$$G_{i,k}\cap \Omega_{\bar\delta}:= \widetilde G_{i,k}\quad\text{and}\quad G_{i,k}\cap B_{R-\delta_0}:=F_{i,k}\cap B_{R-\delta_0}\quad\text{for $i=1,\ldots,m$}\,.$$
By construction, $\mathfrak{G}_k$ satisfies $G_{i,k}\triangle E_{i,k}\subset B_{R-\delta_2}\Subset B_R$ for every $i\in\{1,\ldots,m\}$, so that the local minimality of $\mathfrak{E}_k$ yields 
\begin{equation}\label{ineqminloc}
\mathscr{P}^{\boldsymbol{\sigma}}_{2s_k}(\mathfrak{E}_k,B_R)\leq \mathscr{P}^{\boldsymbol{\sigma}}_{2s_k}(\mathfrak{G}_k,B_R)\,.
\end{equation}
\vskip5pt

\noindent{\it Step 3. (Estimates)} We are now going to estime $ \mathscr{P}^{\boldsymbol{\sigma}}_{2s_k}(\mathfrak{G}_k,B_R)$ from above. First we write 
\begin{equation}\label{estilocmin1}
\mathscr{P}^{\boldsymbol{\sigma}}_{2s_k}(\mathfrak{G}_k,B_R)= \mathscr{E}^{\boldsymbol{\sigma}}_{2s_k}(\mathfrak{G}_k,B_R)+ \mathscr{F}^{\boldsymbol{\sigma}}_{2s_k}(\mathfrak{G}_k,B_R)\,.
\end{equation}
For the first term, we have
\begin{equation}\label{estilocmin2}
\mathscr{E}^{\boldsymbol{\sigma}}_{2s_k}(\mathfrak{G}_k,B_R)=\mathscr{E}^{\boldsymbol{\sigma}}_{2s_k}(\mathfrak{G}_k,\Omega_{\bar\delta}) +\mathscr{E}^{\boldsymbol{\sigma}}_{2s_k}(\mathfrak{F}_k,B_{R-\delta_0}) +\mathscr{R}_k\,,
\end{equation}
with 
$$ \mathscr{R}_k:=\frac{1}{2}\sum_{i=1}^m\bar\sigma_{ij}\big[\mathcal{I}_{2s_k}(G_{i,k}\cap\Omega_{\bar\delta},F_{j,k}\cap B_{R-\delta_0})+ \mathcal{I}_{2s_k}(F_{i,k}\cap B_{R-\delta_0},G_{j,k}\cap \Omega_{\bar\delta} )\big]\,.$$
Using that $G_{i,k}\cap A_{R-\delta_1,R-\delta_0}=F_{i,k}\cap A_{R-\delta_1,R-\delta_0}$, we obtain  from \eqref{layerestim}, 
\begin{equation}\label{estiRklocmin}
 \mathscr{R}_k\leq  \mathscr{E}^{\boldsymbol{\sigma}}_{2s_k}(\mathfrak{F}_k,A_{R-\delta_1,R-\bar\delta})+\frac{C\boldsymbol{\sigma}_{\rm max}}{(\bar\delta)^{n+2s_k}}\leq 
\frac{2\varepsilon}{(1-2s_k)} +\frac{C\boldsymbol{\sigma}_{\rm max}}{(\bar\delta)^{n+2s_k}}\,,
\end{equation}
for a constant $C$ depending only on $n$. 

On the other hand, recalling that $G_{i,k}\cap A_{R-\delta_2,R}=E_{i,k}\cap A_{R-\delta_2,R}$, 
\begin{multline*}
\mathscr{F}^{\boldsymbol{\sigma}}_{2s_k}(\mathfrak{G}_k,B_R)\leq \frac{1}{2}\sum_{i=1}^m\bar\sigma_{ij}\big[\mathcal{I}_{2s_k}(E_{i,k}\cap A_{R-\delta_2,R},E_{j,k}\cap A_{R,R+\bar\delta})\\
+ \mathcal{I}_{2s_k}(E_{i,k}\cap  A_{R,R+\bar\delta}) , A_{R-\delta_2,R})\big] +\frac{C\boldsymbol{\sigma}_{\rm max}}{(\bar\delta)^{n+2s_k}}\,,
\end{multline*}
so that \eqref{layerestim} yields again 
\begin{equation}\label{estilocmin3}
\mathscr{F}^{\boldsymbol{\sigma}}_{2s_k}(\mathfrak{G}_k,B_R)\leq  \mathscr{E}^{\boldsymbol{\sigma}}_{2s_k}(\mathfrak{E}_k,A_{R-\delta_2,R+\bar\delta})+\frac{C\boldsymbol{\sigma}_{\rm max}}{(\bar\delta)^{n+2s_k}}\leq \frac{2\varepsilon}{(1-2s_k)} +\frac{C\boldsymbol{\sigma}_{\rm max}}{(\bar\delta)^{n+2s_k}}\,.
\end{equation}
Combining \eqref{estilocmin1}, \eqref{estilocmin2}, \eqref{estiRklocmin} and \eqref{estilocmin3}, we are led to 
\begin{equation}\label{quasidoneprev}
\mathscr{P}^{\boldsymbol{\sigma}}_{2s_k}(\mathfrak{G}_k,B_R)\leq \mathscr{E}^{\boldsymbol{\sigma}}_{2s_k}(\mathfrak{G}_k,\Omega_{\bar\delta}) +\mathscr{E}^{\boldsymbol{\sigma}}_{2s_k}(\mathfrak{F}_k,B_{R-\delta_0})+ \frac{4\varepsilon}{(1-2s_k)} +\frac{C\boldsymbol{\sigma}_{\rm max}}{(\bar\delta)^{n+2s_k}}\,,
\end{equation}
and it only remains to estimate $ \mathscr{E}^{\boldsymbol{\sigma}}_{2s_k}(\mathfrak{G}_k,\Omega_{\bar\delta})$. 
\vskip5pt

We first infer from \eqref{energcuti} and \eqref{estiperlastcomp} that
 \begin{equation}\label{estilocmin4}
 \mathscr{E}^{\boldsymbol{\sigma}}_{2s_k}(\mathfrak{G}_k,\Omega_{\bar\delta})\leq \boldsymbol{\sigma}_{\rm max}\sum_{i=1}^m\mathcal{I}_{2s_k}(\widetilde G_{i,k}\cap \Omega_{\bar\delta}, (\widetilde G_{i,k})^c\cap \Omega_{\bar\delta})
\leq 2 \boldsymbol{\sigma}_{\rm max}\sum_{i=1}^{m-1} \mathcal{F}_{2s_k}(d_{i,k},\Omega_{\bar\delta})\,.
\end{equation}
Since the distance fonction is $1$-Lipschitz, we have for each $i\in\{1,\ldots,m-1\}$, 
\begin{equation}\label{deferrWk}
\mathcal{F}_{2s_k}(d_{i,k},\Omega_{\bar\delta})\leq  \iint_{\Omega_{\bar\delta}\times \Omega_{\bar\delta}}\frac{|w_{k}(x)-w_{k}(y)|}{|x-y|^{n+2s_k}}\,\de x\de y=:\mathscr{W}_k\,.
\end{equation}
To estimate $\mathscr{W}_k$, we write 
\begin{multline*}
w_{k}(x)-w_{k}(y)=(1-\varphi(x))(u_k(x)-u_k(y))+\varphi(x)(v_k(x)-v_k(y))\\ +(\varphi(x)-\varphi(y))(v(y)-u(y))\,,
\end{multline*}
so that 
$$|w_{k}(x)-w_{k}(y)|\leq |u_k(x)-u_k(y)|+|v_k(x)-v_k(y)|+\frac{C}{\bar\delta}|v(y)-u(y)|\,|x-y|\,,$$
and consequently,
\begin{multline}\label{estierrWk}
\mathscr{W}_k\leq  \iint_{\Omega_{\bar\delta}\times \Omega_{\bar\delta}}\frac{|u_{k}(x)-u_{k}(y)|}{|x-y|^{n+2s_k}}\,\de x\de y
+ \iint_{\Omega_{\bar\delta}\times \Omega_{\bar\delta}}\frac{|v_{k}(x)-v_{k}(y)|}{|x-y|^{n+2s_k}}\,\de x\de y\\
+\frac{C}{\bar\delta}\iint_{\Omega_{\bar\delta}\times \Omega_{\bar\delta}}\frac{|v_{k}(y)-u_{k}(y)|}{|x-y|^{n-1+2s_k}}\,\de x\de y\,.
\end{multline}
We estimate the last term above as follows, 
\begin{align}
\nonumber\frac{1}{\bar\delta}\iint_{\Omega_{\bar\delta}\times \Omega_{\bar\delta}}\frac{|v_{k}(y)-u_{k}(y)|}{|x-y|^{n-1+2s_k}}\,\de x\de y
\nonumber&\leq \frac{1}{\bar\delta}\int_{\Omega_{\bar\delta}}|v_{k}(y)-u_{k}(y)| \Big(\int_{B_{\bar \delta}(y)}\frac{\de x}{|x-y|^{n-1+2s_k}}\Big)\,\de y+\frac{C}{(\bar\delta)^{n+2s_k}}\\
\nonumber &\leq \frac{n\omega_n}{(1-2s_k)(\bar\delta)^{2s_k}}\|v_k-u_k\|_{L^1(\Omega_{\bar\delta})} +\frac{C}{(\bar\delta)^{n+2s_k}}\\
\label{estierrWk1} &\leq \frac{n\omega_n}{(1-2s_k)(\bar\delta)^{2s_k}}\sum_{j=1}^m\|\chi_{F_{j,k}}-\chi_{E_{j,k}}\|_{L^1(\Omega_{\bar\delta})} +\frac{C}{(\bar\delta)^{n+2s_k}}\,,
\end{align}
for a constant $C$ depending only on $n$. 

Recalling \eqref{layerestim}, we have
\begin{align}
\nonumber \iint_{\Omega_{\bar\delta}\times \Omega_{\bar\delta}}\frac{|u_{k}(x)-u_{k}(y)|}{|x-y|^{n+2s_k}}\,\de x\de y
\nonumber &\leq \sum_{j=1}^m\iint_{\Omega_{\bar\delta}\times \Omega_{\bar\delta}}\frac{|\chi_{E_{j,k}}(x)-\chi_{E_{j,k}}(y)|}{|x-y|^{n+2s_k}}\,\de x\de y\\
\nonumber  &\leq 2 \sum_{j=1}^m\mathcal{I}_{2s_k}(E_{j,k}\cap \Omega_{\bar\delta}, E^c_{j,k}\cap \Omega_{\bar\delta})\\
\label{estierrWk2} &\leq \frac{2}{\boldsymbol{\sigma}_{\rm min}} \mathscr{E}^{\bar{\boldsymbol{\sigma}}}_{2s_k}(\mathfrak{E}_k,\Omega_{\bar\delta})\leq  \frac{4\varepsilon}{\boldsymbol{\sigma}_{\rm min}(1-2s_k)}\,,
\end{align}
and similarly 
\begin{equation}\label{estierrWk3}
 \iint_{\Omega_{\bar\delta}\times \Omega_{\bar\delta}}\frac{|v_{k}(x)-v_{k}(y)|}{|x-y|^{n+2s_k}}\,\de x\de y 
\leq \frac{2}{\boldsymbol{\sigma}_{\rm min}} \mathscr{E}^{\bar{\boldsymbol{\sigma}}}_{2s_k}(\mathfrak{F}_k,\Omega_{\bar\delta})\leq  \frac{4\varepsilon}{\boldsymbol{\sigma}_{\rm min}(1-2s_k)}\,.
\end{equation}
Gathering \eqref{estilocmin4}, \eqref{deferrWk}, \eqref{estierrWk},  \eqref{estierrWk1},  \eqref{estierrWk2} and  \eqref{estierrWk3}, we conclude that 
\begin{equation}\label{quasidone}
 \mathscr{E}^{\boldsymbol{\sigma}}_{2s_k}(\mathfrak{G}_k,\Omega_{\bar\delta})
\leq \frac{Cm\boldsymbol{\sigma}_{\rm max}}{(1-2s_k)(\bar\delta)^{2s_k}}\sum_{j=1}^m\|\chi_{F_{j,k}}-\chi_{E_{j,k}}\|_{L^1(\Omega_{\bar\delta})}+\frac{16m\boldsymbol{\sigma}_{\rm max}\varepsilon}{\boldsymbol{\sigma}_{\rm min}(1-2s_k)}+\frac{Cm\boldsymbol{\sigma}_{\rm max}}{(\bar\delta)^{n+2s_k}}\,,
\end{equation}
for a constant $C$ depending only on $n$. 
\vskip5pt

Finally, combining \eqref{quasidoneprev} with \eqref{quasidone} and observing that $\mathscr{E}^{\boldsymbol{\sigma}}_{2s_k}(\mathfrak{F}_k,B_{R-\delta_0})\leq \mathscr{P}^{\boldsymbol{\sigma}}_{2s_k}(\mathfrak{F}_k,B_R)$, we have reached our last estimate
\begin{multline}\label{optupbdcompet}
(1-2s_k)\mathscr{P}^{\boldsymbol{\sigma}}_{2s_k}(\mathfrak{G}_k,B_R)\leq (1-2s_k)\mathscr{P}^{\boldsymbol{\sigma}}_{2s_k}(\mathfrak{F}_k,B_{R})\\ 
+ \frac{Cm\boldsymbol{\sigma}_{\rm max}}{(\bar\delta)^{2s_k}}\sum_{j=1}^m\|\chi_{F_{j,k}}-\chi_{E_{j,k}}\|_{L^1(\Omega_{\bar\delta})}+\frac{20m\boldsymbol{\sigma}_{\rm max}\varepsilon}{\boldsymbol{\sigma}_{\rm min}}+\frac{Cm\boldsymbol{\sigma}_{\rm max}(1-2s_k)}{(\bar\delta)^{n+2s_k}}\,,
\end{multline} 
 for a constant $C$ depending only on $n$. 
 \vskip5pt
 
 \noindent{\it Step 4. (Conclusion)} Now we recall that $\chi_{F_{j,k}}\to \chi_{F_{j}}$ in $L^1(\Omega)$ for every $j\in\{1,\ldots,m\}$, and by our choice of $\bar\delta$, we have 
 $F_{j}\cap\Omega_{\bar\delta}=E_{j}\cap\Omega_{\bar\delta}$. On the other hand, $\chi_{E_{j,k}}\to \chi_{E_j}$ in $L^1(\Omega)$ for every $j\in\{1,\ldots,m\}$, and therefore 
 \begin{equation}\label{convl1competit}
 \sum_{j=1}^m\|\chi_{F_{j,k}}-\chi_{E_{j,k}}\|_{L^1(\Omega_{\bar\delta})}\mathop{\longrightarrow}\limits_{k\to\infty} 0\,.
 \end{equation}
 Using Theorem \ref{GCthm}, \eqref{approxcompetit}, \eqref{ineqminloc}, \eqref{optupbdcompet}, and \eqref{convl1competit}, we deduce that 
 \begin{multline*}
 \omega_{n-1}\mathscr{P}_1^{\bar{\boldsymbol{\sigma}}}(\mathfrak{E},B_R)\leq \liminf_{k\to\infty}\, (1-2s_k)\mathscr{P}^{\boldsymbol{\sigma}}_{2s_k}(\mathfrak{E}_k,B_R)\leq
  \limsup_{k\to\infty}\, (1-2s_k)\mathscr{P}^{\boldsymbol{\sigma}}_{2s_k}(\mathfrak{E}_k,B_R) \\
 \leq \limsup_{k\to\infty}\, (1-2s_k)\mathscr{P}^{\boldsymbol{\sigma}}_{2s_k}(\mathfrak{G}_k,B_R)\leq \omega_{n-1}\mathscr{P}_1^{\bar{\boldsymbol{\sigma}}}(\mathfrak{F},B_R)
 +\frac{20m\boldsymbol{\sigma}_{\rm max}\varepsilon}{\boldsymbol{\sigma}_{\rm min}}\,.
 \end{multline*}
 In view of the arbitrariness of $\varepsilon$, we conclude that $\mathscr{P}_1^{\bar{\boldsymbol{\sigma}}}(\mathfrak{E},B_R)\leq \mathscr{P}_1^{\bar{\boldsymbol{\sigma}}}(\mathfrak{F},B_R)$,   
 and thus $\mathfrak{E}$ is a local minimizer of  $\mathscr{P}_1^{\bar{\boldsymbol{\sigma}}}(\cdot,B_R)$. In addition, choosing $\mathfrak{F}=\mathfrak{E}$ in the chain of inequalities above 
 and letting $\varepsilon\to 0$ shows that 
 $$ \lim_{k\to\infty}\, (1-2s_k)\mathscr{P}^{\boldsymbol{\sigma}}_{2s_k}(\mathfrak{E}_k,B_R) = \omega_{n-1}\mathscr{P}_1^{\bar{\boldsymbol{\sigma}}}(\mathfrak{E},B_R)\,.$$
In particular, $\boldsymbol{\alpha}(B_R)= \omega_{n-1}\mathscr{P}_1^{\bar{\boldsymbol{\sigma}}}(\mathfrak{E},B_R)$. 
\vskip3pt

We may now complete the proof exactly as in \cite[Proof of Theorem 3]{ADPM} noticing that the argument  can be reproduced for any open set $\Omega^\prime\Subset\Omega$ 
with Lipschitz boundary and satisfying $\boldsymbol{\alpha}(\partial\Omega^\prime)=0$ (instead of a ball $B_R$), and concluding by regularity (of the monotone set functions involved) that  $\boldsymbol{\alpha}=\omega_{n-1}\mathscr{P}_1^{\bar{\boldsymbol{\sigma}}}(\mathfrak{E},\cdot)$. 
 \end{proof}

\vskip20pt


\end{document}